\documentclass[11pt,twoside]{amsart}

\usepackage[a4paper,inner=3.0cm,outer=2.5cm,top=2.8cm,bottom=2.8cm]{geometry}
\usepackage[T1]{fontenc}
\usepackage[english]{babel}
\usepackage{lmodern}
\usepackage{amssymb,amsthm,amsmath,mathtools}
\usepackage{mathrsfs}
\usepackage{tikz-cd}
\usepackage{enumitem}
\usepackage{xcolor}
\usepackage[pagebackref]{hyperref}
\usepackage[nameinlink,noabbrev]{cleveref}
\usepackage{aliascnt}
\usepackage[protrusion=true,expansion=false]{microtype}

\hypersetup{
  colorlinks=true,
  linkcolor=blue!45!black,
  citecolor=blue!45!black,
  urlcolor=blue!55!black,
  pdftitle={Finite-coefficient K-theory of henselian valued fields and Gersten injectivity},
  pdfauthor={Niels Feld}
}

\renewcommand*{\backrefalt}[4]{%
  \ifcase #1\relax
  \or {\footnotesize\textup{[p.~#2]}}%
  \else {\footnotesize\textup{[pp.~#2]}}%
  \fi
}

\newcommand{\FilSp}{\operatorname{FilSp}}
\newcommand{\Fil}{\operatorname{Fil}}
\DeclareMathOperator{\Frac}{Frac}
\DeclareMathOperator{\Perf}{Perf}
\DeclareMathOperator{\Spec}{Spec}
\newcommand{\et}{\mathrm{\acute{e}t}}
\newcommand{\mot}{\operatorname{mot}}
\newcommand{\sep}{\operatorname{sep}}
\newcommand{\syn}{\operatorname{syn}}
\newcommand{\stackstag}[1]{%
  \href{https://stacks.math.columbia.edu/tag/#1}{\textup{Tag~#1}}}
\newcommand{\arxivhtmlref}[3]{%
  \href{https://arxiv.org/html/#1\##2}{#3}}
\newcommand{\arxivabsref}[2][]{%
  \href{https://arxiv.org/abs/#2}{arXiv:#2#1}}

\newtheorem{theorem}{Theorem}[section]
\crefname{theorem}{theorem}{theorems}
\Crefname{theorem}{Theorem}{Theorems}
\newaliascnt{proposition}{theorem}
\newtheorem{proposition}[proposition]{Proposition}
\aliascntresetthe{proposition}
\crefname{proposition}{proposition}{propositions}
\Crefname{proposition}{Proposition}{Propositions}
\newaliascnt{lemma}{theorem}
\newtheorem{lemma}[lemma]{Lemma}
\aliascntresetthe{lemma}
\crefname{lemma}{lemma}{lemmas}
\Crefname{lemma}{Lemma}{Lemmas}
\newaliascnt{corollary}{theorem}
\newtheorem{corollary}[corollary]{Corollary}
\aliascntresetthe{corollary}
\crefname{corollary}{corollary}{corollaries}
\Crefname{corollary}{Corollary}{Corollaries}
\theoremstyle{definition}
\newaliascnt{definition}{theorem}
\newtheorem{definition}[definition]{Definition}
\aliascntresetthe{definition}
\crefname{definition}{definition}{definitions}
\Crefname{definition}{Definition}{Definitions}
\newaliascnt{construction}{theorem}
\newtheorem{construction}[construction]{Construction}
\aliascntresetthe{construction}
\crefname{construction}{construction}{constructions}
\Crefname{construction}{Construction}{Constructions}
\newaliascnt{convention}{theorem}
\newtheorem{convention}[convention]{Convention}
\aliascntresetthe{convention}
\crefname{convention}{convention}{conventions}
\Crefname{convention}{Convention}{Conventions}
\theoremstyle{remark}
\newaliascnt{remark}{theorem}
\newtheorem{remark}[remark]{Remark}
\aliascntresetthe{remark}
\crefname{remark}{remark}{remarks}
\Crefname{remark}{Remark}{Remarks}

\providecommand{\llbracket}{[\![}
\providecommand{\rrbracket}{]\!]}
\numberwithin{equation}{section}
\title[Finite-coefficient K-theory of henselian valued fields and Gersten injectivity]
{Finite-coefficient K-theory of henselian valued fields and Gersten injectivity}
\author{Niels Feld}
\address{IRMAR, University of Rennes, 263 Avenue G\'en\'eral Leclerc,
35000 Rennes, France}
\email{\href{mailto:niels.feld@univ-rennes.fr}{niels.feld@univ-rennes.fr}}
\date{15 August 2026}

\begin{document}

\begin{abstract}
Let \(W\) be a henselian valuation ring with fraction field \(L\), residue
field \(k\), and value group \(\Gamma_W\). Let \(N=\ell^\nu\) be
invertible in \(W\). Choose an ordered \(\mathbf Z/N\)-basis \(B\) of
\(\Gamma_W/N\Gamma_W\). Products of suitable filtration-one classes
define an equivalence of complete filtered spectra
\[
\bigoplus_{\substack{J\subseteq B\\J\text{ finite}}}
\Sigma^{|J|}\Fil_{\mot}^{\bullet-|J|}K(k;\mathbf Z/N)
\xrightarrow{\ \simeq\ }
\Fil_{\mot}^{\bullet}K(L;\mathbf Z/N).
\]
The summand indexed by \(\varnothing\) is the generic restriction map,
after the rigidity equivalence
\(K(W;\mathbf Z/N)\simeq K(k;\mathbf Z/N)\), and is therefore split
injective. Independently of this splitting, excision yields a
lifting theorem for regular henselian pairs. In particular, this yields finite-coefficient Gersten injectivity for noetherian henselian regular
local rings. Applications to completions along regular primes
give relative and sometimes nonhenselian examples. More generally, if \(P\) is a Prüfer
ring and \(R\) is a henselian local ind-smooth \(P\)-algebra, then \(R\)
is a domain and
\(K_n(R;\mathbf Z/N)\to K_n(\Frac(R);\mathbf Z/N)\) is injective for
every \(n\), provided \(N\in R^\times\). These injectivity consequences
extend from prime-power to arbitrary finite invertible coefficients by
primary decomposition. 

\end{abstract}

\maketitle

\tableofcontents

\section{Introduction}\label{sec:introduction}

Let \(W\) be a henselian valuation ring. Write \(L\) for its fraction
field, \(k\) for its residue field, and \(\Gamma_W\) for its value group.
Fix a prime power \(N=\ell^\nu\) which is invertible in \(W\). Gabber's
rigidity theorem identifies \(K(W;\mathbf Z/N)\) with
\(K(k;\mathbf Z/N)\); see
\cite[\arxivhtmlref{1803.10897v2}{S1.Thmdefinition4}{Thm.~1.4}]{ClausenMathewMorrow21HenselianPairs} and
\cite{Gabber92Rigidity}. The generic restriction map
\[
K(W;\mathbf Z/N)\longrightarrow K(L;\mathbf Z/N)
\]
is not part of that theorem. We describe it by combining tame Galois
cohomology with the motivic filtration of Bouis.

\subsection*{Main results}

Bouis associates with every qcqs scheme \(T\) a decreasing multiplicative
filtration \(\Fil_{\mot}^\bullet K(T)\). Its graded pieces are
\(\mathbf Z(q)_{\mot}(T)[2q]\)
\cite[Defs.~3.18 and~3.20, Thm.~C\textup{(1)}]{BouisMixedCharacteristicMotivicCohomology}.
For finite coefficients on a field, the comparison theorem of
Bouis--Kundu identifies these graded pieces with truncated étale Tate
twists
\cite[\arxivhtmlref{2506.09910v2}{S5.Thmtheorem1}{Thm.~5.1}]{BouisKunduBL}.
We recall the precise conventions in
\Cref{sec:finite-coefficient-motivic-filtrations}.
Filtered closed-point rigidity and generic restriction define the map
\[
s_{W,N}\colon \Fil_{\mot}^\bullet K(k;\mathbf Z/N)
\longrightarrow \Fil_{\mot}^\bullet K(L;\mathbf Z/N)
\]
used below; see \Cref{lem:rigidity-specialization-filtered-inflation}.

\begin{theorem}[value-exterior splitting]
\label{thm:intro-value-exterior-splitting}
Choose a \(\mathbf Z/N\)-basis \(B\) of
\(\Gamma_W/N\Gamma_W\) and a total order on \(B\).

For each \(b\in B\), choose \(\varpi_b\in L^\times\) whose value has
class \(b\) in \(\Gamma_W/N\Gamma_W\). There is a map
\[
\Phi_B\colon
\bigoplus_{\substack{J\subseteq B\\J\text{ finite}}}
\Sigma^{|J|}\Fil_{\mot}^{\bullet-|J|}K(k;\mathbf Z/N)
\longrightarrow
\Fil_{\mot}^{\bullet}K(L;\mathbf Z/N)
\]
such that, on the summand indexed by
\(J=\{b_1<\cdots<b_r\}\), its homotopy class is specialization followed
by right multiplication, in this order, by
\(
\widetilde\kappa_L(\varpi_{b_1}),\ldots,
\widetilde\kappa_L(\varpi_{b_r}).
\)
The map \(\Phi_B\) is an equivalence of complete filtered spectra.
Consequently,
\[
K(L;\mathbf Z/N)\simeq
\bigoplus_{\substack{J\subseteq B\\J\text{ finite}}}
\Sigma^{|J|}K(k;\mathbf Z/N).
\]
Under this equivalence, the filtered specialization \(s_{W,N}\) is the
inclusion of the summand indexed by \(\varnothing\).
\end{theorem}

The term \emph{value-exterior} refers to the indexing by finite subsets
of \(B\) and to the exterior algebra in
\Cref{lem:henselian-tame-inertia-cohomology}\textup{(iii)}.

The Kummer-normalized classes
\(\widetilde\kappa_L(\varpi_b)\) are defined in
\Cref{def:kummer-normalized-weight-one-classes}, and the map
\(\Phi_B\) is constructed in \Cref{con:value-exterior-map}. It depends
on the ordered basis \(B\) and the lifts \(\varpi_b\). Constructing an
actual map also requires representatives of the filtration-one classes,
but its homotopy class is independent of those representatives. The
filtered and unfiltered decompositions are therefore noncanonical, but
the empty summand is canonical. The theorem is
\Cref{thm:value-exterior-splitting}. No finiteness assumption is made on
the rank of \(W\) or on the cardinality of \(B\).

\begin{corollary}[Gersten injectivity for henselian valuation rings]
\label{cor:intro-henselian-valuation-generic-injectivity}
For every \(n\in\mathbf Z\), the map
\[
K_n(W;\mathbf Z/N)\longrightarrow K_n(L;\mathbf Z/N)
\]
is split injective.
\end{corollary}

This is \Cref{thm:henselian-valuation-generic-injectivity}.

\begin{remark}[arbitrary finite coefficients]
\label{rem:primary-decomposition-generic-injectivity}
The splitting theorem is stated for a prime power in order to use a
\(\mathbf Z/\ell^\nu\)-basis of the value quotient. Its injectivity
consequences, including the Prüfer-base and regular applications below,
extend to every finite \(N\) invertible in the relevant ring. Write
\(N=\prod_{\ell\mid N}\ell^{\nu_\ell}\). 
We have the natural primary decomposition
\[
K(-;\mathbf Z/N)\simeq
\bigoplus_{\ell\mid N}K(-;\mathbf Z/\ell^{\nu_\ell}).
\]
Apply the prime-power cases to the summands; the kernel of a generic
restriction map decomposes accordingly.
\end{remark}

\begin{theorem}[Prüfer-base Gersten injectivity]
\label{thm:intro-prufer-ind-smooth-generic-injectivity}
Let \(P\) be a Prüfer ring, let \(R\) be a henselian local ind-smooth
\(P\)-algebra, and let \(N=\ell^\nu\) be invertible in \(R\). Then
\(R\) is a domain, and
\[
K_n(R;\mathbf Z/N)\longrightarrow K_n(\Frac(R);\mathbf Z/N)
\]
is injective for every \(n\in\mathbf Z\).
\end{theorem}

This is \Cref{thm:prufer-ind-smooth-generic-injectivity}. In particular,
the domain conclusion is not an additional hypothesis.

\begin{theorem}[regular henselian-pair lifting]
\label{thm:intro-regular-henselian-pair-lifting}
Let \(R\) be a noetherian regular domain of finite Krull dimension. Let
\(\mathfrak p\subset R\) be prime, and let \(N=\ell^\nu\) be invertible
in \(R\). Assume that \(R/\mathfrak p\) is regular and that
\((R,\mathfrak p)\) is a henselian pair
\cite[\stackstag{09XE}]{StacksProject}. Put
\(F=\Frac(R)\) and \(E=\Frac(R/\mathfrak p)\). If
\[
K_n(R/\mathfrak p;\mathbf Z/N)\longrightarrow
K_n(E;\mathbf Z/N)
\]
is injective for every \(n\in\mathbf Z\), then
\[
K_n(R;\mathbf Z/N)\longrightarrow K_n(F;\mathbf Z/N)
\]
is injective for every \(n\in\mathbf Z\).
\end{theorem}

This is \Cref{thm:regular-henselian-pair-lifting}. Its maximal-ideal
case is useful to record; the proof below also gives a short induction
from Gabber rigidity and Gillet's arbitrary-DVR theorem.

\begin{corollary}[henselian regular local rings]
\label{cor:intro-henselian-regular-local-generic-injectivity}
Let \(A\) be a noetherian henselian regular local ring with fraction
field \(F\). If \(N=\ell^\nu\) is invertible in \(A\), then
\[
K_n(A;\mathbf Z/N)\longrightarrow K_n(F;\mathbf Z/N)
\]
is injective for every \(n\in\mathbf Z\).
\end{corollary}

See \Cref{cor:henselian-regular-local-generic-injectivity}. Its proof records
the short induction from Gabber rigidity and Gillet's theorem for an
arbitrary discrete valuation ring
\cite{Gillet86TorsionGerstenDVR}. The étale-cohomological analogue for an
arbitrary henselian regular local ring is
\cite[Prop.~5]{Sakagaito20}; in dimension two, Sakagaito proves exactness
of the full étale Gersten complex \cite[Thm.~9]{Sakagaito20}.
The relative consequences of
\Cref{thm:intro-regular-henselian-pair-lifting} include the completed and
nonhenselian examples in
\Cref{cor:regular-prime-adic-completion,cor:excellent-nonhenselian-mixed-characteristic}.


For regular local rings containing a field, Gersten injectivity with
finite coefficients is known. Quillen proves the full integral Gersten
conjecture in the essentially finite type case
\cite[\S7, Thm.~5.11]{Quillen73}; Grayson's universal exactness shows that
the augmentation remains injective after applying both
\(M\mapsto M/N\) and \(M\mapsto M[N]\)
\cite[Cor.~6]{Grayson85UniversalExactness}. The coefficient exact sequence
then gives the finite-coefficient assertion at such a stage. For an
arbitrary regular local ring containing a field, Popescu approximation
\cite[\stackstag{07GC}]{StacksProject} and continuity of nonconnective
\(K\)-theory \cite[Def.~8.1 and proof of Thm.~9.8]
{BGT13KTheoryStableCategories}
allow a kernel class and a denominator witnessing its generic vanishing
to descend to a common smooth stage.

For local rings essentially smooth over a discrete valuation ring,
the finite-coefficient assertion is already recorded in
\cite[Cor.~6 and Rem.~\textup{(ii)} following it, p.~68]
{GilletLevine87RelativeGersten}, for coefficients invertible
on the base; see also \cite{Gillet86TorsionGerstenDVR}.
Popescu approximation \cite[\stackstag{07GC}]{StacksProject}
and continuity extend generic injectivity to noetherian local
rings receiving a regular local homomorphism from such a
discrete valuation ring.
These results do not settle the general nonhenselian
mixed-characteristic case. Exactness beyond the augmentation
is a separate and substantially stronger problem.

If \(A\) is a nonhenselian
mixed-characteristic regular local ring and \(A^h\) is its henselization,
applying \Cref{cor:intro-henselian-regular-local-generic-injectivity} to
\(A^h\) gives Gersten injectivity only after henselization. To descend
back to \(A\), one would need, in particular, to control
\[
K_n(A;\mathbf Z/N)\longrightarrow K_n(A^h;\mathbf Z/N),
\]
and this is not supplied by rigidity.

\subsection*{Relation with previous work}

The principal contribution of this paper is the equivalence of complete
filtered spectra in \Cref{thm:intro-value-exterior-splitting}, for the
specified map of \Cref{con:value-exterior-map}, together with the
identification of generic restriction. We retain detailed proofs of
standard preparatory results to fix the maps, signs, filtration
conventions, and hypotheses used in this construction. Their inclusion
is not intended as a claim of novelty.

The associated graded of \Cref{thm:intro-value-exterior-splitting} is the
tame value-degree decomposition of Galois cohomology. Under the additional
assumption \(\mu_N\subset L\), Wadsworth proves, for an arbitrary basis of
\(\Gamma_W/N\Gamma_W\) and without a finiteness hypothesis on the value
group, a graded-ring isomorphism rather than merely an additive
decomposition
\cite[Hypotheses~3.5 and Thm.~3.6]{Wadsworth83pHenselian}. He also proves
the parallel value-degree decomposition for Milnor \(K\)-theory modulo
\(N\), with the henselian case requiring no roots-of-unity hypothesis
\cite[Prop.~2.3 and Cor.~2.4]{Wadsworth83pHenselian}.

In the roots-of-unity setting, Wadsworth further records, crediting Jacob,
a Hochschild--Serre proof
\cite[Rem.~3.15]{Wadsworth83pHenselian}. Without roots of unity,
\cite[Rem.~3.16]{Wadsworth83pHenselian} explicitly notes that the same
spectral-sequence method still yields some direct decomposition. What is
left unspecified there is its explicit Tate-twisted form and the
identification of its summands.

The value-exterior splitting \Cref{thm:intro-value-exterior-splitting}
upgrades these antecedents to finite-coefficient nonconnective algebraic
\(K\)-theory, at the level of complete filtered spectra and with the
specialization summand identified. For comparison, Bouis proves that,
for every henselian valuation ring and every prime-power coefficient
invertible in it, the finite-coefficient motivic complex is naturally the
corresponding truncated étale complex
\cite[\arxivhtmlref{2608.05220v1}{S5.Thmtheorem5}{Prop.~5.5}]{BouisSingularRings}.
Applied to \(W\), \(k\), and \(L\), this supplies the relevant objectwise
comparisons. It does not identify the generic restriction \(W\to L\), its
Hochschild--Serre filtration, or the value-exterior splitting; these are
the additional assertions proved here.
\Cref{thm:tame-cohomology-package-level-nu} gives the explicit twisted
summands without a roots-of-unity assumption and identifies the
basis-dependent direct sum with the canonical Hochschild--Serre
filtration.
The filtration in
\Cref{thm:tame-cohomology-package-level-nu} is multiplicative by
\Cref{lem:multiplicative-hochschild-serre}; its explicit
direct-sum decomposition is additive.

Three steps in that lift require care. First, filtered closed-point
rigidity transports completeness from the residue field to valuation rings
of arbitrary rank; see \Cref{rem:completeness-from-rigidity}. Second,
multiplication by a filtration-one class is compared with Kummer cup
product; see \Cref{lem:weight-one-operators-are-cup-products}. A Laurent-series test shows that the comparison scalar is a
unit; see \Cref{lem:laurent-series-weight-one-test}. Third, the source of \(\Phi_B\) is an infinite coproduct when \(B\)
is infinite. Its completeness is proved directly in
\Cref{lem:value-exterior-source-complete}. The principal result is the
filtered equivalence, including its arbitrary-rank form and the
identification of the generic restriction map.

The regular henselian-pair results of \Cref{sec:henselian-pair-injectivity} require less. The parameter-chain
valuation used there has lexicographic value group \(\mathbf Z^d\).
Prime-step excision \Cref{thm:valuation-prime-step-k-excision} (see also \cite[Thm.~21]{Dahlhausen25SemiValuation}) reduces its Gersten injectivity to the
henselian discrete-valuation case; see
\Cref{thm:lexicographic-finite-rank-generic-injectivity}. Thus those results
are logically independent of the filtered splitting.

The Prüfer-base \Cref{thm:prufer-ind-smooth-generic-injectivity} is an application of Gersten
injectivity for valuation rings and hence, through
\Cref{thm:henselian-valuation-generic-injectivity}, of the filtered
splitting. Its passage beyond valuation rings uses a different reduction.

Bouis--Kundu prove that, for a finitary deflatable Nisnevich sheaf of
spectra \(\mathscr F\), vanishing of a fixed homotopy group
\(\pi_j\mathscr F\) on henselian valuation rings that are ind-smooth
\(P\)-algebras propagates to every henselian local ind-smooth
\(P\)-algebra
\cite[\arxivhtmlref{2506.09910v2}{S3.Thmtheorem13}{Cor.~3.13}]{BouisKunduBL}.
The direct argument of \Cref{sec:prufer-ind-smooth-generic-injectivity}
avoids having to compare raw finite-coefficient \(K\)-groups with their
Nisnevich sheafification.

For valuation rings containing a field, Kelly--Morrow proved integral
Gersten injectivity
\cite[\arxivhtmlref{1810.12203v1}{S1.Thmtheorem2}{Thm.~1.2}]{KellyMorrowValuationRings} (their proof should also give injectivity with finite coefficients invertible in the
valuation ring).

Kundu proved the corresponding statement for semilocalizations of smooth
integral algebras over an equicharacteristic valuation ring
\cite[Thm.~1.2]{Kundu26SmoothValuation}. 

\Cref{thm:henselian-valuation-generic-injectivity}
also treats mixed characteristic, with arbitrary ramification
and arbitrary rank, and gives a retraction of spectra.

Integral Gersten injectivity does not by itself formally imply its
finite-coefficient analogue. Indeed, the coefficient sequence
\[
0\longrightarrow K_n(D)/N\longrightarrow K_n(D;\mathbf Z/N)
\longrightarrow K_{n-1}(D)[N]\longrightarrow0
\]
shows that integral injectivity controls the right-hand torsion term; the
possible obstruction is injectivity of
\(K_n(D)/N\to K_n(\Frac(D))/N\). The
invertibility assumption in \Cref{cor:intro-henselian-regular-local-generic-injectivity} cannot be
omitted in general:
according to
\cite[\arxivhtmlref{2608.05005v1}{S1.Thmtheorem1}{Thm.~1.1}]{Feld26Counterexample},
Gersten injectivity fails with
residue-characteristic coefficients for a complete, hence henselian, ramified mixed-characteristic
regular local ring. Those coefficients lie outside the present hypotheses.

\subsection*{Proof strategy}

The comparison of Bouis--Kundu gives
\[
\operatorname{gr}_{\mot}^qK(F;\mathbf Z/N)
\simeq
\tau^{\le q}R\Gamma(F_{\et},\mu_N^{\otimes q})[2q]
\]
for a field \(F\); see
\Cref{lem:ordinary-prufer-bl-range-input,thm:field-finite-coefficient-k-filtration-input}.
For the henselian valued field \(L\), the Hochschild--Serre spectral
sequence for inertia is multiplicative (\Cref{lem:multiplicative-hochschild-serre}). Its \(E_2\)-page is the tensor
product of the cohomology of \(k\) with the exterior algebra on
\(\Gamma_W/N\Gamma_W\). The value-one Kummer classes are permanent
cycles and generate the second factor; see
\Cref{thm:tame-cohomology-package-level-nu}. The spectral sequence therefore
degenerates. This proves \Cref{thm:tame-cohomology-package-level-nu}.

The motivic filtration is multiplicative
(\Cref{lem:global-motivic-filtration-kmodell-input}). A unit of \(L\)
defines a class in filtration one
(\Cref{lem:weight-one-edge-kummer-compatibility,def:kummer-normalized-weight-one-classes}).
Étale sheafification compares multiplication by this class with cup product
by its Kummer class, up to a scalar in \(\mathbf Z/N\)
(\Cref{lem:etale-sheafification-of-motivic-complexes}). The scalar is
preserved by field extension, and hence is constant within each
characteristic
(\Cref{lem:etale-sheafification-of-motivic-complexes}\textup{(iii)}).
We compute it over \(k((t))\), with \(k\) separably closed
(\Cref{lem:laurent-series-weight-one-test}). Localization and rigidity
show that multiplication by \([t]\) is an isomorphism in the relevant
\(K\)-groups
(\Cref{lem:laurent-series-weight-one-test}\textup{(iii)}). Hence the
scalar is a unit
(\Cref{lem:weight-one-scalar-is-a-unit}). The associated graded of
\(\Phi_B\) is then the tame cohomological isomorphism, up to units on
the summands
(\Cref{lem:value-exterior-associated-graded}).

\Cref{lem:filtered-whitehead-criterion} applies because the
source is complete by \Cref{lem:value-exterior-source-complete}
and the field target is complete by
\Cref{lem:global-motivic-filtration-kmodell-input}.
This proves \Cref{thm:value-exterior-splitting}.
Completeness for \(W\) is recorded separately in
\Cref{rem:completeness-from-rigidity}.

For \Cref{thm:intro-regular-henselian-pair-lifting}, one chooses a
valuation ring dominating \(R_{\mathfrak p}\), with residue field \(E\)
and lexicographic value group \(\mathbf Z^d\). Its henselization has the
same residue field and value group. Prime-step excision reduces Gersten
injectivity for this valuation ring to the henselian discrete-valuation
case (\Cref{thm:valuation-prime-step-k-excision}). The result then follows from the two rigidity equivalences.

More generally,
\Cref{thm:finite-rank-prime-to-residue-characteristic-generic-injectivity}
records that finite-rank Gersten injectivity follows from the rank-one
case by prime-step excision and closed-point rigidity, independently of
the filtered splitting.

For \Cref{thm:intro-prufer-ind-smooth-generic-injectivity}, a class and a
nonzero denominator are descended to one smooth \(P\)-algebra. After
localizing at the point selected by the henselian target and then
henselizing, a composite valuation has the same residue field. Both
local rings have the finite-coefficient \(K\)-theory of that field, so
Gersten injectivity for the henselian valuation ring detects the class.

\subsection*{Organization}

\Cref{sec:conventions} fixes the coefficient, filtration, and valuation
conventions. \Cref{sec:finite-coefficient-motivic-filtrations} records
the motivic filtration and rigidity inputs. The tame cohomological
decomposition is proved in \Cref{sec:tame-inertia}.
\Cref{sec:filtered-specialization} lifts it to filtered \(K\)-theory.
The Gersten injectivity results and the independent prime-step induction
are proved in \Cref{sec:henselian-pair-injectivity}. The Prüfer-base
extension is proved in
\Cref{sec:prufer-ind-smooth-generic-injectivity}.

\section{Conventions}\label{sec:conventions}

Whenever the notation \(N=\ell^\nu\) is used, \(\ell\) is a prime and
\(\nu\geq1\) is a natural number, unless a statement explicitly says otherwise.

\begin{convention}[nonconnective \(K\)-theory and finite coefficients]
\label{conv:nonconnective-k-groups}
We work in the stable \(\infty\)-category \(\mathrm{Sp}\) of spectra.
All cofibers, limits, colimits, and equivalences of spectra below are
formed in \(\mathrm{Sp}\).
For a quasi-compact quasi-separated scheme \(Y\), write
\[
K(Y):=K(\Perf(Y)),\qquad K_n(Y):=\pi_nK(Y)\quad(n\in\mathbf Z)
\]
for nonconnective algebraic \(K\)-theory
\cite[\arxivhtmlref{1001.2282v4}{S9}{\S9}]{BGT13KTheoryStableCategories}.
It satisfies localization for perfect complexes
\cite[Thm.~7.4]{ThomasonTrobaugh90}. In the language of small stable
\(\infty\)-categories, nonconnective \(K\)-theory is a localizing invariant
\cite[\arxivhtmlref{1001.2282v4}{S9.SS3}{\S9.3}]{BGT13KTheoryStableCategories}. For \(N\ge1\), put
\begin{equation}\label{eq:finite-coefficient-k-theory}
K(Y;\mathbf Z/N):=\operatorname{cofib}\bigl(K(Y)\xrightarrow{N}K(Y)\bigr).
\end{equation}
On affine regular noetherian schemes of finite Krull dimension, the
resolution theorem identifies \(K(\Perf(Y))\) in nonnegative degrees with
Quillen \(K\)-theory, and negative \(K\)-groups vanish. Thus
\eqref{eq:finite-coefficient-k-theory} agrees there with
finite-coefficient Quillen \(K\)-theory in
nonnegative degrees.
\end{convention}

\begin{definition}[Gersten injectivity]
\label{def:generic-injectivity}
Let \(R\) be a domain and let \(N\ge1\) be invertible in \(R\). We say
that \(R\) satisfies \emph{Gersten injectivity with
\(\mathbf Z/N\)-coefficients} if
\[
K_n(R;\mathbf Z/N)\longrightarrow K_n(\Frac(R);\mathbf Z/N)
\]
is injective for every \(n\in\mathbf Z\).
\end{definition}

For a regular local domain, this is precisely injectivity at the first
arrow of the Gersten complex. We use the same terminology for the
valuation-ring and Prüfer-base settings below; no assertion about a full
Gersten complex or exactness at its subsequent terms is intended.

\begin{convention}[valuation-theoretic rank and henselianity]
\label{conv:valuation-rank}
A valuation ring is allowed to be a field. Its \emph{rank} is its Krull
dimension. Equivalently, it has finite rank \(r\) if and only if its
totally ordered value group has exactly \(r\) proper convex subgroups
\cite[Lemma~2.3.1]{EnglerPrestel05ValuedFields}.

Its ideals, hence its prime ideals, are
linearly ordered \cite[\stackstag{0ASN}]{StacksProject}. Consequently, a valuation ring of finite
rank \(r\) has exactly \(r+1\) prime ideals, a maximal chain has \(r\) strict inclusions,
rank zero is equivalent to being a field, and a positive finite-rank valuation ring has a
smallest nonzero prime.

A valuation ring \(W\) is \emph{henselian} if it is henselian as a local ring, that is, if a
simple root modulo \(\mathfrak m_W\) of a polynomial over \(W\) lifts to a root in \(W\)
\cite[\stackstag{04GG}\textup{(2)}]{StacksProject}. For a valuation ring this is equivalent
to the valued field \((\Frac(W),v)\) being henselian in the sense of
\cite[\S4.1, p.~86]{EnglerPrestel05ValuedFields}. In that reference, henselianity is
defined by uniqueness of prolongation to every algebraic extension; its equivalence with
simple-root lifting is \cite[Thm.~4.1.3\textup{(4)}]{EnglerPrestel05ValuedFields}.
\end{convention}

\begin{convention}[Galois cohomology, Tate twists, and Kummer classes]
\label{conv:galois-cohomology-kummer}
For a field \(F\), write \(G_F\) for its absolute Galois group and
\(H^a(F,M):=H^a_{\mathrm{cont}}(G_F,M)\)
for continuous cohomology with coefficients in a discrete torsion \(G_F\)-module \(M\).
If \(m\) is invertible in \(F\), set \(\mu_m^{\otimes0}:=\mathbf Z/m\); for \(q>0\) use
the usual tensor power, and for \(q<0\) set
\[
\mu_m^{\otimes q}:=
\bigl(\operatorname{Hom}_{\mathbf Z/m}(\mu_m,\mathbf Z/m)\bigr)^{\otimes(-q)}.
\]
The same notation denotes the corresponding locally constant étale sheaf on a scheme on
which \(m\) is invertible. If \(I\) is a closed normal subgroup of a profinite group \(G\)
and \(M\) is a discrete \(G/I\)-module, then
\(\operatorname{inf}_{G/I}^{G}\colon
H^a_{\mathrm{cont}}(G/I,M)\to H^a_{\mathrm{cont}}(G,M)\)
denotes the map induced by precomposition with \(G\twoheadrightarrow G/I\), where \(M\)
is regarded as a \(G\)-module through this quotient. We use the same notation for the
induced map \(R\Gamma(G/I,M)\to R\Gamma(G,M)\), and write simply
\(\operatorname{inf}\) when \(G\) and \(I\) are clear.

We use the identification of étale and Galois cohomology over a field:
\(H^a_{\et}(\Spec(F),\mu_m^{\otimes q})=H^a(F,\mu_m^{\otimes q})\)
\cite[\stackstag{03QQ} and \stackstag{03QU}]{StacksProject}. For a
scheme \(X\), write
\[
L_{\et}\colon
D\bigl(\operatorname{PSh}(X_{\et},\mathbf Z)\bigr)
\longrightarrow
D\bigl(\operatorname{Sh}(X_{\et},\mathbf Z)\bigr)
\]
for derived étale sheafification of complexes of abelian presheaves on
the small étale site. We use the same notation after changing the
coefficient ring.
Sheafification of modules is exact and commutes with tensor products
\cite[\stackstag{03CX} and \stackstag{03EK}]{StacksProject}; hence its derived functor is an
exact symmetric monoidal localization. The derived category of sheaves
carries the usual derived tensor product
\cite[\stackstag{0FPT}]{StacksProject}. Stalks at separable closures are conservative on the
small étale site \cite[\stackstag{03PN}]{StacksProject}. For \(a\in F^\times\),
write \(\delta_m(a)\in H^1(F,\mu_m)\) for its Kummer class.
\end{convention}

\begin{convention}[filtered spectra, tensor products, and completeness]
\label{conv:filtered-spectral-sequence-indexing}
A \emph{filtered spectrum} is a functor
\(\mathbf Z^{\mathrm{op}}\to\mathrm{Sp}\). We write
\(\FilSp:=\operatorname{Fun}(\mathbf Z^{\mathrm{op}},\mathrm{Sp})\), and write a
filtered spectrum as \(F^\bullet E\) or
\(\Fil^\bullet E\). 
We equip \(\FilSp\) with the Day-convolution symmetric monoidal structure
induced by addition on \(\mathbf Z^{\mathrm{op}}\). Explicitly,
\[
(E^\bullet\otimes G^\bullet)^q
\simeq
\operatorname*{colim}_{a+b\ge q}
E^a\wedge G^b.
\]
Thus a filtered pairing is specified by compatible maps
\(E^a\wedge G^b\to H^{a+b}\).
Its graded pieces are
\[
\operatorname{gr}^qE:=\operatorname{cofib}(F^{q+1}E\to F^qE).
\]
An \(\mathbf N\)-indexed filtration is extended to \(\mathbf Z\) by
\(F^qE=F^0E\) for \(q\le0\); thus
\(\operatorname{gr}^qE=0\) for \(q<0\). The filtration is
\emph{complete} if
\(\displaystyle\lim_{q\to\infty}F^qE\simeq0\) in \(\mathrm{Sp}\),
and its \emph{underlying spectrum} is
\[
|F^\bullet E|:=\operatorname*{colim}_{q\to-\infty}F^qE.
\]
For an \(\mathbf N\)-indexed filtration this is \(F^0E\). If the
filtration is equipped with compatible maps \(F^qE\to E\), it is
\emph{exhaustive} when the induced map \(|F^\bullet E|\to E\) is an
equivalence.
\end{convention}

\section{Finite-coefficient motivic filtrations}
\label{sec:finite-coefficient-motivic-filtrations}

We use Bouis's multiplicative motivic filtration on nonconnective
\(K\)-theory and the Beilinson--Lichtenbaum comparison of Bouis--Kundu.
The precise forms needed below are recorded in this section.

\begin{definition}[Bouis--Kundu notation]
\label{def:bouis-kundu-notation}
A \emph{Prüfer domain} is a domain whose localization at every prime is a
valuation ring. A \emph{Prüfer ring} is a finite product of Prüfer domains.
A \(P\)-algebra is \emph{ind-smooth} if it is a filtered colimit of smooth
\(P\)-algebras; an affine \(P\)-scheme is ind-smooth when its ring of
global sections is. We use cohomological grading; \(L_{\mathrm{Nis}}\) denotes
derived Nisnevich sheafification. Bouis's motivic complexes are normalized by
\[
\mathbf Z(q)_{\mot}(T)
=\operatorname{gr}^q_{\mot}K(T)[-2q],
\qquad
\mathbf Z/m(q)_{\mot}
=\mathbf Z(q)_{\mot}\otimes^{\mathbf L}\mathbf Z/m.
\]
We write
\[
H^j_{\mot}(T,\mathbf Z(q)):=H^j\bigl(\mathbf Z(q)_{\mot}(T)\bigr),
\qquad
H^j_{\mot}(T,\mathbf Z/m(q)):=H^j\bigl(\mathbf Z/m(q)_{\mot}(T)\bigr).
\]
For the construction see
\cite[Def.~3.18]{BouisMixedCharacteristicMotivicCohomology}; the
normalization is \cite[Def.~3.20 and Thm.~C\textup{(1)}]{BouisMixedCharacteristicMotivicCohomology}.
The complex \(\mathbf Z/p^\nu(q)_{\syn}\) denotes the syntomic Tate
twist used in \cite{BouisKunduBL}; when \(p\) is invertible it is
the étale Tate twist \(\mu_{p^\nu}^{\otimes q}\)
\cite[\arxivhtmlref{2202.04818v2}{S1.Thmtheorem2}{Ex.~1.2}]{BhattMathew23SyntomicTateTwists}.
\end{definition}

\begin{lemma}[Prüfer Beilinson--Lichtenbaum comparison]
\label{lem:ordinary-prufer-bl-range-input}
Let \(P\) be a Prüfer ring, let \(Y\) be ind-smooth over \(P\), let
\(p\) be a prime, and let \(\nu\ge1\) and \(q\ge0\). The comparison
map of \cite[\arxivhtmlref{2506.09910v2}{S4.Thmtheorem1.E1}{(4.1.1)}]{BouisKunduBL} is an equivalence
\[
\bigl(L_{\mathrm{Nis}}\tau^{\le q}
\mathbf Z/p^\nu(q)_{\syn}\bigr)(Y)
\xrightarrow{\ \simeq\ }
\mathbf Z/p^\nu(q)_{\mot}(Y).
\]
If \(p\) is invertible on \(Y\), its source is
\(
\bigl(L_{\mathrm{Nis}}\tau^{\le q}
R\Gamma_{\et}(-,\mu_{p^\nu}^{\otimes q})\bigr)(Y).
\)
If \(Y=\Spec(A)\) with \(A\) henselian local, Nisnevich sheafification
does not change evaluation at \(A\); hence
\[
\tau^{\le q}R\Gamma(A_{\et},\mu_{p^\nu}^{\otimes q})
\xrightarrow{\ \simeq\ }
\mathbf Z/p^\nu(q)_{\mot}(A)
\]
whenever \(p\in A^\times\). The same statement applies to fields.
\end{lemma}

\begin{proof}
The first map is an equivalence by
\cite[\arxivhtmlref{2506.09910v2}{S5.Thmtheorem1}{Thm.~5.1}]{BouisKunduBL}. On the
locus where \(p\) is
invertible, the syntomic Tate twist is the étale Tate twist by
\cite[\arxivhtmlref{2202.04818v2}{S1.Thmtheorem2}{Ex.~1.2}]{BhattMathew23SyntomicTateTwists}. Finally, every
Nisnevich covering of the spectrum of a henselian local ring has a
member containing a point over the closed point with the same residue
field. After replacing that member by an affine open neighbourhood of
the point, henselian lifting extends it to a section over
\(\Spec(A)\) \cite[\stackstag{04GG}\textup{(7),(8)}]{StacksProject}. 
Thus every Nisnevich covering sieve of \(\Spec(A)\) is maximal;
equivalently, \(\Spec(A)\) is local for the Nisnevich topology.
Consequently, evaluation at \(A\) is unchanged by derived Nisnevich
sheafification.

\end{proof}

\begin{lemma}[valuative dimension of a field]
\label{lem:valuation-inputs-finite-valuative-dimension}
For every field \(F\), the scheme \(\Spec(F)\) has valuative dimension
zero.
\end{lemma}

\begin{proof}
By \cite[\arxivhtmlref{2002.11647v3}{S2.SS3}{\S2.3}]{ElmantoHoyoisIwasaKelly2021}, the valuative dimension of
\(\Spec(F)\) is the supremum of the ranks of valuation rings \(V\) with
\(F\subseteq V\subseteq\Frac(F)=F\). The only such ring is \(F\), which
has rank zero.
\end{proof}

\begin{lemma}[finite-coefficient motivic filtration]
\label{lem:global-motivic-filtration-kmodell-input}
Let \(T\) be a qcqs scheme. Bouis's motivic filtration is a
contravariantly functorial multiplicative \(\mathbf N\)-indexed
filtration \(\Fil_{\mot}^{\bullet}K(T)\) with
\[
\Fil_{\mot}^{0}K(T)\simeq K(T),
\qquad
\operatorname{gr}_{\mot}^{q}K(T)
\simeq\mathbf Z(q)_{\mot}(T)[2q].
\]
If \(T\) has finite valuative dimension, this filtration is complete.

For \(m\ge1\), put
\[
\Fil_{\mot}^{\bullet}K(T;\mathbf Z/m)
:=\Fil_{\mot}^{\bullet}K(T)\wedge\mathbf S/m.
\]
It is exhaustive and, when \(T\) has finite valuative dimension,
complete. Its graded pieces are
\[
\operatorname{gr}_{\mot}^{q}K(T;\mathbf Z/m)
\simeq\mathbf Z/m(q)_{\mot}(T)[2q].
\]
It is naturally a filtered module over
\(\Fil_{\mot}^{\bullet}K(T)\).
\end{lemma}

\begin{proof}
Contravariant functoriality and multiplicativity are part of the
construction in
\cite[Def.~3.18]{BouisMixedCharacteristicMotivicCohomology}. The
normalization and the graded-piece formula are
\cite[Def.~3.20 and Thm.~C\textup{(1)}]{BouisMixedCharacteristicMotivicCohomology},
while \(\mathbf N\)-indexedness and exhaustiveness are
\cite[Prop.~4.46]{BouisMixedCharacteristicMotivicCohomology}.
Completeness in finite valuative dimension is
\cite[\arxivhtmlref{2412.06635v2}{S4.Thmtheorem49}{Prop.~4.49}]{BouisMixedCharacteristicMotivicCohomology}.
The Moore
spectrum \(\mathbf S/m\) is finite, hence dualizable. Smashing with a
dualizable spectrum has both a left and a right adjoint, given by smashing
with its dual, and therefore preserves colimits and limits. It is exact, so
it preserves the graded-piece formula, exhaustiveness, and completeness.
The module structure is induced by the multiplicative filtration.
\end{proof}

\begin{theorem}[Gabber rigidity for henselian pairs]
	\label{thm:gabber-rigidity-finite-coefficient-k}
	Let \((A,I)\) be a henselian pair and let \(m\ge1\) be invertible in \(A\). Then
	reduction induces an equivalence of connective \(K\)-theory spectra
	\[
	K^{\mathrm{cn}}(A;\mathbf Z/m)
	\xrightarrow{\ \simeq\ }
	K^{\mathrm{cn}}(A/I;\mathbf Z/m).
	\]
	In particular
	\(\pi_iK^{\mathrm{cn}}(A;\mathbf Z/m)\to\pi_iK^{\mathrm{cn}}(A/I;\mathbf Z/m)\) is an
	isomorphism for every \(i\ge0\).
\end{theorem}
\begin{proof}
	This is Gabber's rigidity theorem \cite{Gabber92Rigidity}. For a published account in
	exactly this form, with the same connectivity convention, see
	\cite[\arxivhtmlref{1803.10897v2}{S1.Thmdefinition4}{Thm.~1.4}]{ClausenMathewMorrow21HenselianPairs}.
\end{proof}

\begin{remark}[connectivity caveat]
The theorem does not give rigidity for nonconnective spectra: if
\(A/I\) is singular, the groups \(K_i(A/I)\) for \(i<0\) need not vanish,
and the comparison in degree zero must be made separately. For valuation
rings this caveat disappears by
\Cref{lem:negative-finite-coefficient-k-valuation-domain}.
\end{remark}

\begin{lemma}[negative \(K\)-theory of a valuation ring, integrally and with finite coefficients]
	\label{lem:negative-finite-coefficient-k-valuation-domain}
	Let \(W\) be a valuation domain. Then
	\(K_i(W)=0 \qquad(i<0)\),
	so \(K(W)\) is connective and in particular \(K_{-1}(W)=0\). Consequently, for every
	\(m\ge1\),
	\(\pi_iK(W;\mathbf Z/m)=0 \qquad(i<0)\).
\end{lemma}
\begin{proof}
	The integral vanishing is
	\cite[\arxivhtmlref{1810.12203v1}{S1.I1.i3}{Thm.~1.3(iii)}]{KellyMorrowValuationRings}, which
	asserts \(K_n(\mathcal O)=0\) for all \(n<0\) and every valuation ring \(\mathcal O\); no
	noetherian, henselian, rank or residue-characteristic hypothesis enters there. A spectrum
	whose negative homotopy groups vanish is connective, so \(K(W)\) is connective.

	For the finite-coefficient assertion, \(K(W;\mathbf Z/m)\) is the cofiber of
	multiplication by \(m\) on \(K(W)\) by \Cref{conv:nonconnective-k-groups}, so
	the universal-coefficient sequence
	\[
	0\longrightarrow K_i(W)/m\longrightarrow\pi_iK(W;\mathbf Z/m)\longrightarrow
	K_{i-1}(W)[m]\longrightarrow0
	\]
	is exact for every \(i\). If \(i<0\) then also \(i-1<0\), so both outer terms vanish by the
	first assertion, and \(\pi_iK(W;\mathbf Z/m)=0\).
\end{proof}

\begin{lemma}[valuation-ring closed-point rigidity for nonconnective finite coefficients]
	\label{lem:valuation-closed-point-finite-coefficient-k-rigidity}
	Let \(W\) be a henselian valuation domain with residue field \(k\), and let \(m\ge1\) be
	invertible in \(W\). Then the reduction map induces an equivalence of nonconnective
	finite-coefficient \(K\)-theory spectra
	\(K(W;\mathbf Z/m) \xrightarrow{\ \simeq\ } K(k;\mathbf Z/m)\).
\end{lemma}

\begin{proof}
	The connective finite-coefficient comparison
	\[
	K^{\mathrm{cn}}(W;\mathbf Z/m)
	\xrightarrow{\ \simeq\ }
	K^{\mathrm{cn}}(k;\mathbf Z/m)
	\]
	holds by \Cref{thm:gabber-rigidity-finite-coefficient-k}, in the henselian-local case applied
	to the pair \((W,\mathfrak m_W)\).

	By \Cref{lem:negative-finite-coefficient-k-valuation-domain} the integral negative
	\(K\)-groups \(K_i(W)\) vanish for \(i<0\). Hence \(K(W)\) is connective and the canonical map
	\(K^{\mathrm{cn}}(W)\to K(W)\) is an equivalence; applying \(-\,/m\) gives
	\(K^{\mathrm{cn}}(W;\mathbf Z/m)\xrightarrow{\ \simeq\ }K(W;\mathbf Z/m)\).
	Since \(k\) is a field, \(K(k)\) is likewise connective, so
	\(K^{\mathrm{cn}}(k;\mathbf Z/m)\xrightarrow{\ \simeq\ }K(k;\mathbf Z/m)\).

	Composing these three equivalences identifies the reduction map
	\(K(W;\mathbf Z/m)\to K(k;\mathbf Z/m)\) as an equivalence.
\end{proof}

\begin{theorem}[field range for the prime-power motivic filtration]
\label{thm:field-finite-coefficient-k-filtration-input}
Let \(F\) be a field, let \(\ell\) be a prime invertible in \(F\), let \(\nu\ge1\), and
put \(N:=\ell^\nu\). For the Bouis motivic filtration of
\Cref{lem:global-motivic-filtration-kmodell-input},
\[
\pi_t\operatorname{gr}^q_{\mot}K(F;\mathbf Z/N)
\cong H^{2q-t}_{\mot}(F,\mathbf Z/N(q)).
\]
This group vanishes unless
\(q\ge0,\qquad 0\le 2q-t\le q\).
In the nonzero range, the Beilinson--Lichtenbaum comparison identifies it functorially with
\(H^{2q-t}_{\et}(F,\mu_N^{\otimes q})\).
Consequently, only finitely many weights occur in each homotopy degree: for \(t>0\),
\(\lceil t/2\rceil\le q\le t\); for \(t=0\), only \(q=0\); and for \(t<0\), none.
\end{theorem}

\begin{proof}
The formula for the graded pieces is
\Cref{lem:global-motivic-filtration-kmodell-input}. Apply
\Cref{lem:ordinary-prufer-bl-range-input} with coefficient prime \(\ell\) and exponent
\(\nu\) to the field \(F\), viewed as a henselian local ind-smooth algebra over itself. It
gives
\(\mathbf Z/N(q)_{\mot}(F)\simeq \tau^{\le q}R\Gamma(F_{\et},\mu_N^{\otimes q})\)
for \(q\ge0\). For \(q<0\) there is nothing to prove: the filtration is \(\mathbf N\)-indexed
(\Cref{lem:global-motivic-filtration-kmodell-input}), so \(\operatorname{gr}^q_{\mot}\) is
zero by convention rather than by a vanishing theorem. The right-hand
side is concentrated in degrees between \(0\) and \(q\): in degrees at most \(q\) by the
truncation, and in nonnegative degrees because derived global sections of a sheaf placed in
degree zero are. This proves both the range and the étale identification. Solving
\(0\le2q-t\le q\) yields the asserted bounds on the weights.
\end{proof}

\begin{lemma}[finite-window truncation of complete filtered spectra]
	\label{lem:complete-filtered-spectrum-finite-window-truncation}
	Let \(E\) be a complete decreasing \(\mathbf N\)-indexed filtered spectrum, with filtration
	\(F^\bullet E\) and associated graded pieces
	\(\operatorname{gr}^qE:=\operatorname{cofib}(F^{q+1}E\to F^qE)\).
	Let \(I=[a,b]\subset\mathbf Z\) be a finite interval and put
	\(I^\sharp:=[a-1,b+1]\).
	Suppose that there is an integer \(h\ge0\) such that
	\(
	\pi_i\operatorname{gr}^qE=0
	\)
	for every \(i\in I^\sharp\) and every \(q>h\). Then
	\(\pi_jF^{h+1}E=0\) for every \(j\) with \(a-1\le j\le b\), and consequently the quotient
	map
	\(E\longrightarrow \operatorname{cofib}(F^{h+1}E\to E)\)
	induces an isomorphism on \(\pi_i\) for every \(i\in I\).
\end{lemma}

\begin{proof}
	Write
	\(
	E_{\le h}:=\operatorname{cofib}(F^{h+1}E\to E).
	\)
	The fibre of the quotient map
	\(
	E\to E_{\le h}
	\)
	is \(F^{h+1}E\). Once
	\(\pi_jF^{h+1}E=0 \qquad (a-1\le j\le b)\)
	is proved, the long exact homotopy sequence gives the second assertion.

	We use completeness in the sense of
	\Cref{conv:filtered-spectral-sequence-indexing}, namely
	\(\displaystyle\lim_{r\to\infty}F^rE\simeq0\).
	For \(r>h+1\), put
	\(Q_{h,r}:=\operatorname{cofib}(F^rE\to F^{h+1}E)\).
	The fibre of the natural map
	\(F^{h+1}E\longrightarrow \displaystyle\lim_{r>h+1}Q_{h,r}\)
	is
	\(
		\displaystyle\lim_{r>h+1}F^rE,
	\)
	which is zero by completeness. Hence
	\(
	F^{h+1}E\simeq \displaystyle\lim_{r>h+1}Q_{h,r}.
	\)

	Each finite quotient \(Q_{h,r}\) has a finite decreasing filtration whose successive
	graded pieces are \(\operatorname{gr}^{q}E\) for \(h+1\le q<r\).
	By assumption, each of these graded pieces has trivial \(\pi_i\) for all
	\(i\in I^\sharp=[a-1,b+1]\).
	Induction on the length of the finite filtration, using the long exact homotopy sequence of
	each successive fibre/cofiber sequence, gives
	\(\pi_iQ_{h,r}=0 \qquad(i\in I^\sharp)\)
	for every \(r>h+1\).

	For \(r>h+1\), let \(u_r\colon Q_{h,r+1}\to Q_{h,r}\) be the
	transition map. The standard formula for the limit of a sequential
	tower in spectra gives a fibre sequence
	\[
	\lim_{r>h+1}Q_{h,r}\longrightarrow
	\prod_{r>h+1}Q_{h,r}
	\xrightarrow{\ d\ }
	\prod_{r>h+1}Q_{h,r},
	\]
	where the \(r\)-th component of \(d\) is
	\(\operatorname{pr}_r-u_r\operatorname{pr}_{r+1}\). Thus \(d\) is
	the difference between the identity and the shift defined by the
	transition maps. Homotopy groups commute with products of spectra. If
	\(a-1\le j\le b\), then \(j,j+1\in I^\sharp\), so the vanishing already
	proved for the \(Q_{h,r}\) gives
	\(
	\pi_j\prod_{r>h+1}Q_{h,r}
	=
	\pi_{j+1}\prod_{r>h+1}Q_{h,r}
	=0.
	\)
	The long exact homotopy sequence of this fibre sequence therefore gives
	\(\pi_j\lim_{r>h+1}Q_{h,r}=0\). Using
	\(F^{h+1}E\simeq\lim_{r>h+1}Q_{h,r}\), we obtain
	\(\pi_jF^{h+1}E=0\) for \(a-1\le j\le b\).
	The long exact homotopy sequence of
	\(F^{h+1}E\longrightarrow E\longrightarrow E_{\le h}\)
	now gives that \(E\to E_{\le h}\) is an isomorphism on \(\pi_i\) for every \(i\in I\).
\end{proof}

\begin{lemma}[connectivity of the motivic filtration of a field]
\label{lem:field-filtration-connectivity}

Let \(F\) be a field. Let
\(E^\bullet=\Fil^\bullet_{\mot}K(F)\). Alternatively, let \(\ell\) be a prime
invertible in \(F\), let \(\nu\ge1\), put \(N:=\ell^\nu\), and let
\(E^\bullet=\Fil^\bullet_{\mot}K(F;\mathbf Z/N)\). Then:
\begin{enumerate}[label=\textup{(\roman*)},leftmargin=*]
\item \(\pi_iE^q=0\) for every \(q\ge0\) and every \(i\le q-2\);

\item if \(q\ge0\) and \(i\in\mathbf Z\) are such that
\(\pi_i\operatorname{gr}^{q'}E^\bullet=0\) for every \(q'\ge q\),
then \(\pi_iE^q=0\).
\end{enumerate}
\end{lemma}

\begin{proof}
\textup{(i)} For \(q=0\) the assertion is connectivity of \(K(F)\) and of
\(K(F;\mathbf Z/N)\), which is
\Cref{lem:negative-finite-coefficient-k-valuation-domain} applied to the valuation ring
\(F\). Let \(q\ge1\). Both filtrations are complete by
\Cref{lem:global-motivic-filtration-kmodell-input}, the scheme \(\Spec(F)\) having valuative
dimension \(0\) by
\Cref{lem:valuation-inputs-finite-valuative-dimension}. With finite coefficients,
\Cref{thm:field-finite-coefficient-k-filtration-input} gives
\(\pi_i\operatorname{gr}^{q'}_{\mot}K(F;\mathbf Z/N)=0\) unless \(q'\le i\); integrally,
\(\pi_i\operatorname{gr}^{q'}_{\mot}K(F)=H^{2q'-i}_{\mot}(F,\mathbf Z(q'))\) vanishes unless
\(q'\le i\), by
\cite[\arxivhtmlref{2507.16501v1}{Thmtheoremintro7}{Thm.~G}]{BouisWeibelVanishingPBF}
applied to the noetherian
zero-dimensional scheme \(\Spec(F)\). In both cases the graded pieces of weight
\(q'\ge q\) have vanishing \(\pi_i\) for \(i\le q-1\). Apply
\Cref{lem:complete-filtered-spectrum-finite-window-truncation} with \(h:=q-1\) and
\(I:=[c,q-2]\) for \(c\le q-2\), so that \(I^\sharp=[c-1,q-1]\): the hypothesis holds, and
	the conclusion is \(\pi_iE^q=0\) for \(c-1\le i\le q-2\). Since \(c\) is arbitrary this
proves \textup{(i)}.

\textup{(ii)} By \textup{(i)} we have \(\pi_iE^{q'}=0\) for \(q'\ge i+2\). For
\(q'\ge q\) the exact sequence
\(\pi_iE^{q'+1}\to\pi_iE^{q'}\to
\pi_i\operatorname{gr}^{q'}E^\bullet=0\) shows that
\(\pi_iE^{q'+1}\to\pi_iE^{q'}\) is surjective. Put
\(Q:=\max\{q,i+2\}\). By \textup{(i)}, \(\pi_iE^Q=0\). If \(Q>q\),
apply the preceding surjections successively for
\(q'=Q-1,Q-2,\ldots,q\). This gives \(\pi_iE^q=0\).
\end{proof}

\section{Tame inertia at prime-power level}
\label{sec:tame-inertia}

We establish below a Tate-twisted form of the tame value-degree
decomposition for arbitrary value groups and coefficients
\(\mathbf Z/\ell^\nu\), without assuming
\(\mu_{\ell^\nu}\subset L\). Under the additional assumption
\(\mu_{\ell^\nu}\subset L\), Wadsworth proves the corresponding
decomposition for arbitrary value groups, including its graded-ring
structure
\cite[Hypotheses~3.5 and Thm.~3.6]{Wadsworth83pHenselian}.
A Hochschild--Serre proof in that setting is recorded, with credit to
Jacob, in \cite[Rem.~3.15]{Wadsworth83pHenselian};
\cite[Rem.~3.16]{Wadsworth83pHenselian} observes that the spectral-sequence
method still gives some decomposition without roots of unity. The
formulation below makes that decomposition explicit, retains the Tate
twists and the Galois action on tame inertia, and identifies it with the
Hochschild--Serre filtration.

Let \(L\) carry a henselian valuation \(v\), with valuation ring \(W\),
residue field \(k\), and value group \(\Gamma_W\). We use the following
facts.
\begin{enumerate}[label=\textup{(V\arabic*)},leftmargin=*]
\item\label{it:henselian-unique-prolongation}
\(v\) has a unique prolongation to every algebraic extension of \(L\); with this
prolongation every algebraic extension is again henselian
\cite[\S4.1, p.~86 and Lem.~4.1.1]{EnglerPrestel05ValuedFields}.

\item\label{it:hensel-lifting}
If \(m\) is invertible in \(W\) and \(u\in W^\times\) has an \(m\)-th root in \(k\), then
\(u\) has an \(m\)-th root in \(W\): the polynomial \(X^m-u\) is separable modulo
\(\mathfrak m_W\), so a simple root of its reduction lifts
\cite[\stackstag{04GG}\textup{(2)}]{StacksProject}. In particular, if \(k\) is separably closed and \(m\) is
invertible, then \(W^\times\) is \(m\)-divisible and \(\mu_m\subset L\).

\item\label{it:value-group-torsion-free}
\(\Gamma_W\) is torsion free, since \(\gamma>0\) implies \(m\gamma\ge\gamma>0\) for every
\(m\ge1\); being torsion free, it is flat over \(\mathbf Z\)
\cite[\stackstag{0AUW}]{StacksProject}.

\item\label{it:residue-field-algebraic}
The residue field of an algebraic extension of \(L\) is algebraic over \(k\). Indeed, let
\(x\) lie in the corresponding valuation ring. Uniqueness of the prolongation to an
algebraic closure implies that every conjugate of \(x\) has nonnegative valuation. The
coefficients of its monic minimal polynomial, being elementary symmetric functions of
these conjugates (counted with multiplicity), lie in \(W\). Thus \(x\) is integral over
\(W\), and reducing its equation shows that its residue class is algebraic over \(k\).
\end{enumerate}

For the classical theory of unramified and ramified extensions used in the following lemma, see
\cite[\S\S5.2--5.3]{EnglerPrestel05ValuedFields}.
We give the proof with the Galois actions and transition maps
used in \Cref{lem:henselian-tame-inertia-cohomology}.

\begin{lemma}[strict henselization and tame quotient of a henselian valuation ring]
	\label{lem:henselian-valuation-strict-henselisation-wild-inertia}
	Let \(W\) be a henselian valuation ring with fraction field \(L\), residue field \(k\) of
	characteristic \(p\ge0\), and value group \(\Gamma_W\). Let \(W^{\mathrm{sh}}\) be a strict
	henselization of \(W\), let \(L^{\mathrm{ur}}\) be its fraction field, and inside a separable
	closure put
	\[
	L^{t}:=
	L^{\mathrm{ur}}\bigl(\sqrt[m]{a}\ :\ a\in(L^{\mathrm{ur}})^\times,\ m\ge1\text{ prime to }p\bigr),
	\]
	with the convention that every \(m\ge1\) is prime to \(p\) when \(p=0\). Then:
	\begin{enumerate}[label=\textup{(\roman*)},leftmargin=*]
		\item \(W^{\mathrm{sh}}\) is a henselian valuation ring with value group \(\Gamma_W\) and
		residue field \(k^{\mathrm{sep}}\), and the sequence
	\[
		1\longrightarrow
		\operatorname{Gal}(L^{\mathrm{sep}}/L^{\mathrm{ur}})
		\longrightarrow\operatorname{Gal}(L^{\mathrm{sep}}/L)
		\longrightarrow\operatorname{Gal}(k^{\mathrm{sep}}/k)\longrightarrow1
	\]
		is exact;

		\item the value group of \(L^{t}\) is the prime-to-\(p\) divisible hull
		\(
		\Gamma_W':=\bigcup_{(m,p)=1}\tfrac1m\Gamma_W
		\)
		of \(\Gamma_W\) inside \(\Gamma_W\otimes_{\mathbf Z}\mathbf Q\), and Kummer theory
		identifies
	\[
		\operatorname{Gal}(L^{t}/L^{\mathrm{ur}})
		\simeq
		\varprojlim_{m}\operatorname{Hom}\bigl(\Gamma_W/m\Gamma_W,\mu_m\bigr),
	\]
		as profinite groups. Here the limit is taken over the integers \(m\ge1\) prime to
		\(p\), ordered by divisibility, and every \(\operatorname{Hom}\)-group carries the
		topology of pointwise convergence. 
		
		For \(m\mid m'\), the transition map
		\[
		\operatorname{Hom}(\Gamma_W/m'\Gamma_W,\mu_{m'})
		\longrightarrow
		\operatorname{Hom}(\Gamma_W/m\Gamma_W,\mu_m)
		\]
		sends \(\chi\) to the character
		\(
		\overline\gamma\longmapsto
		\chi(\gamma\bmod m'\Gamma_W)^{m'/m}.
		\)
		In particular this
		group is abelian;

		\item \(\operatorname{Gal}(L^{\mathrm{sep}}/L^{t})\) is a pro-\(p\) group, trivial when
		\(p=0\).
	\end{enumerate}
\end{lemma}

\begin{proof}
\textup{(i)} By \cite[\stackstag{0ASK}]{StacksProject}, the strict
henselization \(W^{\mathrm{sh}}\) is a valuation ring and
\(W\subseteq W^{\mathrm{sh}}\) induces an isomorphism of value groups. Its
residue field is the chosen separable closure \(k^{\sep}\), by the
construction of strict henselization
\cite[\stackstag{04GW}]{StacksProject}.

Let \(\mathcal O\) be the valuation ring of the unique prolongation to
\(L^{\sep}\), and let \(\kappa\) be its residue field. Choose
\(W^{\mathrm{sh}}\subseteq\mathcal O\) using the chosen pointed finite
étale neighbourhoods defining the strict henselization. On residue
fields this realizes the fixed inclusion \(k^{\sep}\subseteq\kappa\).
By condition~\ref{it:residue-field-algebraic}, the extension
\(\kappa/k\) is algebraic. Hence \(k^{\sep}\) is precisely the
subfield of \(\kappa\) consisting of the elements separable over \(k\).
Since \(W\) is henselian, reduction induces an equivalence between finite
étale \(W\)-algebras and finite étale \(k\)-algebras
\cite[\stackstag{04GK}]{StacksProject}. Uniqueness of the prolongation
makes \(\mathcal O\) stable under \(G_L\). The preceding characterization
of \(k^{\sep}\) makes this subfield stable under the reduction action, and
reduction followed by restriction to \(k^{\sep}\subseteq\kappa\) defines
the homomorphism \(G_L\to G_k\) that sends
\(\sigma\) to \(\overline{\sigma}|_{k^{\sep}}\).

The finite étale equivalence shows that this homomorphism is surjective, and
its kernel consists precisely
of the automorphisms fixing the fraction fields of all pointed connected
finite étale \(W\)-algebras. Their compositum is
\(\Frac(W^{\mathrm{sh}})=L^{\mathrm{ur}}\), which gives the exact sequence in
\textup{(i)}.

	(ii) We first record a general fact. Let \(M\) be a henselian valued field and let
	\(M'/M\) be a finite separable extension of degree \(d\). By condition~\ref{it:henselian-unique-prolongation}, the valuation of \(M\) extends uniquely to \(M'\);
	since \(v\circ\sigma\) is again such a prolongation, \(v\circ\sigma=v\) for every
	\(M\)-embedding \(\sigma\) of \(M'\) into a separable closure. Separability gives
	\(N_{M'/M}(x)=\prod_\sigma\sigma(x)\), the product running over the \(d\) distinct
	embeddings, so \(v\bigl(N_{M'/M}(x)\bigr)=d\,v(x)\) for every \(x\in(M')^\times\). Hence
	\[
	d\,\Gamma_{M'}\subseteq\Gamma_M.
	\tag{$\ast$}
	\]
	Condition~\ref{it:henselian-unique-prolongation} shows that every field occurring below is henselian, so
	\((\ast)\) applies to each of them.

	Every element of \((L^{t})^\times\) lies in some
	\(M'=L^{\mathrm{ur}}(\sqrt[m_1]{a_1},\dots,\sqrt[m_r]{a_r})\) with all \(m_i\) prime to \(p\),
	and \([M':L^{\mathrm{ur}}]\) divides \(m_1\cdots m_r\), hence is prime to \(p\); so
	\(\Gamma_{L^t}\subseteq\Gamma_W'\) by \((\ast)\) and (i). Conversely \(\sqrt[m]{a}\in L^{t}\)
	has value \(\tfrac1mv(a)\), so \(\Gamma_{L^t}=\Gamma_W'\).

	Since \(k^{\mathrm{sep}}\) is separably closed and \(m\) is prime to \(p\),
	condition~\ref{it:hensel-lifting} makes \((W^{\mathrm{sh}})^\times\) an \(m\)-divisible group and
	gives \(\mu_m\subset L^{\mathrm{ur}}\); the exact sequence
	\(1\to(W^{\mathrm{sh}})^\times\to(L^{\mathrm{ur}})^\times\to\Gamma_W\to0\) therefore
	identifies \((L^{\mathrm{ur}})^\times/\bigl((L^{\mathrm{ur}})^\times\bigr)^m\) with
\(\Gamma_W/m\Gamma_W\). Put
\(L_m:=L^{\mathrm{ur}}\bigl(\sqrt[m]{(L^{\mathrm{ur}})^\times}\bigr)\).
	The Kummer pairing gives an isomorphism
	\[
	\operatorname{Gal}(L_m/L^{\mathrm{ur}})
	\simeq
	\operatorname{Hom}\left(
	(L^{\mathrm{ur}})^\times/\bigl((L^{\mathrm{ur}})^\times\bigr)^m,
	\mu_m\right);
	\]
	see \cite[\S3]{Wadsworth83pHenselian}. For completeness, this is obtained
	by applying finite Kummer duality to the extensions generated by finitely
	many classes in
	\((L^{\mathrm{ur}})^\times/((L^{\mathrm{ur}})^\times)^m\) and taking the
	inverse limit over these finite subsets. This gives the full character
	group with the topology of pointwise convergence.

	Since \(L^t=\bigcup_m L_m\), where \(m\) ranges over the integers
	prime to \(p\), and \(L_m\subseteq L_{m'}\) whenever \(m\mid m'\),
	taking Galois groups gives
	\(\operatorname{Gal}(L^t/L^{\mathrm{ur}})
	\simeq\varprojlim_m\operatorname{Gal}(L_m/L^{\mathrm{ur}})\).
	Under the Kummer identifications, for \(m\mid m'\) the restriction map
	\(\operatorname{Gal}(L_{m'}/L^{\mathrm{ur}})
	\to\operatorname{Gal}(L_m/L^{\mathrm{ur}})\)
	is exactly the transition map specified in part \textup{(ii)}.

		(iii) Let \(\ell\ne p\) be a prime, with every prime \(\ell\) allowed
	when \(p=0\), and let \(F/L^t\) be a finite subextension of
	\(L^{\mathrm{sep}}/L^t\), of degree \(d\). Since
	\(L^{\mathrm{ur}}\subseteq F\), the residue field of \(F\) contains
	\(k^{\mathrm{sep}}\). By
	condition~\ref{it:residue-field-algebraic} it is algebraic over \(k\),
	hence algebraic, and therefore purely inseparable, over
	\(k^{\mathrm{sep}}\). In particular it is separably closed.
	
	The field \(F\), endowed with the unique prolongation of the valuation,
	is henselian by condition~\ref{it:henselian-unique-prolongation}.
	Since \(\ell\ne p\), the integer \(\ell\) is invertible in its valuation
	ring. Applying condition~\ref{it:hensel-lifting} to this henselian
	valuation ring shows that its group of units is \(\ell\)-divisible and
	that \(\mu_\ell\subset F\).
	
	By part \textup{(ii)}, \(\Gamma_{L^t}=\Gamma_W'\). Hence
	\((\ast)\), applied to \(F/L^t\), shows that
	\(d\Gamma_F\subseteq\Gamma_W'\). Moreover,
	\(\Gamma_F/\Gamma_W'\) has no \(\ell\)-torsion. Indeed, if
	\(\gamma\in\Gamma_F\) satisfies
	\(\ell\gamma\in\Gamma_W'\), then
	\(\Gamma_W'=\ell\Gamma_W'\) gives
	\(\ell\gamma=\ell\delta\) for some \(\delta\in\Gamma_W'\).
	Since \(\Gamma_F\) is torsion free by
	condition~\ref{it:value-group-torsion-free}, one has
	\(\gamma=\delta\in\Gamma_W'\).
	
	Write \(d=\ell^a d_0\) with \(a\ge0\) and
	\(\gcd(d_0,\ell)=1\). Since multiplication by \(\ell\) is injective
	on \(\Gamma_F/\Gamma_W'\) and this group is killed by \(d\), it is
	killed by \(d_0\). Choose integers \(u,v\) such that
	\(u\ell+vd_0=1\). For every \(\gamma\in\Gamma_F\), one has
	\(d_0\gamma\in\Gamma_W'=\ell\Gamma_W'\), say
	\(d_0\gamma=\ell\delta\) with \(\delta\in\Gamma_W'\). Therefore
	\(
	\gamma
	=
	u\ell\gamma+vd_0\gamma
	=
	\ell(u\gamma+v\delta).
	\)
	Thus \(\Gamma_F=\ell\Gamma_F\).
	
	The valuation exact sequence
	\[
	1\longrightarrow\mathcal O_F^\times
	\longrightarrow F^\times
	\longrightarrow\Gamma_F
	\longrightarrow0
	\]
	and the \(\ell\)-divisibility of both
	\(\mathcal O_F^\times\) and \(\Gamma_F\) imply
	\(F^\times=(F^\times)^\ell\). Since \(\mu_\ell\subset F\), Kummer
	theory gives \(H^1(F,\mu_\ell)=0\). Consequently \(F\) has no cyclic
	extension of degree \(\ell\).
	
	It remains to determine
	\(\operatorname{Gal}(L^{\mathrm{sep}}/L^t)\). If \(p>0\) and this
	group were not pro-\(p\), or if \(p=0\) and it were nontrivial, there
	would exist an open normal subgroup
	\[
	N\trianglelefteq\operatorname{Gal}(L^{\mathrm{sep}}/L^t)
	\]
	such that the finite quotient has order divisible by a prime
	\(\ell\ne p\), with arbitrary \(\ell\) allowed when \(p=0\).
	By Cauchy's theorem this quotient contains a subgroup of order
	\(\ell\). Let \(U\) be its inverse image. Then \(U\) is open,
	\(N\subseteq U\), and \(U/N\simeq\mathbf Z/\ell\). Hence
	\(
	(L^{\mathrm{sep}})^N/(L^{\mathrm{sep}})^U
	\)
	is a cyclic extension of degree \(\ell\), while
	\((L^{\mathrm{sep}})^U/L^t\) is finite. This contradicts the preceding
	paragraph. Therefore
	\(\operatorname{Gal}(L^{\mathrm{sep}}/L^t)\) is pro-\(p\) when
	\(p>0\), and is trivial when \(p=0\).

\end{proof}

We shall use the following multiplicative refinement of the
Hochschild--Serre spectral sequence. The underlying continuous-cohomology
spectral sequence is \cite[(2.4.1)]{NeukirchSchmidtWingberg}; the
multiplicative structure, filtration convention, and edge-map
normalizations needed in \Cref{thm:tame-cohomology-package-level-nu} are included in the proof.

\begin{lemma}[multiplicative Hochschild--Serre spectral sequence]
	\label{lem:multiplicative-hochschild-serre}
	Let \(G\) be a profinite group, let \(I\subseteq G\) be a closed normal subgroup, and let
	\(R\) be a discrete \(G\)-module carrying a \(G\)-equivariant commutative ring structure.
	There is a first quadrant spectral sequence
	\[
	E_2^{s,r}=H^s\bigl(G/I,H^r(I,R)\bigr)
	\Longrightarrow
	H^{s+r}(G,R),
	\qquad
	d_m\colon E_m^{s,r}\longrightarrow E_m^{s+m,r-m+1},
	\]
	with the following properties.
	\begin{enumerate}[label=\textup{(\roman*)},leftmargin=*]
		\item For every \(m\ge2\) and all \(s,r,s',r'\ge0\), there are products
		\[
		E_m^{s,r}\otimes E_m^{s',r'}\longrightarrow E_m^{s+s',r+r'}.
		\]
		They induce the products on the next page, and
		\[
		d_m(xy)=d_m(x)y+(-1)^{s+r}x\,d_m(y)
		\qquad\bigl(x\in E_m^{s,r}\bigr).
		\]
		On the \(E_2\)-page this is the Hochschild--Serre cup product.

		\item For every \(a\ge0\), the abutment carries a finite increasing filtration
		\(0=F_{-1}\subseteq F_0\subseteq\dots\subseteq F_a=H^a(G,R)\) with
		\(F_r/F_{r-1}=E_\infty^{a-r,r}\). It satisfies \(F_r\cdot F_{r'}\subseteq F_{r+r'}\)
		and induces on the associated graded the product of \(E_\infty\).

		\item \(F_0H^a(G,R)\) is the image of the inflation map
		\(H^a(G/I,R^I)\to H^a(G,R)\), and the composite of the projection
		\(H^a(G,R)\to F_a/F_{a-1}=E_\infty^{0,a}\) with the inclusions
		\(E_\infty^{0,a}\subseteq E_2^{0,a}=H^0(G/I,H^a(I,R))\subseteq H^a(I,R)\)
		is restriction along \(I\subseteq G\).
	\end{enumerate}
\end{lemma}
\begin{proof}
	Work in the unbounded derived \(\infty\)-categories of discrete modules,
	equipped with the standard \(t\)-structure and the derived tensor product
	with diagonal action. Inflation from \(G/I\) to \(G\) is exact symmetric
	monoidal, and its right adjoint is \(R\Gamma(I,-)\), with its natural
	\(G/I\)-action. Likewise, inflation from the trivial group to \(G/I\) and
	to \(G\) has right adjoints \(R\Gamma(G/I,-)\) and \(R\Gamma(G,-)\),
	respectively. Since inflation from the trivial group to \(G\) is the
	composite of the first two inflation functors, composition of adjunctions
	gives a canonical equivalence
	\[
	R\Gamma(G,-)
	\simeq
	R\Gamma\bigl(G/I,R\Gamma(I,-)\bigr).
	\]
	Right adjoints of symmetric monoidal functors are lax symmetric monoidal;
	hence this is an equivalence of lax symmetric monoidal functors. In
	particular, \(R\Gamma(I,R)\) is a commutative algebra object in the
	derived category of discrete \(G/I\)-modules, and
	\[
	R\Gamma(G,R)
	\simeq
	R\Gamma\bigl(G/I,R\Gamma(I,R)\bigr)
	\]
	as commutative algebra objects. Its cohomology objects are
	\(H^r(I,R)\), and the induced multiplication on \(H^\ast(I,R)\) is the
	usual continuous cup product; this is the derived-invariants construction
	underlying \cite[(2.4.1)]{NeukirchSchmidtWingberg}.
	
	For every \(r\ge0\), the Postnikov filtration of \(R\Gamma(I,R)\) has
	cofiber sequence
	\[
	\tau^{\le r-1}R\Gamma(I,R)
	\longrightarrow
	\tau^{\le r}R\Gamma(I,R)
	\longrightarrow
	H^r(I,R)[-r].
	\]
	It is multiplicative. Indeed, compatibility of the standard
	\(t\)-structure with the derived tensor product gives
	\[
	\tau^{\le r}R\Gamma(I,R)\otimes^{\mathbf L}
	\tau^{\le r'}R\Gamma(I,R)
	\longrightarrow
	\tau^{\le r+r'}R\Gamma(I,R)
	\]
	for all \(r,r'\ge0\).

	Applying the exact lax symmetric monoidal functor
	\(R\Gamma(G/I,-)\) therefore gives a multiplicative filtered object
	with underlying object \(R\Gamma(G,R)\) and successive cofibers
	\[
	R\Gamma\bigl(G/I,H^r(I,R)\bigr)[-r].
	\]
	
	For \(a\ge0\) and \(0\le r\le a\), let
	\[
	F_rH^a(G,R):=
	\operatorname{im}\!\left(
	H^aR\Gamma\bigl(G/I,\tau^{\le r}R\Gamma(I,R)\bigr)
	\longrightarrow H^a(G,R)
	\right),
	\qquad
	F_{-1}H^a(G,R):=0.
	\]
	The functor \(R\Gamma(G/I,-)\) is left \(t\)-exact. Hence
	\(R\Gamma(G/I,\tau^{\ge r+1}R\Gamma(I,R))\) is concentrated in
	degrees at least \(r+1\). From the truncation cofiber sequence it
	follows that
	\[
	H^aR\Gamma\bigl(G/I,\tau^{\le r}R\Gamma(I,R)\bigr)
	\xrightarrow{\ \sim\ }
	H^a(G,R)
	\qquad(r\ge a).
	\]
	Thus \(F_aH^a(G,R)=H^a(G,R)\), and the filtration of each abutment
	group is finite.
	
	The spectral sequence associated with this filtration, with the usual
	Hochschild--Serre numbering, has
	\[
	E_2^{s,r}
	=
	H^s\bigl(G/I,H^r(I,R)\bigr),
	\qquad
	d_m\colon E_m^{s,r}\longrightarrow E_m^{s+m,r-m+1}.
	\]
	Indeed, the cohomology in total degree \(s+r\) of the \(r\)-th
	successive cofiber is
	\(H^s(G/I,H^r(I,R))\). The preceding finiteness in each total degree
	gives convergence and
	\[
	F_rH^a(G,R)/F_{r-1}H^a(G,R)
	\simeq E_\infty^{a-r,r}.
	\]
	
	The exact couple is multiplicative by the pairings of filtered objects
	constructed above. Hence each page is multiplicative, the product
	passes to the next page, and
	\[
	d_m(xy)=d_m(x)y+(-1)^{s+r}x\,d_m(y)
	\qquad
	\bigl(x\in E_m^{s,r}\bigr);
	\]
	compare \cite[5.4.8]{Weibel94}. The \(E_2\)-product is the
	Hochschild--Serre cup product: on the inner cohomology it is the cup
	product on \(H^\ast(I,R)\), and applying \(R\Gamma(G/I,-)\) gives the
	cup product in \(G/I\)-cohomology. Multiplicativity of the filtration
	gives \(F_rF_{r'}\subseteq F_{r+r'}\), and the induced product on its
	associated graded is the product on \(E_\infty\). This proves
	\textup{(i)} and \textup{(ii)}.
	
	For \textup{(iii)}, one has
	\(\tau^{\le0}R\Gamma(I,R)\simeq R^I\). Under the transitivity
	equivalence for derived invariants, the map
	\[
	R\Gamma(G/I,R^I)
	\longrightarrow
	R\Gamma\bigl(G/I,R\Gamma(I,R)\bigr)
	\simeq R\Gamma(G,R)
	\]
	is inflation. Hence \(F_0H^a(G,R)\) is the image of
	\(H^a(G/I,R^I)\to H^a(G,R)\).
	
	It remains to identify the other edge map. Since
	\(R\Gamma(G/I,-)\) is left \(t\)-exact, the preceding argument for
	\(r=a\) gives a canonical isomorphism
	\[
	H^aR\Gamma\bigl(G/I,\tau^{\le a}R\Gamma(I,R)\bigr)
	\xrightarrow{\ \sim\ } H^a(G,R).
	\]
	After forgetting the \(G/I\)-action, the counits of the
	inflation--derived-invariants adjunction give the commutative square
	\[
	\begin{tikzcd}[column sep=large]
		R\Gamma\bigl(G/I,\tau^{\le a}R\Gamma(I,R)\bigr)
		\ar[r] \ar[d] &
		R\Gamma\bigl(G/I,H^a(I,R)[-a]\bigr)
		\ar[d] \\
		\tau^{\le a}R\Gamma(I,R)
		\ar[r] &
		H^a(I,R)[-a].
	\end{tikzcd}
	\]
	The horizontal arrows are induced by the canonical quotient
	\[
	\tau^{\le a}R\Gamma(I,R)
	\longrightarrow H^a(I,R)[-a].
	\]
	On \(H^a\), the lower horizontal map is the identity of
	\(H^a(I,R)\), while the right vertical map is the inclusion
	\[
	H^0\bigl(G/I,H^a(I,R)\bigr)
	\hookrightarrow H^a(I,R).
	\]
	Under the preceding identification of the upper-left term with
	\(H^a(G,R)\), the left vertical map is restriction along
	\(I\subseteq G\). The upper horizontal map is the Hochschild--Serre
	edge morphism
	\[
	H^a(G,R)\longrightarrow E_2^{0,a}.
	\]
	Its kernel is \(F_{a-1}H^a(G,R)\), and its image is
	\(E_\infty^{0,a}\). Since no differential enters the column \(s=0\),
	\(E_\infty^{0,a}\) is canonically a subgroup of
	\(E_2^{0,a}\). The commutative square therefore identifies
	restriction with the composite
	\[
	H^a(G,R)\longrightarrow
	F_a/F_{a-1}=E_\infty^{0,a}
	\hookrightarrow
	E_2^{0,a}
	\hookrightarrow
	H^a(I,R),
	\]
	as asserted.
\end{proof}

\begin{lemma}[inertia cohomology of a henselian valuation ring at level \(\nu\)]
	\label{lem:henselian-tame-inertia-cohomology}
	Let \(W\) be a henselian valuation ring with fraction field \(L\), residue field \(k\) and
	value group \(\Gamma_W\). Let \(\ell\) be a prime invertible in \(k\), let \(\nu\ge1\), put
	\(N:=\ell^\nu\), and let \(I:=\operatorname{Gal}(L^{\mathrm{sep}}/L^{\mathrm{ur}})\) be the
	inertia group of \Cref{lem:henselian-valuation-strict-henselisation-wild-inertia}. Then:
	\begin{enumerate}[label=\textup{(\roman*)},leftmargin=*]
		\item \(\Gamma_W/N\Gamma_W\) is a free \(\mathbf Z/N\)-module, and any family in
		\(\Gamma_W\) lifting an \(\mathbf F_\ell\)-basis of \(\Gamma_W/\ell\Gamma_W\) reduces to a
		\(\mathbf Z/N\)-basis of it;

		\item the valuation and Kummer theory induce an isomorphism
	\(H^1\bigl(I,\mu_N\bigr)\simeq\Gamma_W/N\Gamma_W\)
		of \(\operatorname{Gal}(k^{\mathrm{sep}}/k)\)-modules, the target carrying the trivial
		action, under which the restriction to \(I\) of the Kummer class of \(\pi\in L^\times\)
		corresponds to \(\overline{v(\pi)}\);

		\item for all \(r\ge0\) and \(q\in\mathbf Z\), cup product induces an isomorphism of
		\(\operatorname{Gal}(k^{\mathrm{sep}}/k)\)-modules
	\[
		\Lambda^r_{\mathbf Z/N}\bigl(\Gamma_W/N\Gamma_W\bigr)
		\otimes_{\mathbf Z/N}\mu_N^{\otimes(q-r)}
		\xrightarrow{\ \sim\ }
		H^r\bigl(I,\mu_N^{\otimes q}\bigr).
	\]
	\end{enumerate}
\end{lemma}
\begin{proof}
	(i) By condition~\ref{it:value-group-torsion-free}, the group \(\Gamma_W\) is flat over
	\(\mathbf Z\); base change shows that \(\Gamma_W/N\Gamma_W\) is flat over
	\(\mathbf Z/N\), hence free by \cite[\stackstag{051G}]{StacksProject}.
	Let \(\{\gamma_b\}_{b\in B}\subset\Gamma_W\) lift an \(\mathbf F_\ell\)-basis of
	\(\Gamma_W/\ell\Gamma_W\), and let \(e_b\) be the class of \(\gamma_b\) in
	\(\Gamma_W/N\Gamma_W\). Their images generate modulo \(\ell\), so iterating the equality
	\(\Gamma_W/N\Gamma_W
	=\sum_{b\in B}(\mathbf Z/N)e_b+\ell(\Gamma_W/N\Gamma_W)\)
	shows that the \(e_b\) generate, because \(\ell^\nu(\Gamma_W/N\Gamma_W)=0\).
	Suppose that
	\(\sum_{b\in B_0}\overline c_be_b=0\) is a nonzero relation, where
	\(B_0\subseteq B\) is finite. Choose integer lifts \(c_b\), put
	\(j:=\min_{b\in B_0}v_\ell(c_b)<\nu\), and write \(c_b=\ell^jd_b\), so that at
	least one \(d_b\) is not divisible by \(\ell\). The relation gives
	\(\ell^j\sum_{b\in B_0}d_b\gamma_b\in\ell^\nu\Gamma_W\).
	Since \(\Gamma_W\) is torsion free, cancellation of \(\ell^j\) gives
	\(\sum_{b\in B_0}d_b\gamma_b\in\ell^{\nu-j}\Gamma_W\subseteq\ell\Gamma_W\).
	Reduction modulo \(\ell\) is a nontrivial relation among the chosen basis elements, a
	contradiction.

	(ii) By
	\Cref{lem:henselian-valuation-strict-henselisation-wild-inertia}, \(W^{\mathrm{sh}}\) is a
	henselian valuation ring with value group \(\Gamma_W\) and separably closed residue field.
	Since \(N\) is invertible there, condition~\ref{it:hensel-lifting} gives
	\(\mu_N\subset L^{\mathrm{ur}}\) and makes \((W^{\mathrm{sh}})^\times\) an \(N\)-divisible
	group. The exact sequence
	\[
	1\longrightarrow(W^{\mathrm{sh}})^\times
	\longrightarrow(L^{\mathrm{ur}})^\times
	\xrightarrow{\ v\ }\Gamma_W\longrightarrow0
	\]
	therefore identifies
	\((L^{\mathrm{ur}})^\times/\bigl((L^{\mathrm{ur}})^\times\bigr)^N\) with
	\(\Gamma_W/N\Gamma_W\), and Kummer theory over \(L^{\mathrm{ur}}\) identifies the former with
	\(H^1(I,\mu_N)\). The valuation is invariant under
	\(\operatorname{Gal}(L^{\mathrm{ur}}/L)\), whence the equivariance and the triviality of the
	action.

	(iii) Put \(p:=\operatorname{char}k\). Apply
	\Cref{lem:multiplicative-hochschild-serre}, with coefficients \(\mathbf Z/N\), to
	\[
	1\longrightarrow\operatorname{Gal}(L^{\mathrm{sep}}/L^t)
	\longrightarrow I
	\longrightarrow\operatorname{Gal}(L^t/L^{\mathrm{ur}})
	\longrightarrow1.
	\]
	By \Cref{lem:henselian-valuation-strict-henselisation-wild-inertia}\textup{(iii)},
	\(\operatorname{Gal}(L^{\mathrm{sep}}/L^t)\) is pro-\(p\) when \(p>0\), and is
	trivial when \(p=0\). Since \(\ell\ne p\), its higher cohomology with
	\(\mathbf Z/N\)-coefficients vanishes
	\cite[Cor.~(3.3.7)]{NeukirchSchmidtWingberg}, and its invariants are
	\(\mathbf Z/N\). The spectral sequence therefore has only its degree-zero row, and
	its inflation edge map gives an isomorphism
	\[
	H^\ast\bigl(\operatorname{Gal}(L^t/L^{\mathrm{ur}}),\mathbf Z/N\bigr)
	\xrightarrow{\ \sim\ }H^\ast(I,\mathbf Z/N).
	\]
	Since \(\mu_N\subset L^{\mathrm{ur}}\), the group \(I\) acts trivially on
	\(\mu_N^{\otimes q}\), so there is a natural \(G_k\)-equivariant isomorphism
	\(H^\ast(I,\mu_N^{\otimes q})\simeq
	H^\ast(I,\mathbf Z/N)\otimes\mu_N^{\otimes q}\).

	Choose the family \(\{\gamma_b\}_{b\in B}\) of part \textup{(i)}. For every \(n\ge1\),
	the argument of part \textup{(i)}, applied with \(n\) in place of \(\nu\), gives
	\(\Gamma_W/\ell^n\Gamma_W \simeq \bigoplus_{b\in B}\mathbf Z/\ell^n\).
	By \Cref{lem:henselian-valuation-strict-henselisation-wild-inertia}(ii), the pro-\(\ell\)
	factor of the tame inertia group is therefore
	\[
	G_\ell
	:=
	\varprojlim_n
	\operatorname{Hom}\bigl(
	\Gamma_W/\ell^n\Gamma_W,\mu_{\ell^n}
	\bigr)
	\simeq
	\prod_{b\in B}\mathbf Z_\ell(1)
	\]
	as topological \(\mathbf Z_\ell\)-modules with continuous \(G_k\)-action, where
	\(\mathbf Z_\ell(1):=\varprojlim_n\mu_{\ell^n}\). The group \(G_k\) acts
	trivially on \(\Gamma_W\) and cyclotomically on the roots of unity, so this
	isomorphism is \(G_k\)-equivariant. After forgetting the \(G_k\)-action, each
	factor \(\mathbf Z_\ell(1)\) is noncanonically isomorphic to \(\mathbf Z_\ell\).

	The Chinese remainder theorem in the inverse-limit description of
	\Cref{lem:henselian-valuation-strict-henselisation-wild-inertia}(ii)
	gives a product decomposition
	\[
	\operatorname{Gal}(L^t/L^{\mathrm{ur}})
	\simeq\prod_{\substack{\lambda\ \text{ prime}\\\lambda\ne p}}G_\lambda,
	\qquad
	G_\lambda:=
	\varprojlim_n
	\operatorname{Hom}\bigl(
	\Gamma_W/\lambda^n\Gamma_W,\mu_{\lambda^n}
	\bigr).
	\]
	Apply Hochschild--Serre to
	\[
	1\longrightarrow
\prod_{\substack{\lambda\ \text{ prime}\\\lambda\ne p,\ell}}G_\lambda
	\longrightarrow\operatorname{Gal}(L^t/L^{\mathrm{ur}})
	\longrightarrow G_\ell\longrightarrow1.
	\]
	The kernel is pro-prime-to-\(\ell\), so its higher cohomology with
	\(\ell\)-primary coefficients vanishes
	\cite[Cor.~(3.3.7)]{NeukirchSchmidtWingberg}, and its invariants on
	\(\mathbf Z/N\) are \(\mathbf Z/N\). Thus inflation identifies
	\(H^*(G_\ell,\mathbf Z/N)\) with
	\(H^*(\operatorname{Gal}(L^t/L^{\mathrm{ur}}),\mathbf Z/N)\).
	By compactness, a continuous cochain on \(G_\ell\) with values in
	the finite discrete module \(\mathbf Z/N\) is constant on cosets of
	some open normal subgroup in each variable. Every such subgroup
	contains a basic open subgroup
	\[
	\prod_{b\notin B_0}\mathbf Z_\ell(1)
	\times\prod_{b\in B_0}\ell^n\mathbf Z_\ell(1)
	\]
	for a finite subset \(B_0\subseteq B\) and some \(n\ge1\). Hence the
	cochain is inflated from a continuous cochain on the finite-coordinate
	product \(\prod_{b\in B_0}\mathbf Z_\ell(1)\). The normalized
	continuous cochain complex is therefore the filtered colimit of the
	corresponding finite-coordinate cochain complexes. Filtered colimits
	of \(\mathbf Z/N\)-modules are exact
		\cite[\stackstag{00DB}]{StacksProject}, whence
	\[
		H^\ast(G_\ell,\mathbf Z/N)
		\simeq
		\operatorname*{colim}_{B_0\subseteq B\ \mathrm{finite}}
		H^\ast\!\left(
		\prod_{b\in B_0}\mathbf Z_\ell(1),\mathbf Z/N
		\right).
	\]
		For a finite subset \(B_0\subseteq B\), put
	\(G_{B_0}:=\prod_{b\in B_0}\mathbf Z_\ell(1)\). We prove by induction on
	\(|B_0|\) that cup product induces an isomorphism of graded
	\(\mathbf Z/N\)-algebras
	\[
	\Lambda_{\mathbf Z/N}^\ast H^1(G_{B_0},\mathbf Z/N)
	\xrightarrow{\ \sim\ }
	H^\ast(G_{B_0},\mathbf Z/N).
	\]
	The case \(B_0=\varnothing\) is immediate. Suppose \(B_0\ne\varnothing\),
	choose \(b_0\in B_0\), and consider the direct-product extension
	\[
	1\longrightarrow\mathbf Z_\ell(1)
	\longrightarrow G_{B_0}
	\longrightarrow G_{B_0\setminus\{b_0\}}
	\longrightarrow1.
	\]
	The underlying topological group of \(\mathbf Z_\ell(1)\) is
	noncanonically \(\mathbf Z_\ell\). Since \(N=\ell^\nu\), for the
	trivial \(\mathbf Z_\ell(1)\)-module \(\mathbf Z/N\) one has
	\(H^0(\mathbf Z_\ell(1),\mathbf Z/N)\simeq\mathbf Z/N\),
	\(H^1(\mathbf Z_\ell(1),\mathbf Z/N)\simeq\mathbf Z/N\), and
	\(H^j(\mathbf Z_\ell(1),\mathbf Z/N)=0\) for every \(j\ge2\)
	\cite[(1.7.7)]{NeukirchSchmidtWingberg}. The quotient
	\(G_{B_0\setminus\{b_0\}}\) acts trivially on these groups because the
	extension is a direct product.
	
	Apply \Cref{lem:multiplicative-hochschild-serre}. Its \(E_2\)-page has
	only rows \(0\) and \(1\). Choose a generator of
	\(H^1(\mathbf Z_\ell(1),\mathbf Z/N)\), and let
	\(u_{b_0}\in H^1(G_{B_0},\mathbf Z/N)\) be its pullback along the
	coordinate projection
	\(G_{B_0}\to\mathbf Z_\ell(1)\). Its restriction to
	\(\mathbf Z_\ell(1)\) is the chosen generator. Hence
	\Cref{lem:multiplicative-hochschild-serre}\textup{(iii)} shows that
	this generator belongs to \(E_\infty^{0,1}\), and therefore is a
	permanent cycle.
	
	Every differential with source in the row \(0\) vanishes for bidegree
	reasons. Moreover, multiplication by the restriction of \(u_{b_0}\)
	identifies, for every \(s\ge0\), the row-zero term
	\(H^s(G_{B_0\setminus\{b_0\}},\mathbf Z/N)\) with
	\(
	H^s\bigl(
	G_{B_0\setminus\{b_0\}},
	H^1(\mathbf Z_\ell(1),\mathbf Z/N)
	\bigr),
	\)
	because \(H^1(\mathbf Z_\ell(1),\mathbf Z/N)\) is a trivial free
	rank-one \(\mathbf Z/N\)-module. Thus the \(E_2\)-page is generated as
	an algebra by its row \(0\) and the class in \(E_2^{0,1}\) represented
	by \(u_{b_0}\). By
	\Cref{lem:multiplicative-hochschild-serre}\textup{(i)}, \(d_2\) is a
	derivation, so \(d_2=0\). Since there are only two rows, no
	differential \(d_m\) with \(m\ge3\) can have a nonzero source and
	target. Consequently \(E_2=E_\infty\).
	
	For every \(j\ge0\), consider
	\[
	\begin{aligned}
		H^j(G_{B_0\setminus\{b_0\}},\mathbf Z/N)
		\oplus H^{j-1}(G_{B_0\setminus\{b_0\}},\mathbf Z/N)
		&\longrightarrow H^j(G_{B_0},\mathbf Z/N),\\
		(\alpha,\beta)&\longmapsto
		\operatorname{inf}(\alpha)
		+\operatorname{inf}(\beta)\cup u_{b_0},
	\end{aligned}
	\]
	with the second summand understood to be zero
	when \(j=0\).
	This is an isomorphism. Indeed,
	\Cref{lem:multiplicative-hochschild-serre}\textup{(ii),(iii)} identifies
	the two associated graded pieces of the Hochschild--Serre filtration
	with \(E_\infty^{j,0}=E_2^{j,0}\) and
	\(E_\infty^{j-1,1}=E_2^{j-1,1}\). On the first graded piece the
	displayed map is the inflation edge map, hence the identity under this
	identification. On the second graded piece it is multiplication by
	the restriction of \(u_{b_0}\), which is an isomorphism because that
	restriction generates
	\(H^1(\mathbf Z_\ell(1),\mathbf Z/N)\). Since the filtration has only
	these two steps, the displayed homomorphism is an isomorphism.
	
	In particular,
	\(H^1(G_{B_0},\mathbf Z/N)\) is the direct sum of the inflated copy of
	\(H^1(G_{B_0\setminus\{b_0\}},\mathbf Z/N)\) and
	\((\mathbf Z/N)u_{b_0}\). Moreover \(u_{b_0}^2=0\), because
	\(u_{b_0}\) is pulled back from \(\mathbf Z_\ell(1)\) and
	\(H^2(\mathbf Z_\ell(1),\mathbf Z/N)=0\). By the induction
	hypothesis, every element of
	\(H^1(G_{B_0\setminus\{b_0\}},\mathbf Z/N)\) has square zero.
	Graded commutativity then shows that every element of
	\(H^1(G_{B_0},\mathbf Z/N)\) has square zero. This
	argument does not divide by \(2\), so it also applies when \(\ell=2\).

	Since every element of \(H^1(G_{B_0},\mathbf Z/N)\) has square zero,
	cup product factors uniquely through
	\(\Lambda^\ast_{\mathbf Z/N}H^1(G_{B_0},\mathbf Z/N)\).
	For every \(j\ge0\), the decomposition above and the induction
	hypothesis identify its degree-\(j\) component with the composite
	\[
	\begin{aligned}
		\Lambda^j_{\mathbf Z/N}H^1(G_{B_0},\mathbf Z/N)
		&\simeq
		\Lambda^j_{\mathbf Z/N}H^1(G_{B_0\setminus\{b_0\}},\mathbf Z/N)
		\oplus
		\Lambda^{j-1}_{\mathbf Z/N}H^1(G_{B_0\setminus\{b_0\}},\mathbf Z/N)
		\wedge u_{b_0}\\
		&\xrightarrow{\ \sim\ }
		H^j(G_{B_0\setminus\{b_0\}},\mathbf Z/N)
		\oplus
		H^{j-1}(G_{B_0\setminus\{b_0\}},\mathbf Z/N)\\
		&\xrightarrow{\ \sim\ }
		H^j(G_{B_0},\mathbf Z/N),
	\end{aligned}
	\]
	where the second summand is omitted for \(j=0\), and the last
	isomorphism sends \((\alpha,\beta)\) to
	\(\operatorname{inf}(\alpha)+\operatorname{inf}(\beta)\cup u_{b_0}\).
	This proves the induction step.

	The resulting cup-product isomorphisms are natural for the projections
	\(G_{B_1}\to G_{B_0}\) when \(B_0\subseteq B_1\) are finite. For every
	\(r\ge0\), exterior powers commute with filtered colimits. Using the
	computation above therefore gives
	\[
	\begin{aligned}
		\Lambda^r_{\mathbf Z/N}H^1(G_\ell,\mathbf Z/N)
		&\simeq
		\operatorname*{colim}_{B_0\subseteq B\ \mathrm{finite}}
		\Lambda^r_{\mathbf Z/N}H^1(G_{B_0},\mathbf Z/N)\\
		&\xrightarrow{\ \sim\ }
		\operatorname*{colim}_{B_0\subseteq B\ \mathrm{finite}}
		H^r(G_{B_0},\mathbf Z/N)
		\simeq H^r(G_\ell,\mathbf Z/N).
	\end{aligned}
	\]
	The inflation isomorphisms already established from \(G_\ell\) to
	\(\operatorname{Gal}(L^t/L^{\mathrm{ur}})\) and then to \(I\) are
	induced by homomorphisms of profinite groups, and therefore preserve
	cup products. Thus, for every \(r\ge0\), cup product induces an
	isomorphism
	\[
	\Lambda^r_{\mathbf Z/N}H^1(I,\mathbf Z/N)
	\xrightarrow{\ \sim\ }
	H^r(I,\mathbf Z/N).
	\]
	This isomorphism is \(G_k\)-equivariant by naturality of cup products
	for the \(G_k\)-action on inertia cohomology.
	
	Finally, the natural \(G_k\)-equivariant coefficient identification
	established above gives
	\(H^1(I,\mu_N)\simeq H^1(I,\mathbf Z/N)\otimes_{\mathbf Z/N}\mu_N\).
	Part \textup{(ii)} therefore identifies
	\(H^1(I,\mathbf Z/N)\) with
	\((\Gamma_W/N\Gamma_W)\otimes_{\mathbf Z/N}\mu_N^{\otimes-1}\).
	Since \(\mu_N\) is an invertible \(\mathbf Z/N\)-module, for every
	\(r\ge0\) and \(q\in\mathbf Z\) the preceding cup-product isomorphism
	gives the \(G_k\)-equivariant composite
	\[
	\begin{aligned}
		\Lambda^r_{\mathbf Z/N}(\Gamma_W/N\Gamma_W)
		\otimes_{\mathbf Z/N}\mu_N^{\otimes(q-r)}
		&\xrightarrow{\ \sim\ }
		\Lambda^r_{\mathbf Z/N}H^1(I,\mathbf Z/N)
		\otimes_{\mathbf Z/N}\mu_N^{\otimes q}\\
		&\xrightarrow{\ \sim\ }
		H^r(I,\mathbf Z/N)\otimes_{\mathbf Z/N}\mu_N^{\otimes q}
		\xrightarrow{\ \sim\ }
		H^r(I,\mu_N^{\otimes q}),
	\end{aligned}
	\]
	which is precisely the cup-product isomorphism of \textup{(iii)}.

\end{proof}

\begin{theorem}[tame value-degree decomposition]
\label{thm:tame-cohomology-package-level-nu}
Let \(W\) be a henselian valuation ring with fraction field \(L\),
residue field \(k\), value group \(\Gamma_W\), and valuation
\(v\colon L^\times\to\Gamma_W\). Let \(\ell\) be a prime, let
\(\nu\ge1\), and put \(N:=\ell^\nu\). Assume that \(N\) is invertible
in \(W\). Fix \(a\ge0\) and \(q\in\mathbf Z\). Choose a
\(\mathbf Z/N\)-basis \(B\) of \(\Gamma_W/N\Gamma_W\), fix a total
order on \(B\), and choose classes
\(w_b\in H^1(L,\mu_N)\) whose restrictions to inertia correspond to \(b\).
For a finite subset \(J\subseteq B\), let \(w_J\) be the cup product of the
\(w_b\), taken in the induced order, and put \(w_\varnothing:=1\). Then
\[
\Theta\colon
\bigoplus_{\substack{J\subseteq B\\ |J|\le a}}
H^{a-|J|}(k,\mu_N^{\otimes(q-|J|)})
\longrightarrow H^a(L,\mu_N^{\otimes q}),
\qquad
(\alpha_J)_J\longmapsto
\sum_J\operatorname{inf}(\alpha_J)\cup w_J
\]
is an isomorphism.

For \(0\le r\le a\), define
\[
F_r^{\mathrm{HS}}H^a(L,\mu_N^{\otimes q})
:=
\operatorname{im}\!\left(
H^aR\Gamma\bigl(
G_k,\tau^{\le r}R\Gamma(I,\mu_N^{\otimes q})
\bigr)
\longrightarrow
H^a(L,\mu_N^{\otimes q})
\right),
\]
where the map is induced by
\(\tau^{\le r}R\Gamma(I,\mu_N^{\otimes q})
\to R\Gamma(I,\mu_N^{\otimes q})\), followed by the transitivity
equivalence
\[
R\Gamma\bigl(G_k,R\Gamma(I,\mu_N^{\otimes q})\bigr)
\simeq
R\Gamma(G_L,\mu_N^{\otimes q}),
\]
and
\(H^aR\Gamma(G_L,\mu_N^{\otimes q})
=H^a(L,\mu_N^{\otimes q})\)
by \Cref{conv:galois-cohomology-kummer}.

Put
\(F_{-1}^{\mathrm{HS}}H^a(L,\mu_N^{\otimes q})=0\), and, for
\(0\le r\le a\), put
\[
\operatorname{gr}_r^{\mathrm{HS}}H^a(L,\mu_N^{\otimes q})
:=
F_r^{\mathrm{HS}}H^a(L,\mu_N^{\otimes q})
/
F_{r-1}^{\mathrm{HS}}H^a(L,\mu_N^{\otimes q}).
\]

Equivalently, this is the weight-\(q\) summand of the filtration in
\Cref{lem:multiplicative-hochschild-serre} applied to the graded
\(G_L\)-algebra
\(\bigoplus_{n\in\mathbf Z}\mu_N^{\otimes n}\).

Then, for every \(0\le r\le a\), \(\Theta\) restricts to an
isomorphism
\[
\bigoplus_{\substack{J\subseteq B\\ |J|\le r}}
H^{a-|J|}\bigl(k,\mu_N^{\otimes(q-|J|)}\bigr)
\xrightarrow{\ \sim\ }
F_r^{\mathrm{HS}}H^a(L,\mu_N^{\otimes q}).
\]

Consequently, its image is independent of the basis \(B\), its ordering,
and the chosen classes \(w_b\). Moreover,
\[
\operatorname{gr}_r^{\mathrm{HS}}H^a(L,\mu_N^{\otimes q})
\simeq
\Lambda_{\mathbf Z/N}^r(\Gamma_W/N\Gamma_W)
\otimes_{\mathbf Z/N}
H^{a-r}(k,\mu_N^{\otimes(q-r)})
\qquad(0\le r\le a).
\]

The image under \(\Theta\) of the summand indexed by
\(J=\varnothing\) is
\(F_0^{\mathrm{HS}}H^a(L,\mu_N^{\otimes q})\), and on this summand
\(\Theta\) is inflation. In particular, the inflation map
\[
H^a(k,\mu_N^{\otimes q})
\longrightarrow
H^a(L,\mu_N^{\otimes q})
\]
is injective.
\end{theorem}

\begin{proof}
For each \(b\in B\), choose \(\gamma_b\in\Gamma_W\) lifting \(b\), and
choose \(\varpi_b\in L^\times\) with \(v(\varpi_b)=\gamma_b\). Then
\(\delta_N(\varpi_b)\) restricts to \(b\) by
\Cref{lem:henselian-tame-inertia-cohomology}\textup{(ii)}. Thus the
classes required in the statement exist.
We henceforth work with the arbitrary chosen classes \(w_b\) of the
statement.

Apply \Cref{lem:multiplicative-hochschild-serre} to the graded
\(G_L\)-algebra
\(\bigoplus_{n\in\mathbf Z}\mu_N^{\otimes n}\), with multiplication
induced by the tensor-product identifications of Tate twists, and to
\[
1\longrightarrow I\longrightarrow G_L\longrightarrow G_k
\longrightarrow1.
\]
In total degree \(a\), the filtration on the weight-\(q\) summand of the
abutment is, by definition, the filtration
\(F_\bullet^{\mathrm{HS}}H^a(L,\mu_N^{\otimes q})\) introduced in the
statement. By
\Cref{lem:henselian-tame-inertia-cohomology}\textup{(iii)}, its
weight-\(q\) \(E_2\)-page is
\[
E_2^{s,r}(q)=
\Lambda_{\mathbf Z/N}^r\bigl(\Gamma_W/N\Gamma_W\bigr)
\otimes_{\mathbf Z/N}
H^s\bigl(k,\mu_N^{\otimes(q-r)}\bigr).
\]
Continuous cohomology commutes with the direct sums occurring here.
Indeed, for a profinite group \(Q\), compactness of \(Q^s\) gives
\(C^s_{\mathrm{cont}}(Q,\bigoplus_iA_i)
=\bigoplus_iC^s_{\mathrm{cont}}(Q,A_i)\), because the image of a
continuous cochain is a compact subset of a discrete module, hence
finite, and the union of the supports of its finitely many values is
finite. Direct sums of modules are exact.

We prove by induction on \(m\ge2\) that \(d_m=0\) on the entire
weight-graded spectral sequence. On the row
\(r=0\), the outgoing differential is zero because its target has second
index \(1-m<0\). Each basis element of
\(E_2^{0,1}(1)=\Gamma_W/N\Gamma_W\) is the restriction of a class \(w_b\) on \(G_L\).
By \Cref{lem:multiplicative-hochschild-serre}(iii), it belongs to
\(E_\infty^{0,1}\); hence all its differentials vanish. The isomorphism in
\Cref{lem:henselian-tame-inertia-cohomology}(iii) is given by cup product,
while the \(E_2\)-product is the cup product of inertia cohomology followed
by that of \(G_k\), by
\Cref{lem:multiplicative-hochschild-serre}(i). Hence every element of
\(E_2^{s,r}(q)\) is a sum of products \(b_J\alpha\), where \(b_J\) is a
product of \(r\) elements of \(E_2^{0,1}(1)\) and
\(\alpha\in E_2^{s,0}(q-r)\). Across all weights, the row \(r=0\) and
these degree-one elements therefore generate \(E_2\) as an algebra. Since
\(d_2\) is a derivation, it vanishes on all of \(E_2\). Suppose
inductively that \(d_2,\ldots,d_{m-1}\) vanish. Then \(E_m=E_2\), with
the same generators. The derivation rule gives \(d_m=0\). Thus
\(E_2=E_\infty\).

For \(0\le r\le a\), put
\[
D_r:=
\bigoplus_{\substack{J\subseteq B\\ |J|\le r}}
H^{a-|J|}\bigl(k,\mu_N^{\otimes(q-|J|)}\bigr),
\qquad D_{-1}:=0.
\]
By \Cref{lem:multiplicative-hochschild-serre}\textup{(iii)}, inflation
maps
\(H^d(k,\mu_N^{\otimes t})\) into
\(F_0^{\mathrm{HS}}H^d(L,\mu_N^{\otimes t})\) for all \(d\ge0\) and
\(t\in\mathbf Z\). For every \(b\in B\), the image of \(w_b\) in
\[
\operatorname{gr}_1^{\mathrm{HS}}H^1(L,\mu_N)
=
E_\infty^{0,1}(1)
=
E_2^{0,1}(1)
\]
is its restriction to inertia, hence \(b\in\Gamma_W/N\Gamma_W\), where
the second equality uses \(E_2=E_\infty\) proved above. By
multiplicativity of the Hochschild--Serre filtration, the restriction of
\(\Theta\) to \(D_r\) therefore takes values in
\(F_r^{\mathrm{HS}}H^a(L,\mu_N^{\otimes q})\). Denote this restriction
by \(\Theta_r\).

For \(0\le r\le a\), the induced homomorphism
\[
D_r/D_{r-1}\longrightarrow
\operatorname{gr}_r^{\mathrm{HS}}H^a(L,\mu_N^{\otimes q})
\]
is, up to the Koszul signs determined by the chosen order and the cup
product convention, the homomorphism
\[
(\alpha_J)_{|J|=r}\longmapsto
\sum_{J=\{b_1<\cdots<b_r\}}
(b_1\wedge\cdots\wedge b_r)\otimes\alpha_J.
\]
It is an isomorphism, since the ordered wedges
\(b_1\wedge\cdots\wedge b_r\) form the basis of
\(\Lambda_{\mathbf Z/N}^r(\Gamma_W/N\Gamma_W)\) determined by \(B\).

The homomorphisms \(\Theta_r\) give a morphism between the short exact
sequences
\[
0\longrightarrow D_{r-1}\longrightarrow D_r
\longrightarrow D_r/D_{r-1}\longrightarrow0
\]
and
\[
0\longrightarrow
F_{r-1}^{\mathrm{HS}}H^a(L,\mu_N^{\otimes q})
\longrightarrow
F_r^{\mathrm{HS}}H^a(L,\mu_N^{\otimes q})
\longrightarrow\operatorname{gr}_r^{\mathrm{HS}}
H^a(L,\mu_N^{\otimes q})\longrightarrow0.
\]

For \(r=0\), the induced map on the unique nonzero quotient is an
isomorphism, hence so is \(\Theta_0\). If \(r\ge1\) and
\(\Theta_{r-1}\) is an isomorphism, the preceding isomorphism on
\(D_r/D_{r-1}\) and the short five lemma applied to these two short
exact sequences show that \(\Theta_r\) is an isomorphism.

Taking \(r=a\) proves
that \(\Theta\) is an isomorphism and that its restriction to \(D_r\)
identifies \(D_r\) with
\(F_r^{\mathrm{HS}}H^a(L,\mu_N^{\otimes q})\). The assertions concerning
the empty summand and inflation follow from
\Cref{lem:multiplicative-hochschild-serre}(iii).
\end{proof}

\begin{lemma}[henselian étale specialization is inflation]
	\label{lem:henselian-etale-specialization-is-inflation}
	Let \(W\) be a henselian valuation ring with fraction field \(L\) and residue field \(k\),
	and let \(m\ge1\) be invertible in \(k\). Fix a separable closure \(L^{\sep}/L\), the
	unique prolongation of the valuation to it, and the resulting surjection \(G_L\to G_k\).
	Then for every \(q\in\mathbf Z\) the closed-point pullback
	\(R\Gamma(W_{\et},\mu_m^{\otimes q})\to R\Gamma(k_{\et},\mu_m^{\otimes q})\) is an
	equivalence. After identifying étale cohomology of the two fields with Galois
	cohomology using the compatible fibre functors defined by the chosen
	separably closed fields, the diagram
	\[
	\begin{tikzcd}[column sep=large]
	R\Gamma(W_{\et},\mu_m^{\otimes q}) \ar[r] \ar[d,"\sim"'] &
	R\Gamma(L_{\et},\mu_m^{\otimes q}) \ar[d,equal] \\
	R\Gamma(k_{\et},\mu_m^{\otimes q}) \ar[r,"\operatorname{inf}"'] &
	R\Gamma(L_{\et},\mu_m^{\otimes q})
	\end{tikzcd}
	\]
	commutes.
\end{lemma}
\begin{proof}
	Since \(m\) is invertible in \(k\), it is a unit in the local ring
	\(W\). Hence \(\mu_m^{\otimes q}\) is a finite locally constant torsion
	sheaf on \(\Spec(W)_{\et}\) for every \(q\in\mathbf Z\), with negative
	twists understood as in \Cref{conv:galois-cohomology-kummer}.
	
	The pair \((W,\mathfrak m_W)\) is henselian. Gabber's affine analogue
	of proper base change
	\cite[\stackstag{09ZI}]{StacksProject}, applied to
	\(\mu_m^{\otimes q}\), therefore gives an equivalence
	\[
	R\Gamma(W_{\et},\mu_m^{\otimes q})
	\xrightarrow{\ \sim\ }
	R\Gamma(k_{\et},\mu_m^{\otimes q}).
	\]
	
	It remains to identify generic restriction under this equivalence.
	Write \(i\colon\Spec(k)\to\Spec(W)\) and
	\(j\colon\Spec(L)\to\Spec(W)\). Let \(\mathcal O\) be the valuation
	ring of the unique prolongation of the valuation to \(L^{\sep}\), and
	let \(\kappa\) be its residue field. As in the proof of
	\Cref{lem:henselian-valuation-strict-henselisation-wild-inertia},
	the chosen strict henselization determines an inclusion
	\(k^{\sep}\subseteq\kappa\). By
	condition~\ref{it:residue-field-algebraic}, \(\kappa/k\) is algebraic;
	hence \(\kappa/k^{\sep}\) is purely inseparable. In particular,
	\(\kappa\) is separably closed.
	
	For every finite étale \(W\)-scheme \(X\), properness of
	\(X\to\Spec(W)\) gives a natural bijection
	\(X(\mathcal O)\xrightarrow{\sim}X(L^{\sep})\), while henselianity of
	\(\mathcal O\) gives a natural bijection
	\(X(\mathcal O)\xrightarrow{\sim}X(\kappa)\). These bijections are
	\(G_L\)-equivariant, since the unique prolongation of the valuation is
	\(G_L\)-stable. 
	Using these bijections to identify the two fibre functors, the
	generic morphism \(\Spec(L)\to\Spec(W)\) induces, under the
	identification
	\(\pi_1(W,\Spec(\kappa))\simeq G_k\)
	furnished by the finite étale equivalence \(W\to k\), precisely
	the surjection \(G_L\to G_k\) fixed in the statement.

	Since \(m\) is invertible, reduction induces compatible
	\(G_L\)-equivariant identifications
	\[
	\mu_m^{\otimes q}(k^{\sep})
	\xrightarrow{\ \sim\ }
	\mu_m^{\otimes q}(\kappa)
	\xleftarrow{\ \sim\ }
	\mu_m^{\otimes q}(L^{\sep})
	\]
	for every \(q\in\mathbf Z\).
	For a connected pointed scheme and a finite locally constant sheaf,
	the canonical Cartan--Leray comparison from continuous cohomology of
	the étale fundamental group to étale cohomology is functorial for
	pointed morphisms. For \(\Spec(k)\) and \(\Spec(L)\) it is an
	equivalence by the étale--Galois comparison
	\cite[\stackstag{03QQ} and \stackstag{03QU}]{StacksProject}. It is also
	an equivalence for \(\Spec(W)\): applying functoriality to \(i\), the
	map on fundamental groups is an isomorphism by the finite étale
	equivalence, the comparison for \(k\) is an equivalence, and the map on
	étale cohomology is the Gabber equivalence proved above.
	
	Functoriality for \(j\) therefore gives a commutative square
	\[
	\begin{tikzcd}[column sep=large]
		R\Gamma\bigl(G_k,\mu_m^{\otimes q}(k^{\sep})\bigr)
		\ar[r,"\sim"]
		\ar[d,"\operatorname{inf}"'] &
		R\Gamma(W_{\et},\mu_m^{\otimes q})
		\ar[d,"j^*"] \\
		R\Gamma\bigl(G_L,\mu_m^{\otimes q}(L^{\sep})\bigr)
		\ar[r,"\sim"'] &
		R\Gamma(L_{\et},\mu_m^{\otimes q}).
	\end{tikzcd}
	\]
	The left vertical map is precomposition with \(G_L\twoheadrightarrow
	G_k\), using the preceding identification of coefficient modules, hence
	is inflation in the sense of
	\Cref{conv:galois-cohomology-kummer}. Naturality for \(i\) identifies
	the upper horizontal equivalence with the inverse of closed-point
	pullback after the standard étale--Galois comparison for \(k\).
	Consequently the displayed square is exactly the square asserted in the
	statement.
\end{proof}

\section{Filtered specialization}
\label{sec:filtered-specialization}

Let \(W\) be a henselian valuation ring with fraction field \(L\),
residue field \(k\), and let \(N=\ell^\nu\) be invertible in \(W\).
We first identify closed-point rigidity with étale inflation on the
graded pieces. We then compare multiplication by \(K_1\)-classes with
Kummer cup products.

\begin{lemma}[naturality of the Prüfer comparison]
\label{lem:ordinary-prufer-bl-functoriality}
For every prime power \(N=\ell^\nu\) and every \(q\ge0\), the comparison
maps of
\cite[\arxivhtmlref{2506.09910v2}{S4.Thmtheorem1.E1}{(4.1.1)}]{BouisKunduBL}
form a natural transformation as the qcqs scheme varies.

In particular, let \(W\) be a henselian valuation ring with fraction
field \(L\) and residue field \(k\), and assume that \(N\) is invertible
in \(W\). The equivalences of
\Cref{lem:ordinary-prufer-bl-range-input} are compatible with
\(W\to k\) and \(W\to L\). Explicitly, for \(T\in\{k,L\}\), the square
\[
\begin{tikzcd}[column sep=large]
\tau^{\le q}R\Gamma(W_{\et},\mu_N^{\otimes q})
  \ar[r,"\sim"] \ar[d] &
\mathbf Z/N(q)_{\mot}(W) \ar[d] \\
\tau^{\le q}R\Gamma(T_{\et},\mu_N^{\otimes q})
  \ar[r,"\sim"'] &
\mathbf Z/N(q)_{\mot}(T)
\end{tikzcd}
\]
commutes.
\end{lemma}

\begin{proof}
The comparison map
\cite[\arxivhtmlref{2506.09910v2}{S4.Thmtheorem1.E1}{(4.1.1)}]{BouisKunduBL}
is an absolute natural transformation in the scheme. The Prüfer base and ind-smoothness enter
only in \Cref{lem:ordinary-prufer-bl-range-input}, where they show that
this natural map is an equivalence for each of \(W,k,L\). The
identification of the syntomic twist with the étale Tate twist on the
invertible locus is natural, as is evaluation of a Nisnevich sheaf at a
henselian local scheme. Naturality for the ring maps \(W\to k\) and
\(W\to L\) therefore gives the squares in the statement.
\end{proof}

\begin{lemma}[closed-point rigidity on motivic graded pieces]
	\label{lem:henselian-valuation-motivic-closed-point-rigidity}
	Let \(W\) be a henselian valuation ring with residue field \(k\). Let \(\ell\) be a
	prime invertible in \(k\), let \(\nu\ge1\), and put \(N:=\ell^\nu\). For every \(q\ge0\),
	reduction to the closed point induces an equivalence
	\(\mathbf Z/N(q)_{\mot}(W) \xrightarrow{\ \sim\ } \mathbf Z/N(q)_{\mot}(k)\).
	Under \Cref{lem:ordinary-prufer-bl-functoriality}, this is the
	truncation of the equivalence
	\(R\Gamma(W_{\et},\mu_N^{\otimes q})\to
	R\Gamma(k_{\et},\mu_N^{\otimes q})\).
\end{lemma}

\begin{proof}
	By \Cref{lem:ordinary-prufer-bl-functoriality}, the motivic
	closed-point map is identified with
	\(\tau^{\le q}R\Gamma(W_{\et},\mu_N^{\otimes q})\to
	\tau^{\le q}R\Gamma(k_{\et},\mu_N^{\otimes q})\).
	The untruncated map is an equivalence by
	\Cref{lem:henselian-etale-specialization-is-inflation}, applied with \(m=N\). Truncating and
	transporting through the Prüfer comparisons proves the claim.
\end{proof}

\begin{lemma}[filtered Whitehead criterion]
\label{lem:filtered-whitehead-criterion}
Let \(f\colon E\to E'\) be a morphism of decreasing
\(\mathbf N\)-indexed filtered spectra. Suppose that
\(\operatorname{gr}^q(f)\) is an equivalence for every \(q\ge0\).
Then the following conditions are equivalent:
\begin{enumerate}[label=\textup{(\roman*)},leftmargin=*]
\item \(F^q(f)\) is an equivalence for every \(q\ge0\);
\item \(F^{q_0}(f)\) is an equivalence for some \(q_0\ge0\);
\item the induced map
\[
\varprojlim_qF^qE\longrightarrow\varprojlim_qF^qE'
\]
is an equivalence.
\end{enumerate}
Consequently, \(f\) is an equivalence of filtered spectra if either
\(F^0(f)\) is an equivalence or both \(E\) and \(E'\) are complete.
\end{lemma}

\begin{proof}
Let \(C:=\operatorname{cofib}(f)\) in \(\FilSp\). Evaluation at a
filtration degree is exact. Moreover, \(\operatorname{gr}^q\), being
the cofiber of the natural transformation \(F^{q+1}\to F^q\) between
exact functors, is exact. Hence, for every \(q\ge0\),
\[
F^qC\simeq\operatorname{cofib}\bigl(F^q(f)\bigr),
\qquad
\operatorname{gr}^qC\simeq
\operatorname{cofib}\bigl(\operatorname{gr}^q(f)\bigr).
\]
The hypothesis gives \(\operatorname{gr}^qC\simeq0\) for every
\(q\ge0\). The cofiber sequence
\(F^{q+1}C\to F^qC\to\operatorname{gr}^qC\) therefore shows that
\(F^{q+1}C\to F^qC\) is an equivalence for every \(q\ge0\).

Fix \(q_0\ge0\). Since all transition maps are equivalences,
\(F^{q_0}C\) is contractible if and only if \(F^qC\) is contractible
for every \(q\ge0\). In view of the first equivalence above, this says
precisely that \(F^{q_0}(f)\) is an equivalence if and only if
\(F^q(f)\) is an equivalence for every \(q\ge0\). This proves the
equivalence of \textup{(i)} and \textup{(ii)}.

Limits are exact in spectra. Therefore
\[
\varprojlim_qF^qC\simeq
\operatorname{cofib}\!\left(
\varprojlim_qF^qE\longrightarrow\varprojlim_qF^qE'
\right).
\]
Since every transition map of \(\{F^qC\}_q\) is an equivalence, the
canonical map \(\varprojlim_qF^qC\to F^{q_0}C\) is an equivalence for
every \(q_0\). Consequently, the map of limits in \textup{(iii)} is an
equivalence if and only if one, equivalently every, \(F^q(f)\) is an
equivalence. This proves the equivalence with \textup{(i)}. The two
final criteria are the cases \(q_0=0\) and
\(\varprojlim_qF^qE\simeq\varprojlim_qF^qE'\simeq0\), respectively.
For \(q<0\), both filtrations are extended constantly from degree zero,
so condition \textup{(i)} is precisely the assertion that \(f\) is an
equivalence in \(\FilSp\).
\end{proof}

\begin{proposition}[filtered Gabber rigidity in arbitrary rank]
	\label{lem:filtered-gabber-rigidity-as-motivic-filtered-equivalence}
	Let \(W\) be a henselian valuation ring of arbitrary rank with residue
	field \(k\). Let \(\ell\) be a
	prime invertible in \(k\), let \(\nu\ge1\), and put \(N:=\ell^\nu\). Reduction to the closed
	point induces an equivalence of filtered spectra
	\[
	\rho_{W,N}\colon
	\Fil^\bullet_{\mot}K(W;\mathbf Z/N)
	\xrightarrow{\ \sim\ }
	\Fil^\bullet_{\mot}K(k;\mathbf Z/N),
	\]
	that is, \(\Fil^q\rho_{W,N}\) is an equivalence for every \(q\). Both sides are complete.
	For every \(q\ge0\), the associated graded map is
	\(\mathbf Z/N(q)_{\mot}(W)[2q] \xrightarrow{\ \sim\ } \mathbf Z/N(q)_{\mot}(k)[2q]\).
	After forgetting the filtration, the underlying map is the rigidity equivalence of
	\Cref{lem:valuation-closed-point-finite-coefficient-k-rigidity}. 
\end{proposition}

\begin{proof}
	The filtered spectra, their \(\mathbf N\)-indexedness, and the
	functoriality of \(\rho_{W,N}\) are supplied by
	\Cref{lem:global-motivic-filtration-kmodell-input}. We apply
	\Cref{lem:filtered-whitehead-criterion} to \(\rho_{W,N}\).

	Its graded hypothesis holds: for \(q\ge0\) the map \(\operatorname{gr}^q\rho_{W,N}\) is the
	equivalence of \Cref{lem:henselian-valuation-motivic-closed-point-rigidity}. Its hypothesis
	in filtration degree zero holds too: the natural identifications
	\(\Fil^0_{\mot}K(-;\mathbf Z/N)\simeq K(-;\mathbf Z/N)\) identify
	\(\Fil^0\rho_{W,N}\), naturally in the scheme, with the reduction map
	\(K(W;\mathbf Z/N)\to K(k;\mathbf Z/N)\), and
	that map is the rigidity equivalence of
	\Cref{lem:valuation-closed-point-finite-coefficient-k-rigidity}. So \(\rho_{W,N}\) is an
	equivalence of filtered spectra, and its underlying map is that rigidity equivalence.

	Finally \(\Spec(k)\) has valuative dimension \(0\) by
	\Cref{lem:valuation-inputs-finite-valuative-dimension}, so
	\(\Fil^\bullet_{\mot}K(k;\mathbf Z/N)\) is complete by
	\Cref{lem:global-motivic-filtration-kmodell-input}; and \(\Fil^\bullet_{\mot}K(W;\mathbf
	Z/N)\) is equivalent to it, hence complete as well.
\end{proof}

\begin{remark}[completeness in arbitrary rank]
\label{rem:completeness-from-rigidity}
Bouis's general completeness theorem
\cite[\arxivhtmlref{2412.06635v2}{S4.Thmtheorem49}{Prop.~4.49}]{BouisMixedCharacteristicMotivicCohomology}
assumes
finite valuative dimension.
For a henselian valuation ring of arbitrary rank,
\Cref{lem:filtered-gabber-rigidity-as-motivic-filtered-equivalence}
instead transports completeness from its residue field. This argument
is specific to finite coefficients invertible in the valuation ring.
\end{remark}

\begin{lemma}[rigidity specialization is filtered étale inflation]
	\label{lem:rigidity-specialization-filtered-inflation}
	Let \(W\) be a henselian valuation ring with fraction field \(L\) and residue field
	\(k\). Let \(\ell\) be a prime invertible in \(k\), let \(\nu\ge1\), and put
		\(N:=\ell^\nu\). Let \(j^*_{W,N}\colon
		\Fil^\bullet_{\mot}K(W;\mathbf Z/N)\to
		\Fil^\bullet_{\mot}K(L;\mathbf Z/N)\) denote the generic restriction map.
		Precomposition with the filtered equivalence
		\(\rho_{W,N}\) of
		\Cref{lem:filtered-gabber-rigidity-as-motivic-filtered-equivalence}
		induces an equivalence from the space of filtered maps
		\(\Fil^\bullet_{\mot}K(k;\mathbf Z/N)\to
		\Fil^\bullet_{\mot}K(L;\mathbf Z/N)\) to the space of filtered maps
		\(\Fil^\bullet_{\mot}K(W;\mathbf Z/N)\to
		\Fil^\bullet_{\mot}K(L;\mathbf Z/N)\).
	Define \(s_{W,N}\) by choosing a point in the homotopy fibre over
	\(j^*_{W,N}\). This fibre is contractible. Thus \(s_{W,N}\) is well-defined up to
	contractible choice, together with a homotopy
	\(s_{W,N}\circ\rho_{W,N}\simeq j^*_{W,N}\).

	For every \(q\ge0\), under
	\Cref{lem:ordinary-prufer-bl-functoriality}, the associated graded map
	is the \([2q]\)-fold suspension of
	\(\tau^{\le q}R\Gamma(k_{\et},\mu_N^{\otimes q})\to
	\tau^{\le q}R\Gamma(L_{\et},\mu_N^{\otimes q})\), namely the
	truncation of inflation along the henselian valuation.
\end{lemma}

\begin{proof}
	Precomposition with an equivalence induces an equivalence of mapping
	spaces, with inverse given by precomposition with an inverse
	equivalence. Hence the homotopy fibre in the statement is contractible,
	and by construction
	\(s_{W,N}\circ\rho_{W,N}\simeq j^*_{W,N}\).
	
	Applying \(\operatorname{gr}^q\) to the homotopy
	\(s_{W,N}\circ\rho_{W,N}\simeq j^*_{W,N}\) gives the commutative
	triangle
	\[
	\begin{tikzcd}[column sep=large]
		\operatorname{gr}_{\mot}^qK(W;\mathbf Z/N)
		\ar[r,"\operatorname{gr}^q\rho_{W,N}"]
		\ar[dr,"\operatorname{gr}^qj^*_{W,N}"'] &
		\operatorname{gr}_{\mot}^qK(k;\mathbf Z/N)
		\ar[d,"\operatorname{gr}^qs_{W,N}"]\\
		& \operatorname{gr}_{\mot}^qK(L;\mathbf Z/N).
	\end{tikzcd}
	\]
	By \Cref{lem:ordinary-prufer-bl-functoriality}, after desuspending by
	\(2q\) this triangle is identified with
	\[
	\begin{tikzcd}[column sep=large]
		\tau^{\le q}R\Gamma(W_{\et},\mu_N^{\otimes q})
		\ar[r,"\sim"]
		\ar[dr] &
		\tau^{\le q}R\Gamma(k_{\et},\mu_N^{\otimes q})
		\ar[d]\\
		& \tau^{\le q}R\Gamma(L_{\et},\mu_N^{\otimes q}).
	\end{tikzcd}
	\]
	By
	\Cref{lem:henselian-etale-specialization-is-inflation}, precomposition
	of truncated inflation with the upper horizontal equivalence is the
	diagonal map. Since precomposition with an equivalence induces an
	equivalence of mapping spaces, the right vertical map is the truncation
	of inflation. This proves the assertion.

\end{proof}

\begin{lemma}[integral weight-one unit classes]
\label{lem:weight-one-edge-kummer-compatibility}
Let \(F\) be a field. Then \(\mathbf Z(1)_{\mot}(F)\simeq F^\times[-1]\); in particular
there is a natural identification
\[
\pi_1\operatorname{gr}^1_{\mot}K(F)
=
H^1_{\mot}(F,\mathbf Z(1))
\xrightarrow{\ \sim\ }
F^\times.
\]
Moreover the canonical maps
\[
K_1(F)\longleftarrow\pi_1\Fil^1_{\mot}K(F)\longrightarrow
\pi_1\operatorname{gr}^1_{\mot}K(F)=H^1_{\mot}(F,\mathbf Z(1))
\]
are both isomorphisms. For \(a\in F^\times=K_1(F)\) write
\(\widetilde u(a)\in\pi_1\Fil^1_{\mot}K(F)\) for the preimage of the class of \(a\), and
\(u_F(a)\in H^1_{\mot}(F,\mathbf Z(1))\) for its image. Then
\[
u_F\colon F^\times\xrightarrow{\ \sim\ }
H^1_{\mot}(F,\mathbf Z(1)),\qquad a\longmapsto u_F(a),
\]
is natural in \(F\).
For every \(m\ge1\), right multiplication by \(\widetilde u(a)\) is a
map of filtered spectra
\[
\Sigma\,\Fil^{\bullet}_{\mot}K(F;\mathbf Z/m)
\longrightarrow
\Fil^{\bullet+1}_{\mot}K(F;\mathbf Z/m),
\]
and the map it induces on \(\operatorname{gr}^{\bullet+1}\) is right
multiplication by \(u_F(a)\).
\end{lemma}

\begin{proof}
For every qcqs scheme \(X\) the natural map
\(R\Gamma_{\mathrm{Nis}}(X,\mathbf G_m)[-1]\to\mathbf Z(1)_{\mot}(X)\) is an isomorphism
in degrees at most three
\cite[\arxivhtmlref{2507.16501v1}{S2.Thmtheorem15}{Ex.~2.15}]{BouisWeibelVanishingPBF}.
Every
Nisnevich covering of \(\Spec(F)\) splits. Thus evaluation is exact on
its small Nisnevich site, the source is \(F^\times[-1]\), and
\(H^1_{\mot}(F,\mathbf Z(1))=F^\times\), while
\(H^j_{\mot}(F,\mathbf Z(1))=0\) for \(j\le0\) and for \(j=2,3\). Motivic Weibel
vanishing
\cite[\arxivhtmlref{2507.16501v1}{Thmtheoremintro7}{Thm.~G}]{BouisWeibelVanishingPBF},
applied to the noetherian
zero-dimensional scheme \(\Spec(F)\), gives
\[
H^j_{\mot}(F,\mathbf Z(i))=0
\qquad(i\ge0,\ j>i),
\]
which covers the degrees \(j>3\) and yields \(\mathbf Z(1)_{\mot}(F)\simeq F^\times[-1]\).

By \cite[\arxivhtmlref{2412.06635v2}{S1.I2.i1}{Thm.~D(1)}]{BouisMixedCharacteristicMotivicCohomology},
\(\mathbf Z(0)_{\mot}(X)\simeq R\Gamma_{\mathrm{cdh}}(X,\mathbf Z)\). A field
is a henselian valuation ring, and henselian valuation rings are precisely the
cdh-local schemes by
\cite[Thm.~2.6]{GabberKelly15Points}. Thus every
cdh covering sieve of \(\Spec(F)\) is maximal. Evaluation at \(\Spec(F)\) is
consequently exact on cdh sheaves, and
\(R\Gamma_{\mathrm{cdh}}(\Spec(F),\mathbf Z)\simeq\mathbf Z\) is concentrated in
degree zero. Hence
\(\pi_i\operatorname{gr}^0_{\mot}K(F)=H^{-i}_{\mot}(F,\mathbf Z(0))=0\) for \(i\ne0\).

We use \Cref{lem:field-filtration-connectivity}\textup{(ii)} for
\(E=\Fil^\bullet_{\mot}K(F)\). For \(q'\ge2\) and \(i\in\{0,1\}\), motivic Weibel vanishing
gives \(\pi_i\operatorname{gr}^{q'}_{\mot}K(F)=H^{2q'-i}_{\mot}(F,\mathbf Z(q'))=0\), since
\(2q'-i>q'\). Hence
\(\pi_0\Fil^2_{\mot}K(F)=\pi_1\Fil^2_{\mot}K(F)=0\). The exact sequence
\[
\pi_1\Fil^2_{\mot}K(F)\longrightarrow
\pi_1\Fil^1_{\mot}K(F)\longrightarrow
\pi_1\operatorname{gr}^1_{\mot}K(F)\longrightarrow
\pi_0\Fil^2_{\mot}K(F)
\]
therefore makes the middle map an isomorphism. Likewise
\(\pi_1\operatorname{gr}^0_{\mot}K(F)=
\pi_2\operatorname{gr}^0_{\mot}K(F)=0\) by the previous paragraph, and
the exact sequence
\[
\pi_2\operatorname{gr}^0_{\mot}K(F)\longrightarrow
\pi_1\Fil^1_{\mot}K(F)\longrightarrow
\pi_1\Fil^0_{\mot}K(F)=K_1(F)\longrightarrow
\pi_1\operatorname{gr}^0_{\mot}K(F)
\]
makes \(\pi_1\Fil^1_{\mot}K(F)\to K_1(F)\) an isomorphism. Both maps are natural in \(F\), and
\(K_1(F)=F^\times\) by \cite[Lem.~III.1.4]{Weibel13KBook}, which gives the asserted natural
isomorphism
\(u_F\colon F^\times\to H^1_{\mot}(F,\mathbf Z(1))\).

The motivic filtration is multiplicative and
	\(\Fil^\bullet_{\mot}K(F;\mathbf Z/m)\) is a filtered module over it by
	\Cref{lem:global-motivic-filtration-kmodell-input}. Choose a map
	\(\mathbf S^1\to\Fil^1_{\mot}K(F)\) representing
	\(\widetilde u(a)\). For every \(q\in\mathbf Z\), the \(q\)-th
	filtration step of right multiplication by this representative is
	the composite
	\[
	\begin{aligned}
	\Sigma\Fil^q_{\mot}K(F;\mathbf Z/m)
	&\simeq
	\Fil^q_{\mot}K(F;\mathbf Z/m)\wedge\mathbf S^1\\
	&\longrightarrow
	\Fil^q_{\mot}K(F;\mathbf Z/m)\wedge\Fil^1_{\mot}K(F)\\
	&\longrightarrow
	\Fil^{q+1}_{\mot}K(F;\mathbf Z/m).
	\end{aligned}
	\]
	Here the right module structure is obtained from the given left
	module structure by the symmetry of spectra.
	Compatibility of the filtered multiplication with the transition maps
	makes these composites a morphism of filtered spectra. On associated
	graded pieces it is right multiplication by the image \(u_F(a)\) of
\(\widetilde u(a)\) in
\(\pi_1\operatorname{gr}^1_{\mot}K(F)\).
\end{proof}

\begin{lemma}[comparison and étale realization]
\label{lem:comparison-inverts-etale-realisation}
Let \(X\) be a qcqs scheme, let \(N=\ell^\nu\) be invertible on \(X\),
fix \(q\ge0\), and let
\[
\rho_q\colon
\mathbf Z/N(q)_{\mot}(X)\longrightarrow
R\Gamma(X_{\et},\mu_N^{\otimes q})
\]
be the syntomic realization map. The composite
\[
\bigl(L_{\mathrm{Nis}}\tau^{\le q}
R\Gamma_{\et}(-,\mu_N^{\otimes q})\bigr)(X)
\xrightarrow{\ \beta_{X,q}\ }
\mathbf Z/N(q)_{\mot}(X)
\xrightarrow{\ \rho_q\ }
R\Gamma(X_{\et},\mu_N^{\otimes q})
\]
is the map induced by the natural transformation
\[
\tau^{\le q}R\Gamma_{\et}(-,\mu_N^{\otimes q})
\longrightarrow
R\Gamma_{\et}(-,\mu_N^{\otimes q})
\]
after factoring through Nisnevich sheafification.

If \(X=\Spec(A)\), where \(A\) is henselian local and ind-smooth over a
Prüfer ring, then, for every \(j\le q\),
\[
H^j(\rho_q)\colon
H^j_{\mot}(A,\mathbf Z/N(q))
\xrightarrow{\ \sim\ }
H^j(A_{\et},\mu_N^{\otimes q})
\]
is inverse to the isomorphism \(H^j(\beta_{A,q})\) induced by the
comparison in \Cref{lem:ordinary-prufer-bl-range-input}.
\end{lemma}

\begin{proof}
Bouis--Kundu define the
Beilinson--Lichtenbaum comparison by taking the vertical map between the
horizontal fibres in the commutative square
\[
\begin{tikzcd}[column sep=huge]
\mathbf Z/N(q)_{\syn}(X) \ar[r] \ar[d,"\mathrm{id}"'] &
\bigl(L_{\mathrm{Nis}}\tau^{>q}\mathbf Z/N(q)_{\syn}\bigr)(X)
\ar[d] \\
\mathbf Z/N(q)_{\syn}(X) \ar[r] &
\bigl(L_{\mathrm{cdh}}\tau^{>q}\mathbf Z/N(q)_{\syn}\bigr)(X).
\end{tikzcd}
\]
The upper fibre is
\(\bigl(L_{\mathrm{Nis}}\tau^{\le q}\mathbf Z/N(q)_{\syn}\bigr)(X)\).
By Bouis's comparison theorem
\cite[Thm.~5.10]{BouisMixedCharacteristicMotivicCohomology}, recalled in
\cite[\arxivhtmlref{2506.09910v2}{S4.I2.i3}{Thm.~4.1\textup{(3)}} and the
diagram preceding
\arxivhtmlref{2506.09910v2}{S4.Thmtheorem1.E1}{\textup{(4.1.1)}}]{BouisKunduBL},
the lower fibre is
\(\mathbf Z/N(q)_{\mot}(X)\); its inclusion into
\(\mathbf Z/N(q)_{\syn}(X)\) is
the syntomic realization \(\rho_q\). Thus taking fibres of the horizontal
arrows in the preceding square gives a commutative triangle
\[
\begin{tikzcd}[column sep=huge]
\bigl(L_{\mathrm{Nis}}\tau^{\le q}\mathbf Z/N(q)_{\syn}\bigr)(X)
\ar[d,"\beta_{X,q}"'] \ar[dr] & \\
\mathbf Z/N(q)_{\mot}(X) \ar[r,"\rho_q"'] &
\mathbf Z/N(q)_{\syn}(X).
\end{tikzcd}
\]
The diagonal map is induced by
\(\tau^{\le q}\mathbf Z/N(q)_{\syn}\to\mathbf Z/N(q)_{\syn}\), after Nisnevich
sheafification. Since \(N\) is invertible on \(X\),
\(\mathbf Z/N(q)_{\syn}\simeq R\Gamma_{\et}(-,\mu_N^{\otimes q})\) by
\cite[\arxivhtmlref{2202.04818v2}{S1.Thmtheorem2}{Ex.~1.2}]{BhattMathew23SyntomicTateTwists}.
This proves the first
assertion.

Now let \(X=\Spec(A)\) as in the statement. The scheme \(X\) is local
for the Nisnevich topology, so evaluation at \(A\) is unchanged by
derived Nisnevich sheafification. By
\Cref{lem:ordinary-prufer-bl-range-input}, \(\beta_{A,q}\) is an
equivalence. For \(j\le q\), the canonical map
\(\tau^{\le q}R\Gamma(A_{\et},\mu_N^{\otimes q})\to
R\Gamma(A_{\et},\mu_N^{\otimes q})\) induces the identity on
\(H^j\). Applying \(H^j\) to the preceding triangle therefore gives
\(H^j(\rho_q)\circ H^j(\beta_{A,q})=\mathrm{id}\), which is the second
assertion.
\end{proof}

\begin{lemma}[étale sheafification of the motivic complexes]
\label{lem:etale-sheafification-of-motivic-complexes}
Let \(F\) be a field, let \(\ell\) be a prime invertible in \(F\), let \(\nu\ge1\) and put
\(N:=\ell^\nu\). Regard \(\mathbf Z(1)_{\mot}\) and the \(\mathbf Z/N(q)_{\mot}\),
\(q\ge0\), as presheaves on the small étale site of \(\Spec(F)\), with
\(L_{\et}\) as in \Cref{conv:galois-cohomology-kummer}. For \(q\ge0\),
let
\[
\rho_q\colon
\mathbf Z/N(q)_{\mot}(F)\longrightarrow
R\Gamma(F_{\et},\mu_N^{\otimes q})
\]
be the syntomic realization map of
\Cref{lem:comparison-inverts-etale-realisation}. For \(a\in F^\times\),
let \(\delta_N(a)\in H^1(F,\mu_N)\) be its Kummer class, as in
\Cref{conv:galois-cohomology-kummer}. Then:
\begin{enumerate}[label=\textup{(\roman*)},leftmargin=*]
\item There is a natural equivalence
\(L_{\et}\mathbf Z/N(q)_{\mot}\simeq\mu_N^{\otimes q}\). Under this
equivalence, the composite
\[
\mathbf Z/N(q)_{\mot}(F)
\longrightarrow
R\Gamma\bigl(F_{\et},L_{\et}\mathbf Z/N(q)_{\mot}\bigr)
\xrightarrow{\ \sim\ }
R\Gamma(F_{\et},\mu_N^{\otimes q})
\]
is \(\rho_q\). In particular, for every \(j\le q\), realization induces
an isomorphism
\[
H^j_{\mot}(F,\mathbf Z/N(q))
\xrightarrow{\ H^j(\rho_q)\ }
H^j(F_{\et},\mu_N^{\otimes q});
\]

\item \(L_{\et}\mathbf Z(1)_{\mot}\simeq\mathbf G_m[-1]\). On
\(H^1\), the unit map induces a natural isomorphism
\(c_F\colon H^1_{\mot}(F,\mathbf Z(1))\to
H^0(F_{\et},\mathbf G_m)=F^\times\);

\item For every \(q\ge0\), there is a scalar
\(\lambda_q\in\mathbf Z/N\) such that, for every
\(z\in H^1_{\mot}(F,\mathbf Z(1))\), every \(j\in\mathbf Z\), and every
\(x\in H^j_{\mot}(F,\mathbf Z/N(q))\),
\[
H^{j+1}(\rho_{q+1})\bigl(x\cdot z\bigr)
=\lambda_q\bigl(H^j(\rho_q)(x)\cup\delta_N(c_F(z))\bigr)
\quad\text{in }H^{j+1}(F,\mu_N^{\otimes(q+1)}).
\]
For every extension of fields \(F\subseteq F'\), one has
\(\lambda_q^{F'}=\lambda_q^F\); in particular the scalar is constant
within each characteristic.
\end{enumerate}
\end{lemma}

\begin{proof}
\textup{(i)} Let \(E/F\) be finite separable. The field \(E\) is a
henselian local ind-smooth algebra over the Prüfer ring \(E\). Hence the comparison \(\beta_{E,q}\) of
\Cref{lem:ordinary-prufer-bl-range-input} is an equivalence, and
\Cref{lem:ordinary-prufer-bl-functoriality} makes these equivalences
natural in \(E\). The commutative triangle of
\Cref{lem:comparison-inverts-etale-realisation} identifies \(\rho_q(E)\),
through this equivalence, with the canonical map
\[
\tau^{\le q}R\Gamma(E_{\et},\mu_N^{\otimes q})
\longrightarrow R\Gamma(E_{\et},\mu_N^{\otimes q}).
\]
The same assertion holds for finite products of such fields, so on the
standard affine basis of \(F_{\et}\) the presheaf
\(\mathbf Z/N(q)_{\mot}\) is naturally equivalent to
\(U\mapsto\tau^{\le q}R\Gamma(U_{\et},\mu_N^{\otimes q})\).

Sheafification of presheaves of abelian groups is exact, so derived
sheafification is t-exact and its cohomology sheaves may be computed on
geometric stalks. Fix a separable closure \(F^{\sep}/F\) and the
corresponding geometric point. For every \(j\ge0\), the stalk of the
\(j\)-th cohomology presheaf is
\[
\varinjlim_{\substack{F\subseteq E\subseteq F^{\sep}\\{}[E:F]<\infty}}
H^j\bigl(G_E,\mu_N(F^{\sep})^{\otimes q}\bigr),
\qquad G_E:=\operatorname{Gal}(F^{\sep}/E).
\]
For \(j=0\), this colimit is
\(\mu_N(F^{\sep})^{\otimes q}\). Let \(j>0\), and represent a class at
some finite separable stage \(E/F\) by a normalized continuous
inhomogeneous cocycle
\(
c\colon G_E^j\longrightarrow \mu_N(F^{\sep})^{\otimes q}.
\)
The target is finite and discrete. Compactness of \(G_E^j\) and
continuity of \(c\) therefore give an open normal subgroup
\(U\triangleleft G_E\), which may also be chosen to act trivially on the
target, such that \(c\) is constant on cosets of \(U\) in each
variable. Equivalently, \(c\) factors through \((G_E/U)^j\). Put
\(E':=(F^{\sep})^U\). For \(u_1,\ldots,u_j\in U\), normalization gives
\[
(\operatorname{res}^{G_E}_U c)(u_1,\ldots,u_j)
=c(u_1,\ldots,u_j)=c(1,\ldots,1)=0.
\]
Thus the class restricts to zero over the finite Galois extension
\(E'/E\). Every positive-degree class therefore dies in the displayed
filtered colimit. The complexes involved are connective, so the negative
cohomology sheaves vanish as well. Consequently
\(
L_{\et}\mathbf Z/N(q)_{\mot}\simeq\mu_N^{\otimes q}.
\)
The preceding identification of \(\rho_q(E)\) with the truncation map on
the affine basis shows that the sheafification unit, followed by this
equivalence and derived global sections, is precisely \(\rho_q\). The
assertion on \(H^j\), \(j\le q\), is then
\Cref{lem:comparison-inverts-etale-realisation} applied to the henselian
local ring \(F\).

\textup{(ii)} By
\Cref{lem:weight-one-edge-kummer-compatibility}, for every finite separable
extension \(F'/F\) the natural map
\(R\Gamma_{\mathrm{Nis}}(F',\mathbf G_m)[-1]
\to \mathbf Z(1)_{\mot}(F')\)
is an equivalence. If
\(U=\coprod_i\Spec(F_i)\) is finite étale over \(F\), both sides on \(U\)
are the finite products of their values on the \(F_i\). Hence the
fieldwise maps assemble to an equivalence on the standard affine basis
of \(F_{\et}\). Since this basis generates the topology, derived
sheafification preserves this equivalence. The presheaf
\(\mathbf G_m\) is representable, hence an fpqc and therefore an étale
sheaf \cite[\stackstag{03O3}]{StacksProject}. Thus \(\mathbf G_m[-1]\) is already
a complex of étale sheaves. Derived sheafification gives
\(L_{\et}\mathbf Z(1)_{\mot}\simeq\mathbf G_m[-1]\). On \(H^1\), its
unit map is the natural isomorphism
\(c_F\colon H^1_{\mot}(F,\mathbf Z(1))
\xrightarrow{\sim}H^0(F_{\et},\mathbf G_m)=F^\times\).

\textup{(iii)} By \Cref{lem:global-motivic-filtration-kmodell-input},
\(\bigoplus_{q\ge0}\mathbf Z(q)_{\mot}\) is a graded ring in presheaves and
\(\bigoplus_{q\ge0}\mathbf Z/N(q)_{\mot}\) is a graded module over it.
We use the right action obtained from the given left action by symmetry.
Derived étale sheafification is symmetric monoidal, so the realization
unit maps carry this right action to the lower horizontal map in the
commutative diagram
\[
\begin{tikzcd}[column sep=huge]
L_{\et}\mathbf Z/N(q)_{\mot}\otimes^{\mathbf L}
L_{\et}\mathbf Z(1)_{\mot}
\ar[r]\ar[d,"\sim"'] &
L_{\et}\mathbf Z/N(q+1)_{\mot}\ar[d,"\sim"]\\
\mu_N^{\otimes q}\otimes^{\mathbf L}\mathbf G_m[-1]
\ar[r,"m_q"'] &
\mu_N^{\otimes(q+1)}.
\end{tikzcd}
\]
The vertical equivalences are those of \textup{(i),(ii)}. Thus the
compatibility between motivic multiplication and realization used here
is a formal consequence of the symmetric monoidality of
\(L_{\et}\), rather than an additional comparison assertion. In weight
one the resulting structure map is a morphism
\(m_q\colon
\mu_N^{\otimes q}\otimes^{\mathbf L}\mathbf G_m[-1]
\to\mu_N^{\otimes(q+1)}\)
in the derived category of étale sheaves on \(\Spec(F)\). Since \(N\) is invertible,
\(\mathbf G_m\) is \(N\)-divisible as an étale sheaf, so the Kummer sequence
\cite[\stackstag{03PL}]{StacksProject} gives
\(\mathbf Z/N\otimes_{\mathbf Z}^{\mathbf L}\mathbf G_m[-1]
\simeq\mu_N\). Therefore
\[
\mu_N^{\otimes q}\otimes_{\mathbf Z}^{\mathbf L}\mathbf G_m[-1]
\simeq
\mu_N^{\otimes q}\otimes_{\mathbf Z/N}^{\mathbf L}
\bigl(\mathbf Z/N\otimes_{\mathbf Z}^{\mathbf L}\mathbf G_m[-1]\bigr)
\simeq \mu_N^{\otimes(q+1)}.
\]
Thus the source of \(m_q\) is concentrated in degree zero. Hence
\(m_q\) is a map between two copies of an object concentrated in degree zero, hence a map of
sheaves; \(\mu_N^{\otimes(q+1)}\) is an invertible \(\mathbf Z/N\)-module sheaf, so its sheaf
of endomorphisms is the constant sheaf \(\mathbf Z/N\), and \(\Spec(F)\) is connected.
Therefore \(m_q\) is the scalar \(\lambda_q\) in
\(\operatorname{Hom}(\mu_N^{\otimes(q+1)},\mu_N^{\otimes(q+1)})
=H^0(F_{\et},\mathbf Z/N)=\mathbf Z/N\).

For a sheaf of \(\mathbf Z/N\)-modules \(A\) and any complex \(C\), the
canonical associativity and change-of-scalars map identifies
\[
A\otimes_{\mathbf Z}^{\mathbf L}C
\simeq
A\otimes_{\mathbf Z/N}^{\mathbf L}
\bigl(\mathbf Z/N\otimes_{\mathbf Z}^{\mathbf L}C\bigr).
\]
Applied to \(A=\mu_N^{\otimes q}\) and
\(C=\mathbf G_m[-1]\), and combined with the Kummer equivalence
\(\mathbf Z/N\otimes_{\mathbf Z}^{\mathbf L}\mathbf G_m[-1]\simeq\mu_N\),
this identifies the pairing induced by \(m_q\) on cohomology with
\(\lambda_q\) times the usual cup product
\[
H^j(F,\mu_N^{\otimes q})\otimes H^1(F,\mu_N)
\longrightarrow H^{j+1}(F,\mu_N^{\otimes(q+1)}).
\]

Under the same identification, the map
\(H^1(F,\mathbf G_m[-1])\to H^1(F,\mathbf G_m[-1]\otimes^{\mathbf L}\mathbf Z/N)
=H^1(F,\mu_N)\)
is induced by the cofiber sequence
\(\mathbf G_m[-1]\xrightarrow{N}\mathbf G_m[-1]\to\mu_N\); its long exact sequence in
cohomology is that of the Kummer sequence of \Cref{conv:galois-cohomology-kummer}, so on
\(H^1(F,\mathbf G_m[-1])=F^\times\) the map is \(\delta_N\).
Naturality of cup products in both variables, together with
\textup{(ii)}, gives the formula in part \textup{(iii)} for every \(j\).

Finally, let \(F\subseteq F'\) be a field extension, and write \(\lambda_q^F\) and
\(\lambda_q^{F'}\) for the two scalars. Choose separable closures with
\(F^{\sep}\subseteq(F')^{\sep}\). For a finite separable subextension \(E/F\) of
\(F^{\sep}/F\), the compositum \(EF'\) inside \((F')^{\sep}\) is finite separable over
\(F'\); the motivic complexes and their multiplication maps are functorial for the ring map
\(E\to EF'\), by \Cref{lem:global-motivic-filtration-kmodell-input}. Passing to the filtered
colimits which compute the two stalks therefore gives a commutative square
\[
\begin{tikzcd}[column sep=large]
	\bigl(\mu_N^{\otimes q}\otimes^{\mathbf L}\mathbf G_m[-1]\bigr)_{F}
	\ar[r,"m_q^{F}"]\ar[d,"\iota"'] &
	\bigl(\mu_N^{\otimes(q+1)}\bigr)_{F}\ar[d,"\iota"]\\
	\bigl(\mu_N^{\otimes q}\otimes^{\mathbf L}\mathbf G_m[-1]\bigr)_{F'}
	\ar[r,"m_q^{F'}"] &
	\bigl(\mu_N^{\otimes(q+1)}\bigr)_{F'}
\end{tikzcd}
\]
of stalks. Under the identifications above, both vertical maps are the inclusion
\(\mu_N(F^{\sep})^{\otimes(q+1)}\to\mu_N((F')^{\sep})^{\otimes(q+1)}\), which is an
isomorphism because \(N\) is invertible and both fields are separably closed, so both groups
are cyclic of order \(N\). A morphism of étale sheaves on \(\Spec(F)\) is determined by its
stalk, and on stalks \(m_q^F\) is multiplication by \(\lambda_q^F\); the square therefore
gives \(\lambda_q^{F}=\lambda_q^{F'}\).
\end{proof}

\begin{definition}[Kummer-normalized weight-one classes]
\label{def:kummer-normalized-weight-one-classes}
Let \(F\) be a field, and let
\[
c_F\colon H^1_{\mot}(F,\mathbf Z(1))
\xrightarrow{\ \sim\ }F^\times
\]
be the isomorphism of
\Cref{lem:etale-sheafification-of-motivic-complexes}\textup{(ii)}. For
\(a\in F^\times\), put
\(
\kappa_F(a):=c_F^{-1}(a).
\)
Let
\[
\widetilde\kappa_F(a)\in\pi_1\Fil^1_{\mot}K(F)
\]
be the unique inverse image of \(\kappa_F(a)\) under the isomorphism
\[
\pi_1\Fil^1_{\mot}K(F)
\xrightarrow{\ \sim\ }
H^1_{\mot}(F,\mathbf Z(1))
\]
of \Cref{lem:weight-one-edge-kummer-compatibility}. Both classes are
natural in \(F\).

For every \(m\ge1\), choosing a map
\(\mathbf S^1\to\Fil^1_{\mot}K(F)\) representing
\(\widetilde\kappa_F(a)\) and using the filtered right module structure
gives a morphism
\[
\Sigma\Fil^\bullet_{\mot}K(F;\mathbf Z/m)
\longrightarrow
\Fil^{\bullet+1}_{\mot}K(F;\mathbf Z/m).
\]
Its homotopy class is independent of the representative. For every
\(q\in\mathbf Z\), its associated graded map
\[
\Sigma\operatorname{gr}^q_{\mot}K(F;\mathbf Z/m)
\longrightarrow
\operatorname{gr}^{q+1}_{\mot}K(F;\mathbf Z/m)
\]
is right multiplication by \(\kappa_F(a)\). Indeed, the filtered
multiplication is compatible with passage to associated graded, and the
image of \(\widetilde\kappa_F(a)\) in
\(\pi_1\operatorname{gr}^1_{\mot}K(F)\) is \(\kappa_F(a)\); this is the
argument in the last paragraph of the proof of
\Cref{lem:weight-one-edge-kummer-compatibility}.
\end{definition}

\begin{remark}[two weight-one normalizations]
\label{rem:two-weight-one-normalizations}
The classes \(\widetilde u(a)\) of
\Cref{lem:weight-one-edge-kummer-compatibility} are normalized by the
canonical identification \(K_1(F)=F^\times\), whereas
\(\widetilde\kappa_F(a)\) is normalized by étale sheafification and the
Kummer class. We do not identify the two natural isomorphisms
\(u_F\colon F^\times\to H^1_{\mot}(F,\mathbf Z(1))\) and
\(c_F^{-1}\colon F^\times\to H^1_{\mot}(F,\mathbf Z(1))\), and no such
identification is needed. The splitting map of
\Cref{con:value-exterior-map} uses the Kummer-normalized classes
\(\widetilde\kappa_F(a)\). The Laurent-series test uses
\(\widetilde u(t)\) only to prove that \(\lambda_q\) is a unit; its
argument is valid for the unspecified element \(c_F(u_F(t))\in F^\times\).
\end{remark}

\begin{lemma}[mixed-coefficient localization projection formula]
\label{lem:mixed-coefficient-localization-projection-formula}
Let \(m\geq 1\), let \(X\) be a qcqs scheme, let
\(i\colon Z\hookrightarrow X\) be a closed immersion, and let
\(j\colon U\hookrightarrow X\) be its quasi-compact open complement. 
Write \(\Perf_Z(X)\subseteq\Perf(X)\) for the full stable subcategory
of perfect complexes whose restriction to \(U\) vanishes, put
\(K^Z(X):=K(\Perf_Z(X))\), and write
\(K^Z(X;\mathbf Z/m):=K^Z(X)\wedge\mathbf S/m\). Write
\[
\partial\colon K_s(U)\longrightarrow K^Z_{s-1}(X),
\qquad
\partial_m\colon K_s(U;\mathbf Z/m)
\longrightarrow K^Z_{s-1}(X;\mathbf Z/m)
\]
for the localization boundaries. For
\(y\in K_r(X;\mathbf Z/m)\) and \(z\in K_s(U)\),
\[
\partial_m\bigl(j^*y\cdot z\bigr)
=(-1)^r\,y\cdot\partial(z)
\quad\text{in}\quad
K^Z_{r+s-1}(X;\mathbf Z/m).
\]
If \(X\) and \(Z\) are regular noetherian, dévissage identifies
\(K^Z(X;\mathbf Z/m)\simeq K(Z;\mathbf Z/m)\), and the right-hand side
corresponds to
\((-1)^r i^*(y)\cdot\partial(z)\).
\end{lemma}

\begin{proof}
Tensor product makes \(\Perf_Z(X)\), \(\Perf(X)\), and \(\Perf(U)\)
module categories over \(\Perf(X)\), where the action on \(\Perf(U)\)
is through \(j^*\). The inclusion and restriction functors are
\(\Perf(X)\)-linear. Nonconnective algebraic \(K\)-theory is lax
symmetric monoidal
\cite[\arxivhtmlref{1103.3923v3}{S5.E9}{Prop.~5.9}]{BGT14MultiplicativeTrace}; it therefore carries these
actions and functors to \(K(X)\)-module structures and \(K(X)\)-linear
maps.

Since nonconnective \(K\)-theory is a localizing invariant
\cite[\arxivhtmlref{1001.2282v4}{S9.SS3}{\S9.3}]{BGT13KTheoryStableCategories}, the underlying sequence
\[
K^Z(X)\longrightarrow K(X)\longrightarrow K(U)
\]
is a cofiber sequence. Cofibers of module spectra are computed on
underlying spectra, so this is a cofiber sequence of \(K(X)\)-modules.
Smashing with \(\mathbf S/m\), and using
\(K(X;\mathbf Z/m)\simeq K(X)\wedge\mathbf S/m\), gives the following
morphism of cofiber sequences:
\[
\begin{tikzcd}[column sep=small]
K(X;\mathbf Z/m)\wedge K^Z(X) \ar[r] \ar[d] &
K(X;\mathbf Z/m)\wedge K(X) \ar[r] \ar[d] &
K(X;\mathbf Z/m)\wedge K(U) \ar[d] \\
K^Z(X;\mathbf Z/m) \ar[r] & K(X;\mathbf Z/m) \ar[r] & K(U;\mathbf Z/m).
\end{tikzcd}
\]
Represent \(y\) by a map
\(\mathbf S^r\to K(X;\mathbf Z/m)\). The induced morphism of
cofiber sequences gives, on the connecting maps, a commutative square
\[
\begin{tikzcd}[column sep=large]
	\Sigma^r K(U)
	\ar[r,"\Sigma^r\partial"]
	\ar[d,"j^*y\cdot -"'] &
	\Sigma^r\Sigma K^Z(X)
	\ar[d,"(-1)^r y\cdot -"]\\
	K(U;\mathbf Z/m)
	\ar[r,"\partial_m"'] &
	\Sigma K^Z(X;\mathbf Z/m).
\end{tikzcd}
\]
Here the sign is induced by the symmetry
\(\mathbf S^r\wedge\mathbf S^1\simeq
\mathbf S^1\wedge\mathbf S^r\), which acts by \((-1)^r\).
Taking \(\pi_s\) and evaluating at \(z\) gives
\(
\partial_m(j^*y\cdot z)=(-1)^r y\cdot\partial(z).
\)

Under the additional regularity hypotheses in the final clause,
\(i_*\) carries \(\Perf(Z)\) to \(\Perf_Z(X)\), dévissage applies, and the
derived projection formula
\(
E\otimes^{\mathbf L}i_*M
\simeq
i_*\bigl(i^*E\otimes^{\mathbf L}M\bigr)
\)
for \(E\in\Perf(X)\) and \(M\in\Perf(Z)\)
identifies the action on the spectrum with supports with the action
through \(i^*\). This proves the final assertion.

\end{proof}

\begin{lemma}[self-intersection of the closed point of a DVR]
	\label{lem:dvr-closed-point-self-intersection}
	Let \(D\) be a discrete valuation ring with residue field \(k\), and
	let \(i\colon\Spec(k)\hookrightarrow\Spec(D)\) be the closed
	immersion. For every \(m\ge1\), the composite
	\[
	K(k;\mathbf Z/m)\xrightarrow{i_*}K(D;\mathbf Z/m)
	\xrightarrow{i^*}K(k;\mathbf Z/m)
	\]
	induces the zero homomorphism on every homotopy group.
\end{lemma}
\begin{proof}
	A uniformizer \(\pi\) of \(D\) gives the resolution
	\(0\to D\xrightarrow{\pi}D\to k\to0\).
	Thus \(k\otimes_D^{\mathbf L}k\) has homotopy groups only in
	degrees zero and one, both isomorphic to \(k\), with their usual
	left and right \(k\)-actions. In degree one the identification
	uses \((\pi)/(\pi^2)\simeq k\).
	Its Postnikov filtration gives a cofiber sequence of derived
	\(k\)-bimodules
	\[
	k[1]\longrightarrow k\otimes_D^{\mathbf L}k\longrightarrow k.
	\]
	Tensoring over the right \(k\)-action with \(M\in\Perf(k)\)
	gives a natural cofiber sequence of exact endofunctors
	\[
	M[1]\longrightarrow i^*i_*M\longrightarrow M.
	\]
	By additivity \cite[1.7.2]{ThomasonTrobaugh90}, the induced map
	on \(K\)-theory is \(K([1])+\mathrm{id}\).
	Applying additivity to \(M\to0\to M[1]\) gives
	\(K([1])=-\mathrm{id}\).
	Hence \(i^*i_*\) is nullhomotopic on integral \(K\)-theory.
	Smashing with \(\mathbf S/m\) proves the assertion.
\end{proof}

\begin{lemma}[Laurent-series test]
\label{lem:laurent-series-weight-one-test}
Let \(k_0\) be a separably closed field, let
\(W_0=k_0\llbracket t\rrbracket\), and let \(F_0=k_0((t))\). Let
\(N=\ell^\nu\) be invertible in \(k_0\). For every \(q\ge0\):
\begin{enumerate}[label=\textup{(\roman*)},leftmargin=*]
\item the groups \(H^0(F_0,\mu_N^{\otimes q})\) and
\(H^1(F_0,\mu_N^{\otimes(q+1)})\) are free of rank one over
\(\mathbf Z/N\). For every \(r\in\mathbf Z\) and every \(a\ge2\),
\(H^a(F_0,\mu_N^{\otimes r})=0\). Cup product with
\(\delta_N(t)\) is an isomorphism from
\(H^0(F_0,\mu_N^{\otimes q})\) to
\(H^1(F_0,\mu_N^{\otimes(q+1)})\);

\item the motivic filtration has a single nonzero graded contribution
to each of the following groups, and the resulting edge morphisms are
isomorphisms
\[
K_{2q}(F_0;\mathbf Z/N)
\xrightarrow{\ \sim\ }
H^0_{\mot}(F_0,\mathbf Z/N(q)),
\qquad
K_{2q+1}(F_0;\mathbf Z/N)
\xrightarrow{\ \sim\ }
H^1_{\mot}(F_0,\mathbf Z/N(q+1));
\]

\item multiplication by \([t]\in K_1(F_0)\) induces an isomorphism
\[
K_{2q}(F_0;\mathbf Z/N)
\xrightarrow{\ \sim\ }
K_{2q+1}(F_0;\mathbf Z/N).
\]
\end{enumerate}
\end{lemma}

\begin{proof}
The ring \(W_0\) is henselian \cite[\stackstag{04GM}]{StacksProject}. Its value
group is \(\mathbf Z\), and its residue field is separably closed.
\Cref{thm:tame-cohomology-package-level-nu}, applied to the one-element
basis and the class \(\delta_N(t)\), gives \textup{(i)}.

We prove \textup{(ii)}. By
\Cref{thm:field-finite-coefficient-k-filtration-input}, the graded piece
of weight \(r\) contributing to \(K_i(F_0;\mathbf Z/N)\) is
\(H^{2r-i}(F_0,\mu_N^{\otimes r})\). Part \textup{(i)} shows that for
\(i=2q\) only \(r=q\) occurs, and for \(i=2q+1\) only \(r=q+1\) occurs.
We give the argument for \(i=2q\). For
\(r\ge q+1\), the groups
\(\pi_{2q}\operatorname{gr}^r_{\mot}\) and
\(\pi_{2q-1}\operatorname{gr}^r_{\mot}\) vanish. Hence
\Cref{lem:field-filtration-connectivity}(ii) gives
\(\pi_{2q}\Fil^{q+1}_{\mot}=\pi_{2q-1}\Fil^{q+1}_{\mot}=0\).
The cofiber sequence in weight \(q\) identifies
\(\pi_{2q}\Fil^q_{\mot}K(F_0;\mathbf Z/N)\) with
\(\pi_{2q}\operatorname{gr}^q_{\mot}K(F_0;\mathbf Z/N)\). For \(r<q\),
\[
\begin{aligned}
\pi_{2q}\operatorname{gr}^r_{\mot}K(F_0;\mathbf Z/N)
&=H^{2r-2q}(F_0,\mu_N^{\otimes r})=0,\\
\pi_{2q+1}\operatorname{gr}^r_{\mot}K(F_0;\mathbf Z/N)
&=H^{2r-2q-1}(F_0,\mu_N^{\otimes r})=0,
\end{aligned}
\]
because étale cohomology vanishes in negative degrees. The cofiber
sequences in weights \(0,\ldots,q-1\) therefore identify
\(\pi_{2q}\Fil^0_{\mot}K(F_0;\mathbf Z/N)\) with
\(\pi_{2q}\Fil^q_{\mot}K(F_0;\mathbf Z/N)\). This proves the first isomorphism.
For the odd degree, use weight \(q+1\): the same proof applies because
the two adjacent homotopy groups of every other graded piece vanish.
This proves the second isomorphism.
The same argument over \(k_0\) gives
\(K_{2r+1}(k_0;\mathbf Z/N)=0\) for every \(r\ge0\).
By closed-point rigidity, the same odd-degree vanishing holds over
\(W_0\).

Let \(i\colon\Spec(k_0)\to\Spec(W_0)\) and
\(j\colon\Spec(F_0)\to\Spec(W_0)\). Nonconnective localization
\cite[Thm.~7.4]{ThomasonTrobaugh90} first gives
\(K^{k_0}(W_0)\to K(W_0)\to K(F_0)\). Since \(i\) is a regular closed
immersion between regular noetherian schemes, dévissage identifies
\(K^{k_0}(W_0)\simeq K(k_0)\). After smashing with \(\mathbf S/N\), we
obtain the fibre sequence
\[
K(k_0;\mathbf Z/N)\xrightarrow{i_*}
K(W_0;\mathbf Z/N)\xrightarrow{j^*}K(F_0;\mathbf Z/N).
\]

According to \cref{lem:dvr-closed-point-self-intersection}, \(i^*i_*=0\) on every \(K_n(k_0;\mathbf Z/N)\). Since \(i^*\) is
an equivalence by
\Cref{lem:valuation-closed-point-finite-coefficient-k-rigidity},
\(i_*=0\) on every homotopy group.

The localization sequence and the odd-degree vanishing over
\(k_0\) give isomorphisms
\[
j^*\colon K_{2q}(W_0;\mathbf Z/N)\xrightarrow{\ \sim\ }
K_{2q}(F_0;\mathbf Z/N),
\qquad
\partial_N\colon K_{2q+1}(F_0;\mathbf Z/N)
\xrightarrow{\ \sim\ }K_{2q}(k_0;\mathbf Z/N).
\]
Let
\(
\partial\colon K_1(F_0)\to K_0(k_0)
\)
be the connecting homomorphism in the integral localization sequence,
using the preceding dévissage equivalence. The lattice description of
the boundary gives
\(\partial([t])=\varepsilon[W_0/tW_0]=\varepsilon[k_0]\), for a sign
\(\varepsilon\in\{\pm1\}\) determined by the localization convention.

Since \(W_0\) and \(k_0\) are regular noetherian, the hypotheses and
the dévissage clause of
\Cref{lem:mixed-coefficient-localization-projection-formula} apply.
Thus, for \(x=j^*(y)\), its projection formula gives
\(\partial_N(x\cdot[t])=\varepsilon\,i^*(y)\).
Thus multiplication by \([t]\) is
\(\varepsilon\,\partial_N^{-1}i^*(j^*)^{-1}\), which proves
\textup{(iii)}.
\end{proof}

\begin{lemma}[the weight-one scalar is a unit]
\label{lem:weight-one-scalar-is-a-unit}
In the situation of
\Cref{lem:etale-sheafification-of-motivic-complexes}, one has
\(\lambda_q\in(\mathbf Z/N)^\times\) for every \(q\ge0\).
\end{lemma}

\begin{proof}
The scalar is unchanged by field extension, by
\Cref{lem:etale-sheafification-of-motivic-complexes}(iii). Choose a
separable closure \(k_0/F\), put \(F_0=k_0((t))\), and use the standard
embeddings \(F\subseteq k_0\subseteq F_0\). Thus
\(\lambda_q^F=\lambda_q^{k_0}=\lambda_q^{F_0}\).
Choose a generator
\(x\in H^0_{\mot}(F_0,\mathbf Z/N(q))\), and let
\(\bar x\in K_{2q}(F_0;\mathbf Z/N)\) be its preimage under the first
isomorphism of
\Cref{lem:laurent-series-weight-one-test}\textup{(ii)}.
Multiplicativity of the filtration and
\Cref{lem:weight-one-edge-kummer-compatibility} identify the image of
\(\bar x\cdot[t]\) under the second isomorphism of
\Cref{lem:laurent-series-weight-one-test}\textup{(ii)} with
\(x\cdot u_{F_0}(t)\). Therefore
\Cref{lem:etale-sheafification-of-motivic-complexes}\textup{(iii)}
gives
\[
H^1(\rho_{q+1})\bigl(x\cdot u_{F_0}(t)\bigr)
=
\lambda_q^{F_0}
\bigl(
H^0(\rho_q)(x)\cup\delta_N(c_{F_0}(u_{F_0}(t)))
\bigr).
\]
By \Cref{lem:laurent-series-weight-one-test}(iii),
\(\bar x\cdot[t]\) is a generator. Since the second isomorphism of
\Cref{lem:laurent-series-weight-one-test}\textup{(ii)} and
\(H^1(\rho_{q+1})\) are isomorphisms, the left-hand side is a generator of
\(H^1(F_0,\mu_N^{\otimes(q+1)})\). This argument does not require
\(c_{F_0}(u_{F_0}(t))=t\). By
\Cref{lem:laurent-series-weight-one-test}\textup{(i)}, the common
target is a free rank-one \(\mathbf Z/N\)-module. If
\(\lambda_q^{F_0}\) were a nonunit, the right-hand side would belong to
\(\ell H^1(F_0,\mu_N^{\otimes(q+1)})\), and hence could not generate
that module. Therefore \(\lambda_q^{F_0}\), and consequently
\(\lambda_q^F\), is a unit.
\end{proof}

\begin{lemma}[weight-one operators are cup products]
\label{lem:weight-one-operators-are-cup-products}
Let \(F\) be a field, let \(\ell\) be a prime invertible in \(F\), let
\(\nu\ge1\), and put \(N:=\ell^\nu\). For every \(q\ge0\), there is a
unit
\(
\lambda_q\in(\mathbf Z/N)^\times,
\)
preserved by every extension of fields, such that, for every \(a\in F^\times\)
and every \(j\le q\), the square
\[
\begin{tikzcd}[column sep=huge]
H^j_{\mot}(F,\mathbf Z/N(q))
\ar[r,"{x\mapsto x\cdot\kappa_F(a)}"]
\ar[d,"H^j(\rho_q)"'] &
H^{j+1}_{\mot}(F,\mathbf Z/N(q+1))
\ar[d,"H^{j+1}(\rho_{q+1})"]\\
H^j(F,\mu_N^{\otimes q})
\ar[r,"{x\mapsto\lambda_q(x\cup\delta_N(a))}"'] &
H^{j+1}(F,\mu_N^{\otimes(q+1)})
\end{tikzcd}
\]
commutes, and its vertical maps are isomorphisms.
\end{lemma}

\begin{proof}
Apply
\Cref{lem:etale-sheafification-of-motivic-complexes}\textup{(i),(iii)}
to \(z=\kappa_F(a)\). By definition,
\(c_F(\kappa_F(a))=a\). The scalar is a unit by
\Cref{lem:weight-one-scalar-is-a-unit}. The vertical maps are
isomorphisms because \(j\le q\) and \(j+1\le q+1\).
\end{proof}

\begin{construction}[the value-exterior map]
\label{con:value-exterior-map}
Let \(W\) be a henselian valuation ring with fraction field \(L\),
residue field \(k\), and value group \(\Gamma_W\). Let
\(N=\ell^\nu\) be invertible in \(W\). By
\Cref{lem:henselian-tame-inertia-cohomology}\textup{(i)},
\(\Gamma_W/N\Gamma_W\) is a free \(\mathbf Z/N\)-module. Choose an
ordered \(\mathbf Z/N\)-basis \(B\). For each \(b\in B\), choose
\(\varpi_b\in L^\times\) such that
\(
v(\varpi_b)\bmod N\Gamma_W=b.
\)

Put
\[
\mathcal E_B^\bullet:=
\bigoplus_{\substack{J\subseteq B\\J\text{ finite}}}
\Sigma^{|J|}
\Fil_{\mot}^{\bullet-|J|}K(k;\mathbf Z/N).
\]
Let \(J=\{b_1<\cdots<b_r\}\subseteq B\) be finite. First apply the
specialization map \(s_{W,N}\) of
\Cref{lem:rigidity-specialization-filtered-inflation}. For
\(1\le i\le r\), choose a map
\(\mathbf S^1\to\Fil^1_{\mot}K(L)\) representing
\(\widetilde\kappa_L(\varpi_{b_i})\). Starting with \(s_{W,N}\),
apply the corresponding right-multiplication maps in the order
\(b_1,\ldots,b_r\). This gives a morphism of filtered spectra
\[
\Sigma^r\Fil_{\mot}^{\bullet-r}K(k;\mathbf Z/N)
\longrightarrow\Fil_{\mot}^\bullet K(L;\mathbf Z/N).
\]
For \(J=\varnothing\), this is \(s_{W,N}\).
The universal property of the coproduct assembles these morphisms into
\[
\Phi_B\colon\mathcal E_B^\bullet
\longrightarrow\Fil_{\mot}^\bullet K(L;\mathbf Z/N).
\]
Any two maps
\(\mathbf S^1\to\Fil^1_{\mot}K(L)\) representing the same element of
\(\pi_1\Fil^1_{\mot}K(L)\) are homotopic, and the module action sends
such a homotopy to a homotopy of the corresponding multiplication
maps. Maps and homotopies out of a coproduct are specified
summandwise. Hence the homotopy class of \(\Phi_B\) is independent of
the chosen representatives.

\end{construction}

\begin{lemma}[completeness of the value-exterior source]
\label{lem:value-exterior-source-complete}
The filtered spectrum \(\mathcal E_B^\bullet\) of
\Cref{con:value-exterior-map} is complete.
\end{lemma}

\begin{proof}
Its \(s\)-th step is
\[
\mathcal E_B^s=
\bigoplus_J
\Sigma^{|J|}\Fil_{\mot}^{s-|J|}K(k;\mathbf Z/N).
\]
Homotopy groups commute with coproducts. If \(|J|\le s\), then
\[
\pi_{i-|J|}
\Fil_{\mot}^{s-|J|}K(k;\mathbf Z/N)=0
\qquad(i\le s-2)
\]
by \Cref{lem:field-filtration-connectivity}(i). If \(|J|>s\), the
filtration step is \(K(k;\mathbf Z/N)\), by
\Cref{conv:filtered-spectral-sequence-indexing}, and its contribution
vanishes for \(i<|J|\) because coefficient \(K\)-theory of a field is
connective. Hence
\[
\pi_i\mathcal E_B^s=0\qquad(i\le s-2).
\]
By convention, completeness is defined using the limit in spectra of the sequential
tower \(\{\mathcal E_B^s\}_{s\ge0}\), with transition maps
\(\mathcal E_B^{s+1}\to\mathcal E_B^s\). For fixed \(i\), both inverse systems
\(\{\pi_i\mathcal E_B^s\}_s\) and
\(\{\pi_{i+1}\mathcal E_B^s\}_s\) are eventually zero. 
The Milnor exact
sequence
\[
0\longrightarrow
\varprojlim\nolimits_s^1\pi_{i+1}\mathcal E_B^s
\longrightarrow
\pi_i\!\left(\lim_s\mathcal E_B^s\right)
\longrightarrow
\varprojlim\nolimits_s\pi_i\mathcal E_B^s
\longrightarrow0
\]
therefore gives \(\pi_i(\lim_s\mathcal E_B^s)=0\) for every \(i\).
Homotopy groups are conservative on spectra, so
\(\lim_s\mathcal E_B^s\) is contractible.
\end{proof}

\begin{lemma}[associated graded value decomposition]
\label{lem:value-exterior-associated-graded}
For every \(q\ge0\), the map
\[
\operatorname{gr}^q(\Phi_B)\colon
\operatorname{gr}^q\mathcal E_B^\bullet
\longrightarrow
\operatorname{gr}_{\mot}^qK(L;\mathbf Z/N)
\]
is an equivalence.
\end{lemma}

\begin{proof}
Coproducts and cofibers in \(\FilSp\) are computed objectwise. Hence,
for a finite subset \(J\subseteq B\), with \(r=|J|\),
\[
\operatorname{gr}^q\!\left(
\Sigma^r\Fil_{\mot}^{\bullet-r}K(k;\mathbf Z/N)
\right)
\simeq
\Sigma^r\operatorname{gr}_{\mot}^{q-r}K(k;\mathbf Z/N).
\]
The right-hand side is zero when \(r>q\), by
\Cref{conv:filtered-spectral-sequence-indexing}. Therefore
\[
\operatorname{gr}^q\mathcal E_B^\bullet
\simeq
\bigoplus_{\substack{J\subseteq B\text{ finite}\\|J|\le q}}
\Sigma^{|J|}\operatorname{gr}_{\mot}^{q-|J|}K(k;\mathbf Z/N).
\]

Fix \(t\in\mathbf Z\) and put \(a:=2q-t\). Using
\(\pi_t\Sigma^{|J|}E=\pi_{t-|J|}E\) and the graded-piece formula of
\Cref{thm:field-finite-coefficient-k-filtration-input}, the map induced
by \(\operatorname{gr}^q(\Phi_B)\) on \(\pi_t\) has the form
\[
\bigoplus_{\substack{J\subseteq B\text{ finite}\\|J|\le q}}
H_{\mot}^{a-|J|}
\bigl(k,\mathbf Z/N(q-|J|)\bigr)
\longrightarrow
H_{\mot}^{a}\bigl(L,\mathbf Z/N(q)\bigr).
\]
If \(a<0\) or \(a>q\), every group in this display is zero by
\Cref{thm:field-finite-coefficient-k-filtration-input}. Assume henceforth
that \(0\le a\le q\). The summands with \(|J|>a\) are then zero, and
the realization isomorphisms identify the remaining source and target
groups with
\[
H^{a-|J|}\bigl(k,\mu_N^{\otimes(q-|J|)}\bigr)
\quad\text{and}\quad
H^a\bigl(L,\mu_N^{\otimes q}\bigr),
\]
respectively.

Let \(J=\{b_1<\cdots<b_r\}\), with \(r\le a\). By
\Cref{con:value-exterior-map}, the map on the \(J\)-summand is
specialization followed by right multiplication by
\(\widetilde\kappa_L(\varpi_{b_1}),\ldots,
\widetilde\kappa_L(\varpi_{b_r})\), in that order.
By \Cref{lem:rigidity-specialization-filtered-inflation},
specialization becomes inflation under realization. Before the
\((i+1)\)-st multiplication, for \(0\le i<r\), the cohomological degree
is \(a-r+i\) and the weight is \(q-r+i\). Since \(a\le q\), one has
\(a-r+i\le q-r+i\), so
\Cref{lem:weight-one-operators-are-cup-products} applies at every step.
Consequently, the \(J\)-summand sends \(\alpha\) to
\[
\varepsilon_{J,t}
\left(\prod_{i=0}^{r-1}\lambda_{q-r+i}\right)
\operatorname{inf}(\alpha)\cup
\delta_N(\varpi_{b_1})\cup\cdots\cup\delta_N(\varpi_{b_r}),
\]
where \(\varepsilon_{J,t}\in\{\pm1\}\) is the sign arising from the
chosen identifications of the iterated suspensions. For
\(J=\varnothing\), the product and cup product are empty and the map is
inflation.

Put \(w_b:=\delta_N(\varpi_b)\). By
\Cref{lem:henselian-tame-inertia-cohomology}\textup{(ii)}, the
restriction of \(w_b\) to inertia corresponds to \(b\). Thus, after
omitting the displayed scalar and sign on each summand, the resulting
direct-sum map is precisely the isomorphism \(\Theta\) of
\Cref{thm:tame-cohomology-package-level-nu}. Each factor
\(\varepsilon_{J,t}\prod_{i=0}^{r-1}\lambda_{q-r+i}\) is a unit in
\(\mathbf Z/N\), by
\Cref{lem:weight-one-scalar-is-a-unit}. Componentwise multiplication
by these units is an automorphism of the direct sum. Hence
\(\pi_t(\operatorname{gr}^q(\Phi_B))\) is an isomorphism for every
\(t\). Homotopy groups detect equivalences of spectra, so
\(\operatorname{gr}^q(\Phi_B)\) is an equivalence.
\end{proof}

\begin{theorem}[value-exterior splitting]
\label{thm:value-exterior-splitting}
In the notation of \Cref{con:value-exterior-map}, the map
\[
\Phi_B\colon
\bigoplus_{\substack{J\subseteq B\\J\text{ finite}}}
\Sigma^{|J|}
\Fil_{\mot}^{\bullet-|J|}K(k;\mathbf Z/N)
\longrightarrow
\Fil_{\mot}^{\bullet}K(L;\mathbf Z/N)
\]
is an equivalence of complete filtered spectra. Consequently
\[
K(L;\mathbf Z/N)\simeq
\bigoplus_{\substack{J\subseteq B\\J\text{ finite}}}
\Sigma^{|J|}K(k;\mathbf Z/N),
\]
and \(s_{W,N}\) is the inclusion of the summand indexed by
\(J=\varnothing\).
\end{theorem}

\begin{proof}
The source is complete by
\Cref{lem:value-exterior-source-complete}. Since \(L\) is a field,
\(\Spec(L)\) has valuative dimension zero by
\Cref{lem:valuation-inputs-finite-valuative-dimension}; hence the target is
complete by \Cref{lem:global-motivic-filtration-kmodell-input}.
The associated graded maps are equivalences by
\Cref{lem:value-exterior-associated-graded}. Apply
\Cref{lem:filtered-whitehead-criterion}. Evaluating this filtered
equivalence in filtration degree zero, and using
\(\Fil_{\mot}^{-|J|}K(k;\mathbf Z/N)=
\Fil_{\mot}^{0}K(k;\mathbf Z/N)=K(k;\mathbf Z/N)\), gives the asserted
underlying equivalence. By \Cref{con:value-exterior-map}, its restriction
to the empty summand is \(s_{W,N}\).
\end{proof}

\begin{corollary}[specialization is split injective]
	\label{cor:specialisation-split-injective}
	Let \(W\) be a henselian valuation ring with fraction field \(L\) and residue
	field \(k\). Let \(\ell\) be a prime invertible in \(k\), let \(\nu\ge1\), and put
	\(N:=\ell^\nu\). Then the underlying map of
	\(s_{W,N}\), namely \(K(k;\mathbf Z/N)\to K(L;\mathbf Z/N)\),
	admits a retraction; in particular it is injective on every homotopy group.
\end{corollary}
\begin{proof}
	Choose a homotopy inverse to the underlying equivalence
	\[
	\Phi_B\colon
	\bigoplus_{\substack{J\subseteq B\\J\text{ finite}}}
	\Sigma^{|J|}K(k;\mathbf Z/N)
	\xrightarrow{\ \sim\ }
	K(L;\mathbf Z/N)
	\]
	of \Cref{thm:value-exterior-splitting}, and compose it with the
	projection onto the summand \(J=\varnothing\). Since the restriction of
	\(\Phi_B\) to this summand is \(s_{W,N}\), the resulting map
	\(K(L;\mathbf Z/N)\to K(k;\mathbf Z/N)\) is a retraction of
	\(s_{W,N}\).
\end{proof}

\begin{remark}[discrete, divisible, and nondiscrete value groups]
\label{rem:rank-one-dvr-check}
If \(W\) is a henselian discrete valuation ring, then
\Cref{thm:value-exterior-splitting} gives
\[
K(L;\mathbf Z/N)\simeq
K(k;\mathbf Z/N)\oplus\Sigma K(k;\mathbf Z/N),
\]
in agreement with localization, rigidity, and dévissage. 
If
\(\Gamma_W\) is \(N\)-divisible, then
\(\Gamma_W/N\Gamma_W=0\), and generic restriction is an equivalence.

Rank one alone does not reduce to the discrete case. For example,
\(\Gamma=\mathbf Z+\mathbf Z\alpha\subset\mathbf R\), with
\(\alpha\notin\mathbf Q\), is a nondiscrete rank-one ordered group and
\(\Gamma/N\Gamma\cong(\mathbf Z/N)^2\). Such henselian examples exist. Indeed, the canonical valuation on the
Hahn field \(k_0((t^\Gamma))\) has value group \(\Gamma\) and residue field
\(k_0\). Henselizing its canonical valuation ring produces a henselian
valuation ring with the same value group and residue field: henselization
preserves the value group by \cite[\stackstag{0ASK}]{StacksProject} and the
residue field by \cite[\stackstag{04GN}]{StacksProject}.

More generally, choose a sequence
\((\alpha_i)_{i\ge1}\subset\mathbf R\) which is linearly independent over
\(\mathbf Q\), and put
\(
\Gamma_\infty:=\bigoplus_{i\ge1}\mathbf Z\alpha_i\subset\mathbf R.
\)
Then \(\Gamma_\infty\) has rank one as an ordered group, has infinite
\(\mathbf Z\)-rank, and
\(
\Gamma_\infty/N\Gamma_\infty
\simeq
\bigoplus_{i\ge1}\mathbf Z/N.
\)
Applying the same Hahn-field construction to
\(k_0((t^{\Gamma_\infty}))\) gives a henselian example with value group
\(\Gamma_\infty\) and residue field \(k_0\). Thus the basis \(B\) in
\Cref{thm:value-exterior-splitting} can be infinite even in rank one. For
either of these rank-one henselian valuation rings \(W\),
\(\Spec(W)=\{0,\mathfrak m_W\}\) and \(W_{\mathfrak m_W}=W\).
Consequently, the induction of
\Cref{thm:finite-rank-prime-to-residue-characteristic-generic-injectivity} makes no reduction in these examples.
In these examples, \Cref{thm:value-exterior-splitting} also identifies
the summands indexed by nonempty finite subsets of \(B\).

\end{remark}

\section{Gersten injectivity for valuation rings and henselian pairs}
\label{sec:henselian-pair-injectivity}

The filtered splitting identifies the generic restriction map for a henselian
valuation ring. We next isolate a weaker prime-step argument for
lexicographic finite-rank valuations, and then compare a regular
henselian pair with such a dominating valuation ring.

\begin{theorem}[Gersten injectivity]
\label{thm:henselian-valuation-generic-injectivity}
Let \(W\) be a henselian valuation ring with fraction field \(L\), and let
\(N=\ell^\nu\) be invertible in \(W\). Then
\(K(W;\mathbf Z/N)\longrightarrow K(L;\mathbf Z/N)\)
admits a retraction; in particular \(K_n(W;\mathbf Z/N)\to K_n(L;\mathbf Z/N)\) is injective
for every \(n\in\mathbf Z\).
\end{theorem}
\begin{proof}
	Let \(k\) be the residue field of \(W\). By
	\Cref{lem:valuation-closed-point-finite-coefficient-k-rigidity}, the
	underlying map of \(\rho_{W,N}\) is an equivalence
	\[
	\rho_{W,N}\colon
	K(W;\mathbf Z/N)\xrightarrow{\ \sim\ }K(k;\mathbf Z/N).
	\]
	Choose a homotopy inverse
	\(\sigma_{W,N}\colon K(k;\mathbf Z/N)\to K(W;\mathbf Z/N)\).
	By \Cref{lem:rigidity-specialization-filtered-inflation}, after
	forgetting the filtrations one has
	\(j^*_{W,N}\simeq s_{W,N}\circ\rho_{W,N}\).
	If
	\(r\colon K(L;\mathbf Z/N)\to K(k;\mathbf Z/N)\)
	is a retraction of \(s_{W,N}\) furnished by
	\Cref{cor:specialisation-split-injective}, then
	\[
	(\sigma_{W,N}\circ r)\circ j^*_{W,N}
	\simeq
	\sigma_{W,N}\circ r\circ s_{W,N}\circ\rho_{W,N}
	\simeq
	\sigma_{W,N}\circ\rho_{W,N}
	\simeq\mathrm{id}.
	\]
	Thus \(\sigma_{W,N}\circ r\) is a retraction of generic restriction.
\end{proof}

We record a formal induction which deduces finite-rank Gersten
injectivity from the rank-one case. Once rank-one Gersten injectivity and
closed-point rigidity are known, the only additional \(K\)-theoretic input is
Tamme's excision theorem.

The square of
\Cref{lem:valuation-prime-step-square} below is well-known and is used for instance in the proof of
\cite[Lem.~3.5\textup{(ii)}]{KellyMorrowValuationRings}.

\begin{lemma}[valuation prime-step square]
	\label{lem:valuation-prime-step-square}
	Let \(A\) be a valuation ring with fraction field \(F\), and let
	\(
	\mathfrak q\subset A
	\)
	be a prime ideal.
	Then the natural square
	\[
	\begin{tikzcd}
		A \ar[r] \ar[d] & A_{\mathfrak q} \ar[d] \\
		A/\mathfrak q \ar[r] & \kappa(\mathfrak q)
	\end{tikzcd}
	\]
	is cartesian.
\end{lemma}

\begin{proof}
	The residue field \(\kappa(\mathfrak q)\) is canonically
	\(\Frac(A/\mathfrak q)\) \cite[\stackstag{00CK}]{StacksProject}. The map
	\[
	A\longrightarrow A_{\mathfrak q}
	\times_{\kappa(\mathfrak q)} A/\mathfrak q
	\]
	is injective because \(A\) is a domain and
	\(A\to A_{\mathfrak q}\subset F\) is injective.

	We prove surjectivity. Let
	\(
	(x,\bar b)\in
	A_{\mathfrak q}\times_{\kappa(\mathfrak q)}A/\mathfrak q
	\)
	be a pair with the same image in \(\kappa(\mathfrak q)\). Write
	\(
	x=a/s
	\)
	with \(a\in A\) and \(s\in A\setminus\mathfrak q\), and choose a lift \(b\in A\) of
	\(\bar b\). The equality of the two images in
	\(
	\kappa(\mathfrak q)=\Frac(A/\mathfrak q)
	\)
	says that
	\(
	\bar a/\bar s=\bar b
	\),
	or equivalently
	\(
	a-bs\in\mathfrak q
	\).
	Set \(d:=a-bs\in\mathfrak q\).

	We claim that \(d/s\in A\). This is clear if \(d=0\). If \(d\ne0\) and \(d/s\notin A\),
	then, since divisibility is total in a valuation ring
	\cite[\stackstag{0ASN}]{StacksProject}, \(s/d\in A\). Hence
	\(
	s=(s/d)d\in\mathfrak q,
	\)
	contradicting \(s\notin\mathfrak q\). Thus \(d/s\in A\). Since
	\(s(d/s)=d\in\mathfrak q\)
	and \(s\notin\mathfrak q\), the primality of \(\mathfrak q\) gives \(d/s\in\mathfrak q\).

	The element \(b+d/s\in A\) maps to \(\bar b\in A/\mathfrak q\),
	because \(d/s\in\mathfrak q\), and maps to
	\(b+d/s=\frac{bs+d}{s}=\frac{a}{s}=x\)
		in \(A_{\mathfrak q}\). Hence
		\(A\to A_{\mathfrak q}\times_{\kappa(\mathfrak q)}A/\mathfrak q\)
		is surjective.
\end{proof}

\begin{lemma}[Tor-unitality of the localization]
	\label{lem:valuation-coarsening-localization-tor-unital}
	With notation as in \Cref{lem:valuation-prime-step-square}, the localization map
	\(
	A\longrightarrow A_{\mathfrak q}
	\)
	is Tor-unital:
	\(A_{\mathfrak q}\otimes_A^{\mathbf L} A_{\mathfrak q}
	\simeq A_{\mathfrak q}\).
\end{lemma}

\begin{proof}
	The map \(A\to A_{\mathfrak q}\) is a localization. Hence
	\(A_{\mathfrak q}\) is flat over \(A\)
	\cite[\stackstag{00HT}]{StacksProject}, and
	\(A_{\mathfrak q}\otimes_A^{\mathbf L}A_{\mathfrak q}
	\simeq A_{\mathfrak q}\otimes_A A_{\mathfrak q}
	\simeq A_{\mathfrak q}\).
\end{proof}

\begin{theorem}[Tamme excision for Tor-unital homotopy Milnor squares]
	\label{thm:tamme-derived-milnor-square-input}
	Let
	\[
	\begin{tikzcd}
		A \ar[r] \ar[d] & A' \ar[d] \\
		B \ar[r] & B'
	\end{tikzcd}
	\]
	be a cartesian square of discrete rings. Assume that, when regarded as a square of ring
	spectra, it is homotopy cartesian; equivalently, for discrete rings, assume that the sequence
	\(0\longrightarrow A\longrightarrow A'\oplus B\longrightarrow B'\longrightarrow 0\)
	is exact. Assume moreover that the map \(A\to A'\) is Tor-unital, i.e.
	\(A'\otimes_A^{\mathbf L}A'\simeq A'\).
	Then, for every weakly localizing invariant \(E\) of small stable
	\(\infty\)-categories with values in a stable \(\infty\)-category,
	that is, every such invariant that sends exact sequences to fibre
	sequences, the induced square
	\[
	\begin{tikzcd}
		E(\operatorname{Perf}(A)) \ar[r] \ar[d] &
		E(\operatorname{Perf}(A')) \ar[d] \\
		E(\operatorname{Perf}(B)) \ar[r] &
		E(\operatorname{Perf}(B'))
	\end{tikzcd}
	\]
	is cartesian. In particular, nonconnective algebraic \(K\)-theory sends the original square
	to a cartesian square of spectra.
\end{theorem}

\begin{proof}
	For discrete rings, Tamme identifies the homotopy-cartesian condition on the underlying square
	of ring spectra with exactness of
	\(0\longrightarrow A\longrightarrow A'\oplus B\longrightarrow B'\longrightarrow 0\)
	and identifies Tor-unitality with the condition that the multiplication map
	\(A'\otimes_A^{\mathbf L}A'\longrightarrow A'\)
	is an equivalence; see \cite[\arxivhtmlref{1703.03331v3}{Thmthm3}{Ex.~3}]{Tamme18Excision}.

	By \cite[\arxivhtmlref{1703.03331v3}{Thmthm28}{Thm.~28}]{Tamme18Excision},
	the functor \(\operatorname{Perf}(-)\) sends the ring
	square to an excisive square of small stable \(\infty\)-categories. Applying the excision
	theorem for excisive squares \cite[\arxivhtmlref{1703.03331v3}{Thmthm18}{Thm.~18}]{Tamme18Excision} to such
	an invariant \(E\) yields the cartesian square in the statement.
	Taking \(E\) to be nonconnective
	algebraic \(K\)-theory gives the final assertion, which is
	\cite[\arxivhtmlref{1703.03331v3}{Thmthm2}{Thm.~2}]{Tamme18Excision}.
\end{proof}

The following statement is a special case of
Dahlhausen's excision theorem for semi-valuation rings
\cite[Thm.~21]{Dahlhausen25SemiValuation}.
We include a direct proof using Tamme's excision theorem.

\begin{theorem}[prime-step \(K\)-excision for valuation rings]
	\label{thm:valuation-prime-step-k-excision}
	Let \(A\) and \(\mathfrak q\) be as in
	\Cref{lem:valuation-prime-step-square}. Then the square of nonconnective \(K\)-theory
	spectra
	\[
	\begin{tikzcd}
		K(A) \ar[r] \ar[d] & K(A_{\mathfrak q}) \ar[d] \\
		K(A/\mathfrak q) \ar[r] & K(\kappa(\mathfrak q))
	\end{tikzcd}
	\]
	is cartesian. Equivalently, there is a functorial fibre sequence
	\[
	K(A) \longrightarrow
	K(A_{\mathfrak q})\oplus K(A/\mathfrak q)
	\longrightarrow K(\kappa(\mathfrak q)),
	\]
	where the second arrow is the difference of the maps induced by
	\(A_{\mathfrak q}\to\kappa(\mathfrak q)\) and
	\(A/\mathfrak q\to\kappa(\mathfrak q)\).

	Smashing with a spectrum is exact, so the same holds after smashing the fibre sequence
	with any spectrum. In particular, for every \(m\ge1\),
	\[
	K(A;\mathbf Z/m)
	\longrightarrow
	K(A_{\mathfrak q};\mathbf Z/m)\oplus K(A/\mathfrak q;\mathbf Z/m)
	\longrightarrow
	K(\kappa(\mathfrak q);\mathbf Z/m)
	\]
	is a functorial fibre sequence, again with the difference map as its second arrow and
	with \(K(-;\mathbf Z/m)\) as in
	\Cref{conv:nonconnective-k-groups}.
\end{theorem}

\begin{proof}
	By \Cref{lem:valuation-prime-step-square}, the ring square
	\[
	\begin{tikzcd}
		A \ar[r] \ar[d] & A_{\mathfrak q} \ar[d] \\
		A/\mathfrak q \ar[r] & \kappa(\mathfrak q)
	\end{tikzcd}
	\]
	is cartesian. Moreover, the map
	\(
	A_{\mathfrak q}\to\kappa(\mathfrak q)
	\)
	is surjective, with kernel \(\mathfrak qA_{\mathfrak q}\). Thus, for the usual difference
	map
	\[
	A_{\mathfrak q}\oplus A/\mathfrak q\longrightarrow\kappa(\mathfrak q),
	\qquad (x,\bar b)\longmapsto \bar x-\bar b,
	\]
	the kernel is the fibre product
	\(A_{\mathfrak q}\times_{\kappa(\mathfrak q)}A/\mathfrak q\simeq A\),
	and the map is surjective. Hence
	\[
	0\longrightarrow A\longrightarrow
	A_{\mathfrak q}\oplus A/\mathfrak q
	\longrightarrow\kappa(\mathfrak q)\longrightarrow0
	\]
	is exact, so the square is homotopy cartesian when regarded as a square of ring spectra.

	By \Cref{lem:valuation-coarsening-localization-tor-unital}, the localization map
	\(A\to A_{\mathfrak q}\) is Tor-unital. Thus the hypotheses of
	\Cref{thm:tamme-derived-milnor-square-input} are satisfied. Therefore nonconnective
	algebraic \(K\)-theory sends the square to a cartesian square of spectra. Equivalently,
	there is a functorial fibre sequence
	\[
	K(A) \longrightarrow
	K(A_{\mathfrak q})\oplus K(A/\mathfrak q)
	\longrightarrow K(\kappa(\mathfrak q)).
	\]

	Smashing with a spectrum is exact on spectra. Applying \(-\wedge\mathbf S/m\) to the
	preceding fibre sequence gives the coefficient form.
\end{proof}

\begin{lemma}[henselian localizations and quotients]
	\label{lem:strategy-d-henselian-localisations}
	Let \(V\) be a henselian valuation ring and let \(\mathfrak p\subset V\) be a prime ideal.
	Then both \(V_{\mathfrak p}\) and \(V/\mathfrak p\) are henselian valuation rings.
\end{lemma}

\begin{proof}
	By \cite[\stackstag{088Y}]{StacksProject}, both \(V_{\mathfrak p}\) and
	\(V/\mathfrak p\) are valuation rings. The maximal ideal of
	\(V_{\mathfrak p}\) is \(\mathfrak pV_{\mathfrak p}\), its residue field is
	\(\kappa(\mathfrak p)=\Frac(V/\mathfrak p)\), and the image of \(V\) in this
	residue field is \(V/\mathfrak p\). Thus the inclusion
	\(V\subseteq V_{\mathfrak p}\) is exactly the situation of
	\cite[Cor.~4.1.4]{EnglerPrestel05ValuedFields}. That corollary asserts that
	\(V\) is henselian if and only if both \(V_{\mathfrak p}\) and the valuation
	ring \(V/\mathfrak p\) of its residue field are henselian. Since \(V\) is
	henselian, the conclusion follows.
\end{proof}

\begin{lemma}[Gersten injectivity in a prime-step square]
	\label{lem:prime-step-generic-injectivity-surjective-residue}
	Let \(A\) be a valuation ring with fraction field \(F\), let
	\(\mathfrak q\subset A\) be a prime ideal.
	Let \(N\ge2\) be an integer. Assume that:
	\begin{enumerate}[label=\textup{(\roman*)}]
		\item
	\(K(A_{\mathfrak q};\mathbf Z/N)\longrightarrow K(F;\mathbf Z/N)\)
		is injective on homotopy groups;

		\item
	\(K(A/\mathfrak q;\mathbf Z/N)\longrightarrow
	K(\kappa(\mathfrak q);\mathbf Z/N)\)
		is injective on homotopy groups;

		\item
	\(K(A_{\mathfrak q};\mathbf Z/N)\longrightarrow
	K(\kappa(\mathfrak q);\mathbf Z/N)\)
		is surjective on homotopy groups.
	\end{enumerate}
	Then
	\(K(A;\mathbf Z/N)\longrightarrow K(F;\mathbf Z/N)\)
	is injective on homotopy groups.
\end{lemma}
\begin{proof}
	By \Cref{thm:valuation-prime-step-k-excision}, there is a fibre
	sequence
	\[
	K(A;\mathbf Z/N)
	\longrightarrow
	K(A_{\mathfrak q};\mathbf Z/N)\oplus K(A/\mathfrak q;\mathbf Z/N)
	\longrightarrow
	K(\kappa(\mathfrak q);\mathbf Z/N).
	\]
	Hypothesis \textup{(iii)} implies that the last map is
	surjective on every homotopy group. Hence the connecting maps
	\(K_{n+1}(\kappa(\mathfrak q);\mathbf Z/N)
	\longrightarrow K_n(A;\mathbf Z/N)\)
	vanish, and therefore
	\[
	K_n(A;\mathbf Z/N) \longrightarrow
	K_n(A_{\mathfrak q};\mathbf Z/N)\oplus K_n(A/\mathfrak q;\mathbf Z/N)
	\]
	is injective.

	Let \(x\in K_n(A;\mathbf Z/N)\) have zero image in
	\(K_n(F;\mathbf Z/N)\). Its image in
	\(K_n(A_{\mathfrak q};\mathbf Z/N)\) is zero by \textup{(i)}. By commutativity
	of the prime-step square (\Cref{lem:valuation-prime-step-square}), its image in
	\(K_n(A/\mathfrak q;\mathbf Z/N)\) maps to zero in
	\(K_n(\kappa(\mathfrak q);\mathbf Z/N)\); it is therefore zero by \textup{(ii)}.
	The preceding injectivity now gives \(x=0\).
\end{proof}

\begin{lemma}[the henselian discrete-valuation case]
\label{lem:henselian-dvr-generic-injectivity}
Let \(D\) be a henselian discrete valuation ring with fraction field
\(F\) and residue field \(k\). If \(N\ge2\) is invertible in \(D\), then
\[
K_n(D;\mathbf Z/N)\longrightarrow K_n(F;\mathbf Z/N)
\]
is injective for every \(n\in\mathbf Z\).
\end{lemma}

\begin{proof}
Let \(\pi\) be a uniformizer of \(D\), let
\(i\colon\Spec(k)\hookrightarrow\Spec(D)\), and let
\(j\colon\Spec(F)\hookrightarrow\Spec(D)\). The discrete-valuation
hypothesis enters at the support term. The ring \(D\) is regular and
noetherian, and \(k=D/\pi D\) has the two-term perfect resolution
\(0\to D\xrightarrow{\pi}D\to k\to0\). Localization gives \(K\)-theory
with support at the closed point, and dévissage identifies that support
term with \(K(k)\). Hence we have a fibre sequence
\[
K(k;\mathbf Z/N)\xrightarrow{i_*}K(D;\mathbf Z/N)
\xrightarrow{j^*}K(F;\mathbf Z/N)
\]
\cite[Thm.~7.4]{ThomasonTrobaugh90} and
\cite[\S4, Thm.~3 and \S5, Thm.~4]{Quillen73}.
%

According to \cref{lem:dvr-closed-point-self-intersection}, \(i^*i_*=0\) on every \(K_n(k;\mathbf Z/N)\). Since \(i^*\) is an
equivalence by
\Cref{lem:valuation-closed-point-finite-coefficient-k-rigidity},
\(i_*=0\) on every homotopy group. The nonconnective localization long
exact sequence therefore makes \(j^*\) injective for every
\(n\in\mathbf Z\).

\end{proof}

\begin{theorem}[lexicographic finite-rank Gersten injectivity]
\label{thm:lexicographic-finite-rank-generic-injectivity}
Let \(O\) be a henselian valuation ring with fraction field \(F\) and
value group \(\mathbf Z^d\) in lexicographic order, where \(d\ge0\).
Here the first coordinate is dominant.
If \(N\ge2\) is invertible in \(O\), then
\[
K_n(O;\mathbf Z/N)\longrightarrow K_n(F;\mathbf Z/N)
\]
is injective for every \(n\in\mathbf Z\).
\end{theorem}

\begin{proof}
We induct on \(d\). For \(d=0\), the ring \(O=F\) is a field. Suppose
\(d>0\), and let \(\mathfrak q\) be the smallest nonzero prime of
\(O\). By \Cref{lem:strategy-d-henselian-localisations}, both
\(O_{\mathfrak q}\) and \(O/\mathfrak q\) are henselian valuation
rings. By the order-reversing correspondence of
\cite[Lemma~2.3.1]{EnglerPrestel05ValuedFields}, the smallest nonzero
prime \(\mathfrak q\) corresponds to the largest proper convex subgroup
\(0\oplus\mathbf Z^{d-1}\) of the lexicographically ordered group
\(\mathbf Z^d\). The value-group calculation immediately following that
lemma identifies the value group of \(O_{\mathfrak q}\) with
\(\mathbf Z^d/(0\oplus\mathbf Z^{d-1})\simeq\mathbf Z\), and that of
\(O/\mathfrak q\) with
\(0\oplus\mathbf Z^{d-1}\simeq\mathbf Z^{d-1}\), with its lexicographic
order. Thus \(O_{\mathfrak q}\) is a discrete valuation ring. The
value-group calculation is essential: rank one alone does not imply
discreteness.

The generic restriction map for \(O_{\mathfrak q}\) is injective by
\Cref{lem:henselian-dvr-generic-injectivity}, and the generic restriction map for
\(O/\mathfrak q\) is injective by induction. Reduction from
\(O_{\mathfrak q}\) to its residue field
\(\kappa(\mathfrak q)\) is an equivalence on mod-\(N\) \(K\)-theory by
\Cref{lem:valuation-closed-point-finite-coefficient-k-rigidity}.
The three hypotheses of
\Cref{lem:prime-step-generic-injectivity-surjective-residue} therefore
hold, and that lemma completes the induction.
\end{proof}

The following theorem is not needed but provides an alternative route.
\begin{theorem}[formal finite-rank induction from the rank-one case]
	\label{thm:finite-rank-prime-to-residue-characteristic-generic-injectivity}
	Let \(O\) be a finite-rank henselian valuation ring with fraction
	field \(F\). Let \(\ell\) be a prime which is a unit in \(O\), let \(\nu\ge1\), and put
	\(N:=\ell^\nu\). Assume that for every rank-one henselian valuation
	ring \(V\) with fraction field \(K\) and \(N\in V^\times\), the map
	\[
	K_n(V;\mathbf Z/N)\longrightarrow K_n(K;\mathbf Z/N)
	\]
	is injective for every \(n\in\mathbf Z\). Then
	\(K(O;\mathbf Z/N)\longrightarrow K(F;\mathbf Z/N)\)
	is injective on every homotopy group.
\end{theorem}
\begin{proof}
	Besides the stated rank-one hypothesis, the induction uses the closed-point
	rigidity equivalence of
	\Cref{lem:valuation-closed-point-finite-coefficient-k-rigidity}, proved above.
	We argue by induction on the rank of \(O\), in the sense of
		\Cref{conv:valuation-rank}. In rank zero the valuation ring is the field \(F\), so the map is
		the identity. Suppose that \(O\) has positive finite rank. Since \(\Spec(O)\) is linearly
		ordered and finite, it has a smallest nonzero prime \(\mathfrak q\).
	By \Cref{lem:strategy-d-henselian-localisations}, \(O_{\mathfrak q}\) and
	\(O/\mathfrak q\) are henselian valuation rings. The prime ideals of
	\(O_{\mathfrak q}\) correspond to the
		prime ideals of \(O\) contained in \(\mathfrak q\), while the prime ideals of
		\(O/\mathfrak q\) correspond to those containing \(\mathfrak q\)
		\cite[\stackstag{00E3} and \stackstag{00E5}]{StacksProject}. Minimality of \(\mathfrak q\) therefore shows
		that \(O_{\mathfrak q}\) has exactly the primes \(0\) and
		\(\mathfrak qO_{\mathfrak q}\), hence rank one, whereas
		\(O/\mathfrak q\) has strictly smaller finite rank. Since \(\ell\in O^\times\),
		its images remain units in these rings and their fraction and residue fields.
		Moreover \(\Frac(O_{\mathfrak q})=\Frac(O)=F\), because
		\(O_{\mathfrak q}\) is a localization of \(O\) inside \(F\), and the residue field
		of \(O_{\mathfrak q}\) at its maximal ideal \(\mathfrak qO_{\mathfrak q}\) is
		\(\kappa(\mathfrak q)\).

	The rank-one hypothesis applies to \(O_{\mathfrak q}\to F\).
	The induction hypothesis applies to
	\(O/\mathfrak q\to\kappa(\mathfrak q)\). Finally, reduction
	from \(O_{\mathfrak q}\) to its residue field \(\kappa(\mathfrak q)\) is an
	equivalence on \(K(-;\mathbf Z/N)\) by
	\Cref{lem:valuation-closed-point-finite-coefficient-k-rigidity}. Thus the three
	hypotheses of \Cref{lem:prime-step-generic-injectivity-surjective-residue} hold, and that
	lemma completes the induction.
\end{proof}

\begin{lemma}[parameter chain valuation of a regular local ring]
	\label{lem:regular-local-parameter-chain-valuation}
	Let \(A\) be a noetherian regular local ring of dimension \(d\), with maximal ideal
	\(\mathfrak m_A\), residue field \(k\), and fraction field \(F\). Then there is a
	valuation ring \(V\subset F\), with maximal ideal \(\mathfrak m_V\), such that
	\begin{enumerate}[label=\textup{(\roman*)},leftmargin=*]
		\item \(A\subseteq V\) and \(\mathfrak m_V\cap A=\mathfrak m_A\);
		\item \(V\) has rank \(d\), with value group \(\mathbf Z^d\) ordered
		lexicographically, the first coordinate dominant;
		\item the induced map \(k=A/\mathfrak m_A\to V/\mathfrak m_V\) is an isomorphism.
	\end{enumerate}
\end{lemma}

\begin{proof}
	We construct the valuation by composing discrete valuations along the flag defined by a
	regular system of parameters (see also \cite[\S2.3]{EnglerPrestel05ValuedFields}). Induct on \(d\). For \(d=0\) the ring \(A=F=k\) is a field
	and the trivial valuation works. Let \(d\ge1\) and let
	\(x_1,\dots,x_d\) be a regular system of parameters. Then
	\(A_2:=A/x_1A\) is regular local of dimension \(d-1\), hence a domain, so \(x_1A\) is a
	prime of height one and \(V_1:=A_{x_1A}\) is a discrete valuation ring with fraction
	field \(F\) and residue field \(\kappa_1=\Frac(A_2)\). Write
	\(\lambda\colon V_1\to\kappa_1\) for the reduction; then \(\lambda(A)=A_2\).

	The ring \(A_2\) is regular local of dimension \(d-1\) with fraction field \(\kappa_1\),
	residue field \(k\), and regular system of parameters the images of \(x_2,\dots,x_d\).
	By induction there is a valuation ring \(V'\subset\kappa_1\) of rank \(d-1\) with
	\(A_2\subseteq V'\), \(\mathfrak m_{V'}\cap A_2=\mathfrak m_{A_2}\), and
	\(A_2/\mathfrak m_{A_2}\xrightarrow{\ \sim\ }V'/\mathfrak m_{V'}\).

	Put \(V:=\lambda^{-1}(V')\). To see directly that this is a valuation
	ring, choose a valuation \(v'\colon\kappa_1^\times\to\mathbf Z^{d-1}\)
	with valuation ring \(V'\). Since \(x_1\) is a uniformizer of \(V_1\),
	every \(y\in F^\times\) has a unique expression \(y=x_1^n u\), with
	\(n\in\mathbf Z\) and \(u\in V_1^\times\). Then
	\(y\mapsto\bigl(n,v'(\lambda(u))\bigr)\) is a valuation with values
	in the lexicographically ordered group
	\(\mathbf Z\oplus\mathbf Z^{d-1}\), the first coordinate dominant, and
	its valuation ring is \(V\). Its maximal ideal is
	\(\lambda^{-1}(\mathfrak m_{V'})\), and its residue field is
	\(V'/\mathfrak m_{V'}\). This also proves \textup{(ii)}.

	Since \(\lambda(A)=A_2\subseteq V'\), we have \(A\subseteq V\). For
	\(a\in A\), one has \(a\in\mathfrak m_V\) if and only if
	\(\lambda(a)\in\mathfrak m_{V'}\), equivalently if and only if
	\(a\in\mathfrak m_A\). Hence \(\mathfrak m_V\cap A=\mathfrak m_A\), and
	the induced residue map is the isomorphism
	\(A/\mathfrak m_A=A_2/\mathfrak m_{A_2}
	\xrightarrow{\sim}V'/\mathfrak m_{V'}=V/\mathfrak m_V\).
\end{proof}

\begin{theorem}[Gersten injectivity for regular henselian pairs]
\label{thm:regular-henselian-pair-lifting}
Let \(R\) be a noetherian regular domain of finite Krull dimension, let
\(\mathfrak p\subset R\) be prime, and assume that \(R/\mathfrak p\) is
regular and that \((R,\mathfrak p)\) is a henselian pair. Let
\(N=\ell^\nu\) be invertible in \(R\). If \(R/\mathfrak p\) satisfies
Gersten injectivity with \(\mathbf Z/N\)-coefficients, then so does \(R\).
\end{theorem}

\begin{proof}
Put \(F=\Frac(R)\), \(E=\Frac(R/\mathfrak p)\), and
\(d=\dim(R_{\mathfrak p})\). Apply
\Cref{lem:regular-local-parameter-chain-valuation} to
\(R_{\mathfrak p}\). It gives a valuation ring \(V\subset F\) which
dominates \(R_{\mathfrak p}\) and has residue field \(E\). Let
\(W=V^h\). The henselization \(W\) is a valuation ring by
\cite[\stackstag{0ASK}]{StacksProject}; it is henselian, local, and has the same residue
field \(E\) as \(V\) by \cite[\stackstag{04GN}]{StacksProject}. Since
\(V\subseteq W\), one has \(F=\Frac(V)\subseteq\Frac(W)\). Moreover,
\Cref{lem:regular-local-parameter-chain-valuation} and
\cite[\stackstag{0ASK}]{StacksProject} identify the value group of
\(W\) with the same lexicographically ordered \(\mathbf Z^d\) as that of
\(V\). Since \(V\) dominates \(R_{\mathfrak p}\) and henselization is
local with unchanged residue field,
\(\mathfrak m_W\cap R=\mathfrak p\). The image of
\(N\in R^\times\) remains a unit in \(W\), and the map \(R\to W\)
induces the canonical map \(R/\mathfrak p\to E\).

Rigidity gives a commutative square
\[
\begin{tikzcd}[column sep=large]
K(R;\mathbf Z/N) \ar[r] \ar[d,"\sim"'] &
K(W;\mathbf Z/N) \ar[d,"\sim"]\\
K(R/\mathfrak p;\mathbf Z/N) \ar[r] & K(E;\mathbf Z/N).
\end{tikzcd}
\]
The left vertical map follows from
\Cref{thm:gabber-rigidity-finite-coefficient-k}; both rings are regular,
so their nonconnective \(K\)-theory is connective. The right vertical
map is \Cref{lem:valuation-closed-point-finite-coefficient-k-rigidity}.

Let \(\alpha\in K_n(R;\mathbf Z/N)\) vanish in
\(K_n(F;\mathbf Z/N)\). Its image in \(K_n(W;\mathbf Z/N)\) vanishes in
\(K_n(\Frac(W);\mathbf Z/N)\), hence is zero by
\Cref{thm:lexicographic-finite-rank-generic-injectivity}. The square shows that
the reduction of \(\alpha\) vanishes in \(K_n(E;\mathbf Z/N)\). It is
zero in \(K_n(R/\mathfrak p;\mathbf Z/N)\) by the hypothesis. The left
vertical equivalence then gives \(\alpha=0\).
\end{proof}

\begin{corollary}[henselian regular local rings]
\label{cor:henselian-regular-local-generic-injectivity}
Let \(A\) be a noetherian henselian regular local ring, with residue field
\(k\) and fraction field \(F\). Let \(\ell\) be a prime invertible in
\(A\), and let \(\nu\ge1\). Then
\[
K_n(A;\mathbf Z/\ell^\nu)\longrightarrow
K_n(F;\mathbf Z/\ell^\nu)
\]
is injective for every \(n\in\mathbf Z\).
\end{corollary}

\begin{proof}
Taking \(\mathfrak p=\mathfrak m_A\) in
\Cref{thm:regular-henselian-pair-lifting} proves the result: \(A\) is a
noetherian regular domain of finite Krull dimension, while
\(A/\mathfrak m_A=k\) is regular and its generic restriction is the
identity. For completeness, we also give the short direct induction.

Put \(N=\ell^\nu\) and argue by induction on \(d=\dim(A)\). The case
\(d=0\) is tautological. Suppose \(d>0\), choose
\(x\in\mathfrak m_A\setminus\mathfrak m_A^2\), and put
\(\mathfrak p=(x)\). Then \(A/\mathfrak p\) is a henselian regular local
ring of dimension \(d-1\), while \(A_{\mathfrak p}\) is a discrete
valuation ring with fraction field \(F\) and residue field
\(
\kappa(\mathfrak p)=\Frac(A/\mathfrak p).
\)
Since \(A\) and \(A/\mathfrak p\) are regular, their nonconnective
\(K\)-theory is connective. Gabber's connective rigidity theorem
therefore applies to both henselian pairs with residue field \(k\), and
the commutative triangle through \(K(k;\mathbf Z/N)\) shows that
\[
K(A;\mathbf Z/N)\longrightarrow K(A/\mathfrak p;\mathbf Z/N)
\]
is an equivalence. By the induction hypothesis,
\(K_n(A/\mathfrak p;\mathbf Z/N)\to K_n(\kappa(\mathfrak p);\mathbf Z/N)\)
is injective. The commutative square induced by
\(A\to A_{\mathfrak p}\to\kappa(\mathfrak p)\) and
\(A\to A/\mathfrak p\to\kappa(\mathfrak p)\) therefore shows that
\[
K_n(A;\mathbf Z/N)\longrightarrow K_n(A_{\mathfrak p};\mathbf Z/N)
\]
is injective.

Since \(\ell\in A^\times\), one also has
\(\ell\in A_{\mathfrak p}^\times\). Gillet's finite-coefficient Gersten
theorem for an arbitrary discrete valuation ring therefore gives the
injectivity of
\(K_n(A_{\mathfrak p};\mathbf Z/N)\to K_n(F;\mathbf Z/N)\) for
\(n\ge0\) \cite{Gillet86TorsionGerstenDVR}. For \(n<0\), all groups in
question vanish by regularity. Composing the two maps proves the claim.
\end{proof}

\begin{remark}
	\label{rem:regular-pair-short-proof}
	The direct induction in the proof of
	\Cref{cor:henselian-regular-local-generic-injectivity}
	also gives a proof of
	\Cref{thm:regular-henselian-pair-lifting}
	without \Cref{thm:valuation-prime-step-k-excision}.
\end{remark}

\begin{corollary}[completion along a regular prime]
\label{cor:regular-prime-adic-completion}
Let \((A,\mathfrak m)\) be a noetherian regular local ring and let
\(\mathfrak p\subset A\) be a prime ideal such that \(A/\mathfrak p\) is
regular. Write
\[
\widehat A^{\,\mathfrak p}:=\varprojlim_n A/\mathfrak p^n.
\]
Let \(N=\ell^\nu\) be invertible in \(A\). If \(A/\mathfrak p\)
satisfies Gersten injectivity with \(\mathbf Z/N\)-coefficients, then so does
\(\widehat A^{\,\mathfrak p}\).
If \(A\) is excellent, then \(\widehat A^{\,\mathfrak p}\) is excellent. If
\(A/\mathfrak p\) is not henselian, then \(\widehat A^{\,\mathfrak p}\) is not henselian as a local
ring.
\end{corollary}

\begin{proof}
Since \(A\) and \(A/\mathfrak p\) are regular local, the prime
\(\mathfrak p\) is generated by part of a regular system of parameters
\cite[\stackstag{00NR}]{StacksProject}. The completion
\(\widehat A^{\,\mathfrak p}\) is noetherian by
\cite[\stackstag{0316}]{StacksProject} and flat over \(A\) by
\cite[\stackstag{00MB}]{StacksProject}. Moreover,
\begin{equation}\label{eq:regular-prime-completion-special-fibre}
\widehat A^{\,\mathfrak p}/\mathfrak p\widehat A^{\,\mathfrak p}
\simeq A/\mathfrak p
\end{equation}
by \cite[\stackstag{05GG}]{StacksProject}, and
\(\mathfrak p\widehat A^{\,\mathfrak p}\) lies in the Jacobson radical by
\cite[\stackstag{05GI}]{StacksProject}. Thus every maximal ideal of the
completion contains \(\mathfrak p\widehat A^{\,\mathfrak p}\), and the
maximal ideals of \(\widehat A^{\,\mathfrak p}\) correspond to those of
the quotient in \eqref{eq:regular-prime-completion-special-fibre}. The
latter is local, so \(\widehat A^{\,\mathfrak p}\) is local. Flatness preserves the chosen regular sequence;
since its quotient is regular local, \(\widehat A^{\,\mathfrak p}\) is regular local by
\cite[\stackstag{00NU}]{StacksProject}, hence a domain. 

The pair
\((\widehat A^{\,\mathfrak p},
\mathfrak p\widehat A^{\,\mathfrak p})\) is henselian because the ring
is complete for the
\(\mathfrak p\widehat A^{\,\mathfrak p}\)-adic topology
\cite[\stackstag{0ALJ}]{StacksProject}. Moreover,
\eqref{eq:regular-prime-completion-special-fibre} identifies its quotient
by \(\mathfrak p\widehat A^{\,\mathfrak p}\) with \(A/\mathfrak p\),
and the image of \(N\in A^\times\) is a unit in
\(\widehat A^{\,\mathfrak p}\). Thus all the hypotheses of
\Cref{thm:regular-henselian-pair-lifting} are satisfied, and that theorem
gives Gersten injectivity for \(\widehat A^{\,\mathfrak p}\).

If \(A\) is excellent, ideal-adic completion preserves excellence
\cite[\arxivhtmlref{1609.09246v2}{Thmmaintheorema2}{Main Theorem~2}]{KuranoShimomoto16IdealAdicCompletion}.
Finally, a
quotient of a henselian local ring is henselian: lift a monic polynomial
and a simple residue-field root to the original local ring, apply
Hensel's lemma there, and reduce the resulting root modulo the quotient
ideal. Consequently, if \(\widehat A^{\,\mathfrak p}\) were henselian,
then its quotient in \eqref{eq:regular-prime-completion-special-fibre},
which is isomorphic to \(A/\mathfrak p\), would be henselian, contrary
to the hypothesis.
\end{proof}

\begin{corollary}[excellent nonhenselian examples over a DVR]
\label{cor:excellent-nonhenselian-mixed-characteristic}
Let \(V\) be an excellent discrete valuation ring with uniformizer \(\pi\),
and let \((A,\mathfrak m)\) be a regular local ring essentially smooth over
\(V\), meaning a localization of a smooth \(V\)-algebra. Assume that
\(\pi\in\mathfrak m\), and write
\(\widehat A^{\,\pi}:=\varprojlim_n A/\pi^nA\). For every
\(N=\ell^\nu\) invertible in \(V\), the map
\[
K_n(\widehat A^{\,\pi};\mathbf Z/N)\longrightarrow
K_n(\Frac(\widehat A^{\,\pi});\mathbf Z/N)
\]
is injective for all \(n\in\mathbf Z\). The ring \(\widehat A^{\,\pi}\) is excellent. If
the regular local ring \(A/\pi A\) is nonhenselian, then \(\widehat A^{\,\pi}\) is a
nonhenselian local ring.
\end{corollary}

\begin{proof}
Smoothness makes \(A\) flat over \(V\), so multiplication by \(\pi\) on \(A\) is
injective. Base change to the residue field of \(V\) shows that \(A/\pi A\) is a regular
local ring essentially smooth over that field; in particular it is a
domain, so \(\pi A\) is prime. Moreover, \(A\) is excellent because it
is essentially of finite type over the excellent ring \(V\)
\cite[\stackstag{07QU}]{StacksProject}.

To apply \Cref{cor:regular-prime-adic-completion}, it remains to prove
Gersten injectivity for \(A/\pi A\). Since \(A/\pi A\) and
\(\Frac(A/\pi A)\) are regular, their negative \(K\)-groups vanish.
The coefficient exact sequence therefore gives
\[
K_n(A/\pi A;\mathbf Z/N)
=
K_n(\Frac(A/\pi A);\mathbf Z/N)
=0
\qquad(n<0).
\]
For the nonnegative degrees, realize \(A/\pi A\) as the local ring at a
point of a smooth integral variety over the residue field of \(V\).

For every \(r\ge0\), Grayson proves that the
augmented Gersten--Quillen resolution in degree \(r\) on this variety is
universally exact \cite[Cor.~6, p.~141]{Grayson85UniversalExactness}.
Thus, after taking the stalk at the chosen point, the complex remains
exact upon application of any additive
endofunctor of abelian groups that commutes with filtered colimits. This
applies to \(M\mapsto M/N\) and \(M\mapsto M[N]\). Exactness at the
augmentation gives
\[
K_r(A/\pi A)/N\hookrightarrow K_r(\Frac(A/\pi A))/N,
\qquad
K_r(A/\pi A)[N]\hookrightarrow K_r(\Frac(A/\pi A))[N].
\]
If \(n=0\), regularity gives
\(K_{-1}(A/\pi A)=K_{-1}(\Frac(A/\pi A))=0\). The coefficient exact
sequence therefore identifies the finite-coefficient generic restriction map in degree zero
with the injection on quotients for \(r=0\). 
Now let \(n\ge1\). Functoriality of the coefficient cofiber sequence gives
the commutative diagram
\[
\begin{tikzcd}[column sep=small]
	0 \ar[r] &
	K_n(A/\pi A)/N \ar[r]\ar[d,hook] &
	K_n(A/\pi A;\mathbf Z/N) \ar[r]\ar[d] &
	K_{n-1}(A/\pi A)[N] \ar[r]\ar[d,hook] &
	0\\
	0 \ar[r] &
	K_n(\Frac(A/\pi A))/N \ar[r] &
	K_n(\Frac(A/\pi A);\mathbf Z/N) \ar[r] &
	K_{n-1}(\Frac(A/\pi A))[N] \ar[r] &
	0 .
\end{tikzcd}
\]
The outer vertical maps are injective by the preceding paragraph;
hence the middle vertical map is injective.

Thus \(A/\pi A\) satisfies Gersten injectivity with
\(\mathbf Z/N\)-coefficients. Apply
\Cref{cor:regular-prime-adic-completion} with \(\mathfrak p=\pi A\).
\end{proof}

In mixed characteristic, the generic injectivity assertion of
\Cref{cor:excellent-nonhenselian-mixed-characteristic} also follows
from \cite[Cor.~6 and Rem.~\textup{(ii)} following it, p.~68]{GilletLevine87RelativeGersten},
\cite{Gillet86TorsionGerstenDVR}, and Popescu approximation
\cite[\stackstag{07GC}]{StacksProject}.
Indeed, \(V\to\widehat A^{\,\pi}\) is flat, its closed fibre
\(A/\pi A\) is geometrically regular over the residue field of \(V\),
and its generic fibre is regular over the characteristic-zero
field \(\Frac(V)\), hence geometrically regular.
Thus this map is regular. The proof above gives a direct application
of \Cref{thm:regular-henselian-pair-lifting}.

\section{An application to Prüfer bases}
\label{sec:prufer-ind-smooth-generic-injectivity}

We now use Gersten injectivity for henselian valuation rings to treat
henselian local ind-smooth algebras over Prüfer rings. The geometric
reduction below is separate from the filtered construction, but its
valuation-theoretic input,
\Cref{thm:henselian-valuation-generic-injectivity}, ultimately comes from
the filtered splitting.

The result is a mixed-characteristic finite-coefficient complement to
Kundu's integral equicharacteristic theorem
\cite[Thm.~1.2]{Kundu26SmoothValuation}. Bouis--Kundu package their abstract Gersten argument in
\cite[\arxivhtmlref{2506.09910v2}{S3.Thmtheorem5}{Thm.~3.5} and
\arxivhtmlref{2506.09910v2}{S3.Thmtheorem13}{Cor.~3.13}]{BouisKunduBL}.
It could give an alternative proof here after passing from the raw
finite-coefficient \(K\)-groups to their Nisnevich sheafifications and
checking the finitarity and Horrocks hypotheses. We use instead the
explicit residue-preserving valuation constructed below. Their
Beilinson--Lichtenbaum theorem remains an input through
\Cref{lem:ordinary-prufer-bl-range-input}; its proof uses their abstract
propagation theorem independently of the Gersten injectivity proved here.

\begin{lemma}[pointed smooth stages over a Prüfer base]
\label{lem:prufer-pointed-smooth-stage}
Let \(P\) be a Prüfer ring and let \((A,\mathfrak m_A,k)\) be an
essentially smooth local \(P\)-algebra. Then \(A\) and \(A^h\) are
normal domains. Write \(P=\prod_{\lambda=1}^r P_\lambda\), with each
\(P_\lambda\) a Prüfer domain. There is a unique index \(\lambda_0\)
through which the map \(P\to A\) factors. Put
\(P_0:=P_{\lambda_0}\) and
\(\mathfrak p:=\mathfrak m_A\cap P_0\). Then
\((P_0)_{\mathfrak p}\) is a valuation domain and \(A^h\) is
ind-smooth over it.
\end{lemma}

\begin{proof}
Let \(e_\lambda\in P\) be the idempotent of the \(\lambda\)-th factor.
Since a local ring has no idempotents other than \(0\) and \(1\),
exactly one \(e_\lambda\) maps to \(1\) in \(A\); all the others map to
\(0\). This proves the asserted unique factorization through \(P_0\).
Every Prüfer domain is
normal, and \((P_0)_{\mathfrak p}\) is a valuation domain. After
localizing the base at \(\mathfrak p\), the ring \(A\) is a localization
of a smooth \((P_0)_{\mathfrak p}\)-algebra. Smooth algebras over normal
rings are normal \cite[\stackstag{033C}]{StacksProject}; hence the local
ring \(A\) is a normal domain. Its henselization \(A^h\) is a normal
domain by
\cite[\stackstag{06DI}]{StacksProject}.

Write \(A=\varinjlim_i A_i\), where each \(A_i\) is a principal open
localization of a smooth \((P_0)_{\mathfrak p}\)-algebra and hence is
smooth over \((P_0)_{\mathfrak p}\). We verify ind-smoothness of
\(A^h\) by the finite-presentation factorization criterion
\cite[\stackstag{07C3}]{StacksProject}. Let \(E\) be a finitely
presented \((P_0)_{\mathfrak p}\)-algebra and let \(E\to A^h\) be a
homomorphism. By \cite[\stackstag{04GN}]{StacksProject}, the
henselization \(A^h\) is the filtered colimit of the pointed étale
neighbourhoods of \(A\); hence \(E\to A^h\) factors through one of
them, say through an étale \(A\)-algebra \(B\) of finite
presentation. After enlarging \(i\), finite-presentation descent gives
an étale \(A_i\)-algebra \(B_i\) of finite presentation with
\(B\simeq A\otimes_{A_i}B_i\); see
\cite[\stackstag{00U2}\textup{(9)}]{StacksProject}. Since
\(B\simeq\varinjlim_{j\ge i}(A_j\otimes_{A_i}B_i)\) and \(E\) is
finitely presented, the map \(E\to B\) factors through
\(A_j\otimes_{A_i}B_i\) for some \(j\ge i\). This algebra is étale
over \(A_j\), hence smooth over \((P_0)_{\mathfrak p}\). Thus every
map from a finitely presented \((P_0)_{\mathfrak p}\)-algebra to
\(A^h\) factors through a smooth \((P_0)_{\mathfrak p}\)-algebra,
and \cite[\stackstag{07C3}]{StacksProject} shows that \(A^h\) is
ind-smooth over \((P_0)_{\mathfrak p}\).
\end{proof}

\begin{lemma}[a residue-preserving valuation at a pointed smooth stage]
\label{lem:prufer-pointed-smooth-residue-valuation}
In the situation of
\Cref{lem:prufer-pointed-smooth-stage}, there are a henselian valuation
domain \(W\) and an injective local map
\[
A^h\longrightarrow W
\]
which induces an isomorphism from \(k\) to the residue field of \(W\).
Consequently, \(\Frac(A^h)\) embeds into \(\Frac(W)\).
\end{lemma}

\begin{proof}
By \Cref{lem:prufer-pointed-smooth-stage}, the local ring \(A\) is a
localization of a smooth \((P_0)_{\mathfrak p}\)-algebra. Since \(A\)
is local, we may choose a smooth \((P_0)_{\mathfrak p}\)-algebra \(C\)
and a prime \(\mathfrak q\subset C\) such that
\(A\simeq C_{\mathfrak q}\); indeed, starting from any localization
presentation, take \(\mathfrak q\) to be the inverse image of
\(\mathfrak m_A\). The contraction of \(\mathfrak q\) to
\((P_0)_{\mathfrak p}\) is the maximal ideal
\(\mathfrak p(P_0)_{\mathfrak p}\). The closed fibre
\(A/\mathfrak pA\) is therefore an essentially smooth local algebra
over the residue field of \((P_0)_{\mathfrak p}\), hence a noetherian
regular local domain with residue field \(k\).

The ring \(C\) is normal by
\cite[\stackstag{033C}]{StacksProject}. The valuation domain
\((P_0)_{\mathfrak p}\) is a domain, so its spectrum is irreducible,
and every nonempty open subspace of an irreducible topological space is
irreducible. Thus every quasi-compact open of
\(\Spec((P_0)_{\mathfrak p})\) has at most one irreducible component.
Applying \cite[\stackstag{07TD}]{StacksProject} to the smooth morphism
\(\Spec(C)\to\Spec((P_0)_{\mathfrak p})\) shows that
\(\Spec(C)\) has finitely many irreducible components. Equivalently,
\(C\) has finitely many minimal primes. It now follows from
\cite[\stackstag{030C}]{StacksProject} that \(C\) is a finite product
of normal domains. The prime \(\mathfrak q\) belongs to exactly one
of the corresponding open-and-closed factors. Replacing \(C\) by
that factor preserves smoothness over \((P_0)_{\mathfrak p}\) and does
not change \(C_{\mathfrak q}\). We may therefore assume that \(C\)
is an integral smooth \((P_0)_{\mathfrak p}\)-algebra.

Since \(A/\mathfrak pA\) is a domain, \(\mathfrak pA\) is prime. Let
\(\mathfrak q_0\subseteq\mathfrak q\) be its contraction under
\(C\to C_{\mathfrak q}=A\). Then
\(\mathfrak q_0C_{\mathfrak q}=\mathfrak pA\), and
\(
A/\mathfrak pA
\simeq (C/\mathfrak q_0)_{\mathfrak q/\mathfrak q_0}.
\)
The special fibre \(C/\mathfrak pC\) is smooth over the residue field
of \((P_0)_{\mathfrak p}\), hence noetherian and reduced. Its
localization at \(\mathfrak q/\mathfrak pC\) is
\(A/\mathfrak pA\), which is a domain. Minimal primes under
localization correspond to minimal primes contained in
\(\mathfrak q/\mathfrak pC\); hence exactly one minimal prime of
\(C/\mathfrak pC\) is contained in \(\mathfrak q/\mathfrak pC\).
The contraction of the zero ideal of the domain
\(A/\mathfrak pA\) is both this minimal prime and, by the definition
of \(\mathfrak q_0\), the prime \(\mathfrak q_0/\mathfrak pC\). Thus
\(\mathfrak q_0\) is the
generic point of the component of the special fibre selected by
\(\mathfrak q\), and
\[
\kappa(\mathfrak q_0)
\simeq \Frac(C/\mathfrak q_0)
\simeq \Frac(A/\mathfrak pA).
\]
Since \(\mathfrak q_0\) is a generic point of the special fibre,
Kundu's valuation lemma applies to the integral smooth
\((P_0)_{\mathfrak p}\)-algebra \(C\). For the singleton
\(\{\mathfrak q_0\}\),
\cite[Lem.~2.13(1) and Rem.~2.14(1)]{Kundu26SmoothValuation}\footnote{We use only
	part~\textup{(1)} of the cited lemma, which has no finite-rank
	hypothesis; the finite-rank assumption there enters only in the
	approximation statement of part~\textup{(2)}.} shows that
\(C_{\mathfrak q_0}\) is a faithfully flat valuation domain over
\((P_0)_{\mathfrak p}\), with the same value group. Its fraction field
is \(\Frac(C)=\Frac(A)\), and its residue field is
\(\Frac(A/\mathfrak pA)\). Moreover,
\(C_{\mathfrak q_0}\simeq A_{\mathfrak pA}\): indeed, localizing
\(C_{\mathfrak q}=A\) at
\(\mathfrak q_0C_{\mathfrak q}=\mathfrak pA\) gives both rings.
Consequently, \(A_{\mathfrak pA}\)
is a valuation ring of \(\Frac(A)\), with residue field
\(\Frac(A/\mathfrak pA)\).

Apply \Cref{lem:regular-local-parameter-chain-valuation} to
\(A/\mathfrak pA\). We obtain a valuation ring
\(\overline U\subseteq\Frac(A/\mathfrak pA)\) which dominates
\(A/\mathfrak pA\) and has residue field \(k\). Let \(U\) be the
inverse image of \(\overline U\) under the residue map
\(C_{\mathfrak q_0}\to\Frac(A/\mathfrak pA)\).

By composition of valuations, \(U\) is a valuation ring of
\(\Frac(A)\). The quotient map \(A\to A/\mathfrak pA\) is the restriction
of the residue map on \(C_{\mathfrak q_0}\), and
\(\overline U\) contains and dominates \(A/\mathfrak pA\). Hence
\(A\subseteq U\), the maximal ideal of \(U\) contracts to
\(\mathfrak m_A\), and the residue field of \(U\) is \(k\). Thus \(U\)
dominates \(A\).

Put \(W:=U^h\). The ring \(W\) is a henselian valuation ring by
\cite[\stackstag{0ASK}]{StacksProject}, and its residue field is still
\(k\) by \cite[\stackstag{04GN}]{StacksProject}. The local map
\(A\to W\) extends, by the universal property of henselization, to a
local map \(A^h\to W\).

It remains to verify injectivity. Since \(C\) is a normal domain and
\(A=C_{\mathfrak q}\), the ring \(A\) is a normal local domain. Hence
\(A^h\) is a normal domain by
\cite[\stackstag{06DI}]{StacksProject}.

The restriction of \(A^h\to W\) to \(A\) is the inclusion
\(A\subseteq U\subseteq W\); therefore localization gives a unital map
\[
(A\setminus\{0\})^{-1}A^h\longrightarrow\Frac(W).
\]
Since \(A^h\) is a domain,
\((A\setminus\{0\})^{-1}A^h\) is a subring of \(\Frac(A^h)\).
By base change for ind-\'etale algebras
\cite[\stackstag{0BSH}]{StacksProject}, it is ind-\'etale over
\(\Frac(A)\), hence algebraic over \(\Frac(A)\). Being an algebraic
domain over a field, it is a field. It contains \(A^h\) and is contained
in \(\Frac(A^h)\); consequently
\[
(A\setminus\{0\})^{-1}A^h=\Frac(A^h).
\]
Thus the preceding unital map is a homomorphism
\(\Frac(A^h)\to\Frac(W)\) between fields and is therefore injective.
Its restriction \(A^h\to W\) is injective as well.

\end{proof}

\begin{lemma}[nonconnective rigidity at a pointed smooth stage]
\label{lem:prufer-pointed-smooth-nonconnective-rigidity}
In the situation of
\Cref{lem:prufer-pointed-smooth-stage}, let \(N\ge1\) be invertible in
\(A\). Reduction induces an equivalence of nonconnective spectra
\[
K(A^h;\mathbf Z/N)\xrightarrow{\ \simeq\ }K(k;\mathbf Z/N).
\]
\end{lemma}

\begin{proof}
By \Cref{lem:prufer-pointed-smooth-stage}, write
\(A^h=\varinjlim_\lambda B_\lambda\), with every \(B_\lambda\) smooth
over the valuation domain \((P_0)_{\mathfrak p}\).
By \cite[\arxivhtmlref{2203.06472v1}{S2.Thmtheorem3}{Cor.~2.3}]{AntieauMathewMorrow22PerfectoidKTheory}, each
\(B_\lambda\) is weakly regular and stably coherent. Therefore
\cite[\arxivhtmlref{2203.06472v1}{S2.I1.i1}{Prop.~2.4\textup{(1)}}]{AntieauMathewMorrow22PerfectoidKTheory}
applies and gives \(K_i(B_\lambda)=0\) for \(i<0\).
Nonconnective \(K\)-theory commutes with filtered colimits \cite[Def.~8.1 and proof of Thm.~9.8]
{BGT13KTheoryStableCategories}; at the ring level this
uses that perfect complexes and their morphisms descend along filtered
ring colimits. Hence \(K_i(A^h)=0\) for
every \(i<0\). Thus \(K(A^h)\) is connective, and the canonical
comparison from connective to nonconnective \(K\)-theory is an
equivalence for \(A^h\). The same holds for the field \(k\). Therefore
the image of \(N^{-1}\in A\) shows that \(N\in(A^h)^\times\), and
\Cref{thm:gabber-rigidity-finite-coefficient-k}, applied to the
henselian pair \((A^h,\mathfrak m_{A^h})\), gives the claimed
equivalence of nonconnective finite-coefficient spectra.
\end{proof}

\begin{proposition}[Gersten injectivity at a pointed smooth stage]
\label{prop:prufer-pointed-smooth-generic-injectivity}
In the situation of
\Cref{lem:prufer-pointed-smooth-stage}, let
\(N=\ell^\nu\) be invertible in \(A\). Then
\[
K_n(A^h;\mathbf Z/N)\longrightarrow
K_n(\Frac(A^h);\mathbf Z/N)
\]
is injective for every \(n\in\mathbf Z\).
\end{proposition}

\begin{proof}
Choose \(A^h\hookrightarrow W\) as in
\Cref{lem:prufer-pointed-smooth-residue-valuation}. The square
\[
\begin{tikzcd}[column sep=large]
K(A^h;\mathbf Z/N) \ar[r] \ar[d,"\sim"'] &
K(W;\mathbf Z/N) \ar[d,"\sim"] \\
K(k;\mathbf Z/N) \ar[r,equal] & K(k;\mathbf Z/N)
\end{tikzcd}
\]
has vertical equivalences by
\Cref{lem:prufer-pointed-smooth-nonconnective-rigidity,lem:valuation-closed-point-finite-coefficient-k-rigidity}.
Here \(N\in A^\times\) remains a unit in \(A^h\) and \(W\), and the
bottom map is the identity under the residue-field identification in
\Cref{lem:prufer-pointed-smooth-residue-valuation}. Hence the top map
is an equivalence. If a class in the source of the
proposition vanishes over \(\Frac(A^h)\), its image in
\(K_n(W;\mathbf Z/N)\)
vanishes over \(\Frac(W)\). It is zero by
\Cref{thm:henselian-valuation-generic-injectivity}, and therefore the
original class is zero.
\end{proof}

\begin{theorem}[Gersten injectivity over a Prüfer base]
\label{thm:prufer-ind-smooth-generic-injectivity}
Let \(P\) be a Prüfer ring, let \(R\) be a henselian local ind-smooth
\(P\)-algebra, and let \(N=\ell^\nu\) be invertible in \(R\). Then
\(R\) is a domain, and
\[
K_n(R;\mathbf Z/N)\longrightarrow
K_n(\Frac(R);\mathbf Z/N)
\]
is injective for every \(n\in\mathbf Z\).
\end{theorem}

\begin{proof}
After a cofinal reindexing, choose a filtered poset \(I\) and smooth
\(P\)-algebras \(C_i\), \(i\in I\), such that
\(R\simeq\varinjlim_{i\in I}C_i\). Let \(\mathfrak m\) be the maximal
ideal of \(R\).

We first prove that \(R\) is a domain. Let \(x,y\in R\) satisfy
\(xy=0\). After increasing an index, there are \(i\in I\) and
\(x_i,y_i\in C_i\) mapping to \(x,y\) such that \(x_iy_i=0\). Put
\(\mathfrak q_i:=\mathfrak m\cap C_i\). The local ring
\((C_i)_{\mathfrak q_i}\) is a domain by
\Cref{lem:prufer-pointed-smooth-stage}; hence, say, \(x_i/1=0\).
There is \(s\in C_i\setminus\mathfrak q_i\) such that \(sx_i=0\).
The image of \(s\) in \(R\) is a unit, so \(x=0\). Thus \(R\) is a
domain.

Let \(\alpha\in K_n(R;\mathbf Z/N)\) have zero image in
\(K_n(\Frac(R);\mathbf Z/N)\). 

Since
\(\Frac(R)=\varinjlim_{0\ne f\in R}R[1/f]\), and finite-coefficient
nonconnective \(K\)-theory commutes with filtered colimits \cite[Def.~8.1 and proof of Thm.~9.8]
{BGT13KTheoryStableCategories},
there is a nonzero \(f\in R\) such that \(\alpha\) vanishes in
\(K_n(R[1/f];\mathbf Z/N)\).

After increasing an index, choose \(i\in I\),
\(\alpha_i\in K_n(C_i;\mathbf Z/N)\), and \(f_i\in C_i\) mapping to
\(\alpha\) and \(f\), respectively. For \(j\ge i\), denote their
images in \(C_j\) by \(\alpha_j\) and \(f_j\). Since
\(R[1/f]\simeq\varinjlim_{j\ge i}C_j[1/f_j]\), there is \(j\ge i\)
such that \(\alpha_j\) vanishes in
\(K_n(C_j[1/f_j];\mathbf Z/N)\). Put
\(C:=C_j\), \(\alpha_C:=\alpha_j\), \(f_C:=f_j\), and
\(\mathfrak q:=\mathfrak m\cap C\).

The integer \(N\) is a unit in \(C_{\mathfrak q}\), because its image
is a unit in \(R\). Hence
\Cref{prop:prufer-pointed-smooth-generic-injectivity} applies to
\((C_{\mathfrak q})^h\). Since \(R\) is henselian, the local map
\(C_{\mathfrak q}\to R\) extends to
\((C_{\mathfrak q})^h\to R\). The image of \(f_C\) in
\((C_{\mathfrak q})^h\) is nonzero, since its image in \(R\) is
\(f\ne0\). The image of \(\alpha_C\) vanishes after inverting \(f_C\),
hence in \(K_n(\Frac((C_{\mathfrak q})^h);\mathbf Z/N)\). The cited
proposition shows that it already vanishes in
\(K_n((C_{\mathfrak q})^h;\mathbf Z/N)\). Its image in
\(K_n(R;\mathbf Z/N)\) is \(\alpha\), so \(\alpha=0\).
\end{proof}

\subsection*{Acknowledgments}
The author is grateful to Fr\'ed\'eric D\'eglise for introducing him to the Gersten
conjecture in January 2018 and for many subsequent discussions around this problem.
The author thanks Christian Dahlhausen for comments on an earlier version of this paper.

\medskip
\noindent\textbf{Use of AI-assisted tools}.
Large language models were used as interactive tools for expanding proof
sketches, checking consistency, and improving exposition. All resulting
suggestions were checked by the author, who assumes full responsibility
for the mathematical statements, proofs, and references.

\end{document}